\documentclass[reqno,english,11pt]{amsart}
\usepackage[margin=1in]{geometry}
\usepackage[latin9]{inputenc}
\usepackage{mathrsfs}
\usepackage{bm}
\usepackage{amsbsy}
\usepackage{amstext}
\usepackage{amsthm}
\usepackage{amssymb}
\usepackage{hyperref}
\usepackage{xcolor}

\usepackage{stmaryrd}
\makeatletter

\newcommand{\lyxmathsym}[1]{\ifmmode\begingroup\def\b@ld{bold}
	\text{\ifx\math@version\b@ld\bfseries\fi#1}\endgroup\else#1\fi}

\numberwithin{equation}{section}
\numberwithin{figure}{section}
\theoremstyle{plain}
\newtheorem{thm}{\protect\theoremname}[section]
\theoremstyle{plain}
\newtheorem{cor}[thm]{\protect\corollaryname}
\theoremstyle{definition}
\newtheorem{defn}[thm]{\protect\definitionname}
\theoremstyle{plain}
\newtheorem{prop}[thm]{\protect\propositionname}
\theoremstyle{remark}
\newtheorem{rem}[thm]{\protect\remarkname}
\theoremstyle{plain}
\newtheorem{lem}[thm]{\protect\lemmaname}

\makeatother

\usepackage{babel}
\providecommand{\definitionname}{Definition}
\providecommand{\lemmaname}{Lemma}
\providecommand{\propositionname}{Proposition}
\providecommand{\corollaryname}{Corollary}
\providecommand{\remarkname}{Remark}
\providecommand{\theoremname}{Theorem}

\def\bR {\mathbb{R}}

\def\cD {\mathcal{D}}

\def\cY {\mathcal{Y}}

\def\grad {{\nabla}}

\def\la {\langle}
\def\ra {\rangle}

\newcommand{\bs}[1]{\boldsymbol{#1}}

\renewcommand{\ker}{\operatorname{ker}}

\newcommand{\eee}{\mathrm e}

\begin{document}
	\title[Multi-solitons for Klein-Gordon equations]{Asymptotic stability and classification of multi-solitons for Klein-Gordon
		equations}
	\author{Gong Chen}
	
	\author{Jacek Jendrej}
	\date{\today}
	\begin{abstract}
		Focusing on multi-solitons for the Klein-Gordon equations, in first part of this paper, 
		we establish their conditional asymptotic stability.  In the second part of this paper,  we classify  pure
		multi-solitons which are solutions converging to  multi-solitons in the energy space as $t\rightarrow\infty$. Using Strichartz estimates developed in our earlier work \cite{CJ2} and the modulation
		techniques, we show that if a solution stays close to the multi-soliton family, then it scatters to the multi-soliton family in the sense that the solution will converge in large time
		to a superposition of Lorentz-transformed solitons (with slightly modified velocities),
		and a radiation term which is at main order a free wave.  Moreover, we construct a finite-codimension centre-stable manifold around the well-separated multi-soliton family. 	Finally, given different Lorentz parameters
		and arbitrary centers, we show that all the  corresponding pure multi-solitons 
		form a finite-dimension manifold.
	\end{abstract}
	\address[Chen]{School of Mathematics, Georgia Institute of Technology, Atlanta, GA 30332-0160, USA }
	\email{gc@math.gatech.edu}
	\address[Jendrej]{CNRS \& LAGA, Universit\'e Sorbonne Paris Nord, UMR 7539, 99 av J.-B.~Cl\'ement, 93430 Villetaneuse, France }
	\email{jendrej@math.univ-paris13.fr}
	\date{\today}
	
	\maketitle
	\setcounter{tocdepth}{1}
	
	\tableofcontents
	
	\section{Introduction}
	Consider the nonlinear Klein-Gordon equation 
	\begin{equation}
		\partial_{t}^{2}\psi-\Delta \psi+\psi-\psi^{p}=0,\,\left(t,x\right)\in\mathbb{R}^{1+d}.\label{eq:nkg}
	\end{equation}
	Eliminating the time dependence, one can find stationary solutions  to the equation above which solve
	\begin{equation}
		-\Delta Q+Q-Q^{p}=0\label{eq:nQ}
	\end{equation}
	and decay exponentially.  In particular, there exists a unique radial positive ground
	state with the least energy
	\[
	E(Q):=\int_{\mathbb{R}^d}\frac{|\nabla Q|^2 +Q^2}{2}-\frac{Q^{p+1}}{p+1}\,dx
	\] among all non-zero solutions to the elliptic problem \eqref{eq:nQ}. We refer to Nakanishi-Schlag \cite{NSch,NSch1,NSch3} for more details.

	%	\smallskip
	
	The nonlinear Klein-Gordon equation \eqref{eq:nkg} is a wave type equation, so one indispensable
	tool to study it is the Lorentz boost. Let $\beta\in\mathbb{R}^{d}$, $|\beta|<1$, be a velocity vector.
	For a function $\phi:\mathbb{R}^{d}\to\mathbb{R}^{d}$, the Lorentz
	boost of $\phi$ with respect to $\beta$ is given by
	\begin{equation}\label{eq:lorentz}
		\phi_{\beta}(x):=\phi(\Lambda_{\beta}x),\quad\Lambda_{\beta}x:=x+(\gamma-1)\frac{(\beta\cdot x)\beta}{|\beta|^{2}},\quad\gamma:=\frac{1}{\sqrt{1-|\beta|^{2}}}.
	\end{equation}
	With these notations, the Lorentz transformation is given by
	\[
	(t',x')=\big(\gamma(t-\beta\cdot x),\ \Lambda_{\beta}x-\gamma\beta t\big)=\big(\gamma(t-\beta\cdot x),\ \Lambda_{\beta}(x-\beta t)\big).
	\]
	It is crucial that for each $\beta\in\mathbb{R}^{d},\,\left|\beta\right|\in[0,1)$, if
	$u$ is a solution of \eqref{eq:nkg} then $u_{\beta}\left(x-\beta t\right)$
	is also a solution.
	
	%Solutions to the Klein-Gordon equation \eqref{eq:nkg} satisfy several
	%symmetries:
	%\begin{itemize}
	%\item Shifts in space and time: if $u\left(t,x\right)$ is a solution then
	%$u\left(t+t_{0},x+x_{0}\right)$ will also be a solution. 
	%\item Lorentz boosts:
	%\end{itemize}

	%\smallskip
	
	Applying Lorentz transforms and the translational symmetry, we
	can obtain a family of traveling waves $Q_{\beta}\left(x-\beta t+x_{0}\right)$
	to \eqref{eq:nkg} from the stationary solution $Q$.  Using these traveling waves as building blocks, for $\beta_j\neq\beta_k$ and arbitrary $x_j$'s, one can study 
	the
	multi-soliton which refers to a superposition of a finite number
	of Lorentz-transformed solitons, moving with distinct speeds:
	\begin{equation}
		R\left(t,x\right)=\sum_{j=1}^{N}\sigma_j Q_{\beta_{j}}\left(x-\beta_{j}t-y_j\right),\,\sigma_j\in\{\pm 1\}\label{eq:multiintro}.
	\end{equation}
	Notice that a multi-soliton above is not a solution to \eqref{eq:nkg} but	
	one can construct a
	pure multi-soliton to \eqref{eq:nkg} in the following sense:
	\begin{equation}
		\psi\rightarrow	R\left(t,x\right)\,\,\text{as}\,\,t\rightarrow\infty,\label{eq:asyMutiintro}
	\end{equation}
	see C\^ote-Mu\~noz \cite{CMu}. Similar pure multi-solitons can be constructed using other type of stationary solutions to \eqref{eq:nQ}, see Bellazzini-Ghimenti-Le Coz \cite{BGL} and C\^ote-Martel \cite{CMart}.
	To understand the long-time dynamics and soliton resolution of \eqref{eq:nkg},
	two crucial problems are to study  dynamics around multi-solitons
	and to give a classification of pure multi-solitons, i.e.,
	solutions satisfying \eqref{eq:asyMutiintro}.

	In this paper, using Strichartz estimates and exponential dichotomy in our earlier works, \cite{CJ,CJ2},  we  are going to analyze  asymptotic stability of multi-solitons \eqref{eq:multiintro}, construct centre-stable manifolds around multi-solitons and classify  pure multi-solitons \eqref{eq:asyMutiintro}.   For concreteness, we will focus on $d=3$ and $p=3$ in \eqref{eq:nkg}.  In particular, the same analysis is expected to work for $d\geq 3$ and for equations with mass supercritical and energy subcritical powers provided that the linearized operator around the ground state has no gap eigenvalue nor threshold resonance.

	%Here, a multi-soliton is defined as a superposition of a finite number
	%of Lorentz-transformed solitons, moving with distinct speeds.
	%Under an appropriate assumption on non-crossing of the trajectories,
	%the problem of asymptotic stability is to prove that a small perturbation
	%of initial data of this form leads to a solution converging in large time
	%to a superposition of Lorentz-transformed solitons (with slightly modified velocities),
	%and a radiation term which is at main order a free wave.
	
	%A pure multi-soliton is a special multi-soliton, for which the energy of the radiation term
	%converges to zero as $t \to \infty$.
	%The problem of uniqueness of multi-solitons is to find all such solutions.
	%Under the assumption of distinct velocities, they should form a finite-dimension manifold,
	%parametrized by the affine asymptotes of the trajectories of the solitons.

	\subsection{Related literature}
	
	The study of the (conditional) stability problem of solitons in nonlinear disperisve PDEs has a long history. For example,  Weinstein \cite{Wein,Wein2}, Pego-Weinstein \cite{PW} etc established orbital stability: starting with a soliton plus a small perturbation, the solution remains in this form for all time. Their work introduced the method of modulation and constituted the foundation of every subsequent attempt. We will be interested in  the asymptotic stability problem which is a stronger property-the situation in which small perturbations not only remain small, but in fact disperse. A first such result
	was established by Soffer-Weinstein \cite{SW2,SW3}. Without trying to be exhaustive, we refer to Beceanu \cite{Bec2}, Cuccagna \cite{Cu1}, Krieger-Schalg \cite{KS,KS2}, Nakanishi-Schlag \cite{NSch,NSch1,NSch2,NSch3}, Perelman \cite{Perl}, Rodnianski-Schlag-Soffer \cite{RSS2}, Schlag\cite{Sch1} etc. We also refer Tao's survey \cite{Tao2} for references. Overall, the stability problem near one single soliton has been studied extensively. For the stability problem around a multi-soliton, Perelman \cite{Perl}, Rodnianski-Schlag-Soffer \cite{RSS2} use pointwise decay estimates to obtain the asymptotic stability in $L^2\bigcap L^1$. Using appropriate Strichartz estimates, one can go beyond these strong topologies and analyze the stability problem in some natural topology, see for example Beceanu \cite{Bec2}, Nakanishi-Schlag \cite{NSch,NSch1,NSch2,NSch3} for one-soliton problem. On the other hand, for the energy and modulation methods to study asymptotic stability of KdV type problems, see Merle-Martel \cite{MM1}, Martel-Merle-Tsai \cite{MMT} in the soliton region. These papers use specific monotonicity formulae of the KdV type equations
	and do not apply to wave equations.  In this paper, with exponential dichotomies and Strichartz estimates for Klein-Gordon equations with several moving potentials in our earlier work \cite{CJ,CJ2},  together with modulation methods,  we will show the conditional asymptotic stability of the multi-solitons consisted of  Lorentz-transformed solitons $Q$, see \eqref{eq:multiintro}, moving with distinct speeds. 
	we will show if a solution stays close the the multi-soliton family in the energy space, then it will converge in large time
	to a superposition of Lorentz-transformed solitons (with slightly modified velocities),
	and a radiation term which is at main order a free wave.  Using modulation techqniues and Strichartz estimates with their refined applications, we also construct the centre-stable manifold around the well-separated multi-soliton family. We further show that all solutions stay close the multi-soltion family satisfying the separation condition if and only if they are in this centre-stable manifold.

	In the second part of the current paper, we study the classification of pure multi-soliton. A pure multi-soliton is a special perturbation of the multi-soliton, for which the energy of the radiation term
	converges to zero as $t \to \infty$. The problem of classification of pure multi-solitons solution is to find all such solutions. To construct pure multi-solitons in various different models has a long history, we refer to  Bellazzini-Ghimenti-Le Coz \cite{BGL}, C\^ote-Mu\~noz \cite{CMu}, C\^ote-Martel \cite{CMart}, Martel \cite{Mart} and references therein for more details. After establishing the existence of pure multi-solitons, it is natural to ask if one can classify all pure multi-solitons and show their uniqueness in appropriate sense.
	%The typical argument is to construct a sequence of solutions which behave as the superposition of several traveling waves at a sequence of time, and  then with uniform estimates and weak continuity of solutions flows, one uses the weak convergence to pass to a limit which would be the desired pure multi-soliton. Since a weak convergence argument is involved, one loses the information about the uniqueness and the classification pure multi-solitons.  
	In the case of the generalized Korteweg-de Vries equation,
	the problem of existence and uniqueness of pure multi-solitons
	was solved by Martel \cite{Mart} in the subcritical and critical cases and by Combet \cite{Com} in supercrical settings. Again in this paper, specific monotonicity formulae are key tools.  In the works by C\^ote-Friederich \cite{CoteF} and Friederich \cite{Fri}, assuming certain algebraic decay rates in time, they are able to prove uniqueness of multi-solitons in those classes for various models.  In our recent work, Chen-Jendrej \cite{CJ3}, instead of the weak convergence argument, we used a fixed point argument, which naturally results in the uniqueness,  to construct pure multi-kink (soliton) solutions for $1+1$ scalar field models but we have to restrict ourselves onto the class of exponential multi-kink solution, i.e., solutions converging to multi-kink exponentially. We would also like to mention the work by Le Coz-Li-Tsai \cite{LLT} where the fixed point argument is also applied but in a different manner  for fast spreading
	(NLS) multi-solitons.   Formally, one should expect the exponential decay rates for any pure multi-solitons since solitons in these cases decay exponentially,  the assumption of different speeds will give us the exponential decay of the interaction of different solions in time. In this paper, we confirm this intuition. With the assumption of convergence \emph{without} any decay rate in the energy space, we show that any pure multi-solitons for \eqref{eq:nkg}  converges to the corresponding multi-solitons exponentially. This also illustrate the intuition that pure multi-soltion solution are quite special. They might have some properties of solutions to elliptic equations, i.e., the existence of solutions can be found by energy arguments but in fact these solutions have better behavior like smoothness and exponential decay. It is also possible to see that pure multi-solitons also have spacial exponential decay but this is not our goal in this paper. We refer to our computations in \cite{CJ2} for some details. Our classification of pure multi-soliton solutions shows that under the assumption of distinct velocities, they should form a finite-dimension manifold.
	
	Although in this paper, we focus on the cubic  nonlinear Klein-Gordon equation \eqref{eq:nkg}  in $\mathbb{R}^{1+3}$,  we believe that the ideas and techniques developed in our works will be useful in the study of multi-solitons in  other dispersive models.
	
	\subsection{Basic settings and main results}
	
	%In spite of this, I believe \eqref{eq:nkg} is mathematically the simplest
	%sub-critical wave model allowing to study multi-solitons.
	%In fact, the problem of the linear instability of the soliton
	%is related to the classical question of \emph{exponential dychotomies},
	%which can be dealt with in multiple ways, and for the Klein-Gordon equation
	%considered here 
	%it was solved %in my joint work with Chen \cite{CJ}.

	%The goal of this paper has two folds: firstly, we establish the conditional
	%stability of multi-soliton solutions. secondly, we also give a classification
	%of solutions which approach the superposition of several solitons
	%as time goes to $\infty$ .
	
	It is convenient to rewrite the equation \eqref{eq:nkg}  as a dynamical system in the energy space $$\mathcal{H}:=H^1(\mathbb{R}^3)\times L^2 (\mathbb{R}^3)$$ using its Hamiltonian  formalism. 
	Denoting\[\boldsymbol{\psi}\left(t\right)=\left(\begin{array}{c}
		\psi\\
		\psi_{t}
	\end{array}\right)\] we can write
	\begin{equation}
		\partial_{t}\boldsymbol{\psi}\left(t\right)=JH_0\boldsymbol{\psi}\left(t\right)+\hm{F}(\hm{\psi}),\label{eq:dynkg}
	\end{equation}
	where $\hm{F}\text{(\ensuremath{\hm{\psi}})=\ensuremath{\left(\begin{array}{c}
				0\\
				\psi^{3}
			\end{array}\right)}}$ 
	and			
	\begin{equation}\label{eq:J}
		J:=\left(\begin{array}{cc}
			0 & 1\\
			-1 & 0
		\end{array}\right),\,\ \ H_0:=\left(\begin{array}{cc}
			-\Delta+1 & 0\\
			0 & 1
		\end{array}\right).
	\end{equation}
	In this paper, we are interested in the multi-soliton structure consisting of the Lorentz transformed  ground
	states. Recall that a ground state refers to the unique least energy stationary state $Q$ which solves
	\begin{equation}\label{eq:Qintro}
		-\Delta Q+Q-Q^{3}=0.   
	\end{equation}
	For $\beta\in\mathbb{R}^3$ such that $|\beta|<1$, we consider the corresponding Lorentz boots of the ground state $Q_\beta$, see \eqref{eq:lorentz}.  Taking the vector version of the Lorentz boost of the ground state, we have  
	\begin{equation}\label{eq:vecQ}
		\hm{Q}_{\beta}:=(Q_{\beta},-\beta\cdot\nabla Q_{\beta}).
	\end{equation}	
	%	$$ $\sigma_{j}\in\{\pm1\}$. 
	Using notations above, the traveling wave in the Hamiltonian formalism
	is given by $\hm{Q}_{\beta}(x-\beta t-y)$. 
	
	Given a fixed natural number
	$N$,  \emph{distinct} Lorentz parameters
	\begin{equation}\label{eq:beta}
		\hm{\beta}=(\beta_{1},\ldots\beta_{N})\in\mathbb{R}^{3N},
	\end{equation} a set of shifts 
	\begin{equation}\label{eq:shift}
		\hm{y}=(y_{1},\ldots,y_{N})\in\mathbb{R}^{3N}  
	\end{equation} and a set of signs
	\begin{equation}\label{eq:signs}
		\hm{\sigma}=(\sigma_{1},\ldots\sigma_{N}),\,\sigma_{j}\in\{\pm1\},
	\end{equation}
	we consider the multi-soliton structure
	given by
	\begin{equation}\label{eq:sumQ}
		\hm{Q}\text{(\ensuremath{\hm{\beta}},\ensuremath{\hm{y}})=\ensuremath{\sum_{j=1}^N\sigma_{j}\hm{Q}_{\beta_{j}}(\cdot-y_{j})}.}
	\end{equation}
	We define the multi-soliton family consisting of $N$ solitons as the collection of multi-solitons
	\begin{equation}\label{eq:mutifamil}
		\mathscr{F}_{\hm{Q}}=\left\{ 	\hm{Q}(\hm{\beta},\hm{y})\large|\,\hm{\beta}\in\mathbb{R}^{3N},\hm{y}\in\mathbb{R}^{3N}\right\}.
	\end{equation}

	\subsection{Main results}

	\label{ssec:results}
	If the trajectories of the solitons approach each other,
	these solitons interact and their behaviour can be complicated which lead to the collision problem.
	Such a study is not our goal in this paper. Throughout this paper, 
	we will always assume that the trajectories of the solitons
	satisfy a \emph{separation condition} which we introduce below.
	\begin{defn}
		\label{def:sep-1}
		We say that the functions $y_j: [0, \infty) \to \bR^3,\,j=1,\ldots,N$
		satisfy the \emph{separation condition} with parameters
		$\rho, \delta > 0$ if
		\begin{equation}
			|y_j(t) - y_k(t)| \geq \delta t + \rho,\qquad\text{for all }t \geq 0\text{ and }j \neq k.
		\end{equation}
	\end{defn}
	\begin{defn}\label{def:sep}
		We say that the vectors $y_{1}^{in}, \ldots, y_{N}^{in}, \beta_{1}^{in}, \ldots, \beta_{N}^{in} \in \bR^3$
		satisfy the \emph{separation condition} with parameters
		$\rho, \delta > 0$ if the affine functions $y_j(t) = y_{j}^{in} + \beta_{j}^{in} t$ satisfy the separation condition with the same parameters. 
		
		%We will denote the  collection of these vectors as $\mathfrak{S}_{\delta,\rho}$. Using these notations, we have $\big(\hm{y}^{in},\hm{\beta}^{in}\big)\in \mathfrak{S}_{\delta,\rho}$.
	\end{defn}
	
	\begin{defn}\label{def:sepfamily} Given $\delta>0$ and $\rho>0$, next we define the well-separated multi-soliton family as
		\begin{equation}\label{eq:wellfamily}
			\mathfrak{S}_{\delta,\rho}:=\left\{ 	\hm{Q}(\hm{\beta},\hm{y})\large|\,\hm{\beta}\in\mathbb{R}^{3N},\hm{y}\in\mathbb{R}^{3N},\,\left|y_j+\beta_j t-(y_k+\beta_kt)\right|\geq\delta t+\rho, t\geq 0, j\neq k \right\}. 
		\end{equation}
	\end{defn}	
	First of all, we have the following asymptotic stability of the multi-solitons staying close to the multi-soliton family. 
	
	%We first claim that the orbital stability of the multi-soliton solution
	%implies the asymptotic stability. Here we need a separation conditions.
	%	Maybe we say $\beta_{i}\neq\beta_{j}$ and $|y_{i}-y_{j}|\gtrsim\frac{1}{\epsilon}t$
	%	etc? how to formulate this?
	\begin{thm}\label{thm:orbitasy}
		For every $\delta > 0$ there exist $\rho, \eta > 0$ such that
		the following holds. Let the initial parameters  $(y_j^{in}, \beta_j^{in})_{j=1}^J$
		satisfy the separation condition in the sense of Definition \ref{def:sep}, and let
		\begin{equation}
			\|\bs \psi_0 - \bs Q( \hm{\beta}^{in}, \hm{y}^{in})\|_{\mathcal{H}} \leq \eta.
		\end{equation}
		If the corresponding solution $\bs \psi$ stays in  a neighborhood of the multi-soliton family \footnote{Note that in particular all pure multi-soliton solutions constructed in \cite{CMu} satisfy this condition. Moreover Corollary \ref{cor:orbitalmanifoldintro} below will give a full characterization of the set of solutions satisfying this condition.}:
		\begin{equation}
			\sup_{t\in\mathbb{R}^{+}}\inf_{\hm{\beta}\in\mathbb{R}^{3N},\hm{y}\in\mathbb{R}^{3N}}\left\Vert \boldsymbol{\psi}(t)-\boldsymbol{Q}(\hm{\beta},\hm{y})\right\Vert _{\mathcal{H}}\lesssim\eta,\label{eq:orbcondintro}
		\end{equation}
		then $\bs \psi$ scatters to the multi-soliton family:
		there exist $\beta_{j}\in\mathbb{R}^{3}$,
		paths $y_{j}(t)\in\mathbb{R}^{3}$ and $\boldsymbol{\psi}_{+}\in \mathcal{H}$ with
		the properties that $\dot{y}_{j}(t)\rightarrow\beta_{j}$ and
		\[
		\lim_{t\rightarrow\infty}\text{\ensuremath{\left\Vert \boldsymbol{\psi}(t)-\boldsymbol{Q}(\hm{\beta},\hm{y}(t))-e^{JH_{0}t}\boldsymbol{\psi}_{+}\right\Vert _{\mathcal{H}}}}=0.
		\]
	\end{thm}
	Next we show that there is a refined structure in the neighborhood of the well-separated multi-soliton family: there is a finite-codimension smooth central-stable manifold in the small neighborhood of the well-separated multi-soliton family.

	\begin{thm}\label{thm:globalmaniintro}
		Fixed a natural number $N$, given $\delta>0$, there exists  $\rho>0$ large, such that there exists a codimension $N$ Lipschitz centre-stable manifold $\mathcal{N}$ around the well-separated multi-soliton family $\mathfrak{S}_{\delta,\rho}$ which is invariant for $t\geq0$ such that for any choice of initial data $\hm{\psi}(0)\in\mathcal{N}$,
		the solution $\hm{\psi}(t)$  to \eqref{eq:dynkg} with initial data $\hm{\psi}(0)$
		exists globally, and it scatters to the multi-soliton family: there
		exist $\beta_{j}\in\mathbb{R}^{3}$, paths $y_{j}(t)\in\mathbb{R}^{3}$
		and $\boldsymbol{\psi}_{+}\in\mathcal{H}$ with the property that $\dot{y}_{j}(t)\rightarrow\beta_{j}$
		and
		\[
		\lim_{t\rightarrow\infty}\text{\ensuremath{\left\Vert \boldsymbol{\psi}(t)-\boldsymbol{Q}(\hm{\beta},\hm{y}(t))-e^{JH_{0}t}\boldsymbol{\psi}_{+}\right\Vert _{\mathcal{H}}}}=0.
		\]
	\end{thm}
	One can further conclude that all orbitally stable solutions are in $\mathcal{N}$ and its converse is also true. 
	\begin{cor}\label{cor:orbitalmanifoldintro}
		For every $\delta > 0$ there exist $\rho, \eta > 0$ such that
		the following holds. Let the initial parameters $(y_j^{in}, \beta_j^{in})_{j=1}^J$
		satisfy the separation condition in the sense of Definition \ref{def:sep}, and let
		\begin{equation}
			\|\bs \psi_0 - \bs Q( \hm{\beta}^{in}, \hm{y}^{in})\|_{\mathcal{H}} \leq \eta.
		\end{equation}
		If the solution $\bs \psi$ to \eqref{eq:dynkg} with initial data $\hm{\psi}_0$  stays  in  a neighborhood of multi-soliton family:
		\begin{equation}\label{eq:orbitalcondintor}
			\sup_{t\in\mathbb{R}^{+}}\inf_{\hm{\beta}\in\mathbb{R}^{3N},\hm{y}\in\mathbb{R}^{3N}}\left\Vert \boldsymbol{\psi}(t)-\boldsymbol{Q}(\hm{\beta},\hm{y})\right\Vert _{\mathcal{H}}\lesssim\eta,
		\end{equation}
		then $\bs \psi(t)\in\mathcal{N}$. On the other hand,  any solution $\hm{\psi}(t)$ with initial data $\hm{\psi}(0)\in\mathcal{N}$ satisfies the orbital stability condition \eqref{eq:orbitalcondintor}.
	\end{cor}
	Next two theorems are concerned with pure multi-solitons. We first introduce a more general setting of pure multi-solitons.
	\begin{defn}\label{def:puremulti}
		We say that a solution $\hm{\psi}$ to the equation \eqref{eq:dynkg} is a pure multi-soliton if there exist  $\hm{\beta}\in\mathbb{R}^{3N}$
		satisfying $|\beta_{j}|<1$,  $\beta_{j}\neq\beta_{k}$ for $j\neq k$,  and  $\hm{y}_{p}(t)\in \mathbb{R}^{3N}$ satisfying $|y_{p,j}(t)-y_{p,k}(t)|\geq L\gg 1$ for  all $t\geq 0$ such that
		\begin{equation}
			\lim_{t\rightarrow\infty}\text{\ensuremath{\left\Vert \boldsymbol{\psi}(t)-\boldsymbol{Q}(\hm{\beta},\hm{y}_{p}(t))\right\Vert _{\mathcal{H}}}}=0\label{eq:puremultidef}
		\end{equation}
	\end{defn}
	
	First of all, we have that every pure multi-soliton actually converges to a multi-soliton with an exponential decay rate.
	\begin{thm}\label{thm:pureexp}
		Suppose $\hm{\psi}(t)$ is a pure multi-soliton in the sense of Definition \ref{def:puremulti}. Then actually there exists $\hm{x}_0\in\mathbb{R}^{3N}$ such that one has
		\begin{equation}
			\text{\text{\ensuremath{\left\Vert \boldsymbol{\psi}(t)-\boldsymbol{Q}(\hm{\beta},\hm{\beta}t+\hm{x}_{0})\right\Vert _{\mathcal{H}}}}\ensuremath{\lesssim e^{-\rho_0 t}}}
		\end{equation}
		for small $\rho_0>0$ which is independent of $\hm{\psi}$. 
	\end{thm}
	Finally, regarding the classification of pure multi-solitons, we have the following theorem:	
	\begin{thm}\label{thm:pureclas}
		For fixed $\hm{\beta}=\text{(\ensuremath{\beta_{1},\ldots,\beta_{N})} \ensuremath{\in\mathbb{R}^{3N}}}$
		satisfying $|\beta_{j}|<1$ and $\beta_{j}\neq\beta_{k}$ for $j\neq k$
		and $\hm{x}_{0}=(x_{1},\ldots,x_{N})\in\mathbb{R}^{3N}$, the set
		of solution $\hm{\psi}$ to \eqref{eq:dynkg} satisfying
		\begin{equation}
			\lim_{t\rightarrow\infty}\text{\ensuremath{\left\Vert \boldsymbol{\psi}(t)-\boldsymbol{Q}(\hm{\beta},\hm{\beta}t+\hm{x}_{0})\right\Vert _{\mathcal{H}}}}=0
		\end{equation} is a dimension
		$N$ Lipschitz manifold.
	\end{thm}	
	
	\subsection{Outline of the paper}
	This paper is organized as follows: In Section \ref{sec:prelim}, we recall basic linear theory including spectral theory, Strichartz estimates. We will also include the modulation computations for the sake of completeness.  In Section \ref{sec:manifold}, we  show that orbital stability of multi-solitons implies their asymptotic stability. Finally, in Section \ref{sec:cla}, we analyze pure multi-solitons and classify them.
	\subsection{Notations}
	
	\textquotedblleft $A:=B\lyxmathsym{\textquotedblright}$ or $\lyxmathsym{\textquotedblleft}B=:A\lyxmathsym{\textquotedblright}$
	is the definition of $A$ by means of the expression $B$. We use
	the notation $\langle x\rangle=\left(1+|x|^{2}\right)^{\frac{1}{2}}$.
	The bracket $\left\langle \cdot,\cdot\right\rangle $ denotes the
	distributional pairing and the scalar product in the spaces $L^{2}$,
	$L^{2}\times L^{2}$ . For positive quantities $a$ and $b$, we write
	$a\lesssim b$ for $a\leq Cb$ where $C$ is some prescribed constant.
	Also $a\simeq b$ for $a\lesssim b$ and $b\lesssim a$. We denote
	$B_{R}(x)$ the open ball of centered at $x$ with radius $R$ in
	$\mathbb{R}^{3}$. We also denote by $\chi$ a standard $C^{\infty}$
	cut-off function, that is $\chi(x)=1$ for $\left|x\right|\leq1$,
	$\chi(x)=0$ for $\left|x\right|>2$ and $0\leq\chi(x)\leq1$ for
	$1\leq\left|x\right|\leq2$ .

	We always use the bold font to denote a pair of functions as a vector function.  For example, $\hm{u}=\left(\begin{array}{c}
		u_1\\
		u_2
	\end{array}\right)$  and $\hm{f}=\left(\begin{array}{c}
		f_1\\
		f_2
	\end{array}\right)$.   We will use the energy space  $\mathcal{H}:=H^{1}\times L^{2}$.
	For a general element in $\boldsymbol{v}\in \mathcal{H}$, $v_{1}$
	denotes the first row and $v_{2}$ denotes the second
	row of $\boldsymbol{v}$ respectively. For any space $X$ which measures a scalar function $f$ in the norm $\Vert f\Vert _X$, we use the notation $X_\mathcal{H}$ to measure the corresponding vector function $\hm{f}=\left(\begin{array}{c}
		f_1\\
		f_2
	\end{array}\right)$ by the norm
	\begin{equation}\label{eq:X_H}
		\Vert\hm{f}\Vert_{X_\mathcal{H}}=\Vert (\sqrt{-\Delta+1})f_1\Vert _X+\Vert f_2\Vert _X.
	\end{equation}
	
	\subsection{Acknowledgement} We would like to thank Marius Beceanu and Chongchun Zeng for enlightening discussions.
	\section{Preliminaries}\label{sec:prelim}
	In this section, we collect several basic facts and fundamental tools we will use in our nonlinear analysis.
	\subsection{Spectral theory}\label{subsec:spectral}

	Recall that, up to translation in space and sign change, the equation \eqref{eq:nkg}
	has a unique least energy stationary state $Q$ which  is called the  ground state and solves
	\begin{equation}\label{eq:Q}
		-\Delta Q+Q-Q^{3}=0.   
	\end{equation}
	Moreover, we also recall that this solution $Q$ is linear unstable in the sense that the linearized operator near $Q$
	\begin{equation}\label{eq:Lop}
		L:=-\Delta+1+V=-\Delta+1-3Q^2
	\end{equation}has a negative eigenvalue $-\nu^2$ with the associated eigenfunction $\phi$:
	\begin{equation}\label{eq:negaeigen}
		L\phi=-\Delta\phi+\phi-3Q^2\phi=-\nu^2\phi.
	\end{equation}
	Due the translational invariance of the equation, one also has
	\begin{equation}\label{eq:zeromodes}
		L\phi^0_m=-\Delta\phi^0_m+\phi^0_m-3Q^2\phi^0_m=0
	\end{equation}where $\phi^0_m=\partial_{x_m}Q,\,m=1,2,3.$  The continuous spectrum of $L$ is $[1,\infty)$. Importantly, $L$ has no gap eigenvalues in $(0,1]$ nor threshold resonance at $1$, see for example the numerical verification in  Demanet-Schlag \cite{DS}. We also refer to the book by Nakanishi-Schlag \cite{NSch1} and reference therein for details.
	
	Fixed a $\beta=(\beta^{1},\beta^{2},\beta^{3})$ with $\left|\beta\right|<1$, using the Hamiltonian formalism, the equation for $\eqref{eq:Q}$ and notations from \eqref{eq:J}, the vector  $\hm{Q}_{\beta}$, see \eqref{eq:vecQ}, satisfies
	\begin{equation}\label{eq:vecQeq}
		JH_{0}\hm{Q}_{\beta}+\beta\cdot\nabla\hm{Q}_{\beta}+\hm{F}(\hm{Q}_{\beta})=0
	\end{equation}
	where $\hm{F}\text{(\ensuremath{\hm{v}})=\ensuremath{\left(\begin{array}{c}
				0\\
				v_{1}^{3}
			\end{array}\right)}}$ for $\hm{v}=\left(\begin{array}{c}
		v_{1}\\
		v_{2}
	\end{array}\right)$.
	
	Given the equation above, we can differentiate \eqref{eq:vecQeq} respect
	to $x_{m}$, $m=1,2,3$ and $\beta^{m},$$m=1,2,3$.  First of all, differentiating with respect to $x_{m}$, we get
	\begin{equation}\label{eq:kerQ}
		JH_{v_{\beta}}\left(\partial_{x_{m}}\hm{Q}_{\beta}\right)+\beta\cdot\nabla\left(\partial_{x_{m}}\hm{Q}_{\beta}\right)=0
	\end{equation}
	where 
	\begin{equation}\label{eq:Hv}
		H_{V_{\beta}}:=\left(\begin{array}{cc}
			\Delta-1+3Q_{\beta}^{2} & 0\\
			0 & 1
		\end{array}\right)=:H_0+\mathcal{V}_{\beta}.
	\end{equation}
	
	Differentiating with respect to $\beta^{m}$, one has
	\begin{equation}\label{eq:gkerQ}
		JH_{v_{\beta}}\left(\partial_{\beta^{m}}\hm{Q}_{\beta}\right)+\beta\cdot\nabla\left(\partial_{\beta^{m}}\hm{Q}_{\beta}\right)=-\partial_{x_{m}}\hm{Q}_{\beta},
	\end{equation}

	%	Need to fix the notations. 
	
	%		\[
	%	\mathcal{V}_{\beta}=\left(\begin{array}{cc}
	%	0 & 0\\
	%	-V_{\beta} & 0
	%	\end{array}\right).
	%	\]
	
	\iffalse
	the following linear Klein-Gordon equation
	\begin{equation}
		\partial_{t}^{2}u=\Delta u-u-\sum_{j=1}^{J}(V_{j})_{\beta_{j}(t)}(\cdot-y_{j}(t))u+F\label{eq:maineq}
	\end{equation}
	under assumptions \eqref{eq:maincond}. 
	\iffalse
	\noindent\textbf{Assumption on potentials.} 
	Let $V_{j}$ be a smooth exponentially decaying
	potential for $j\in\{1,2,\ldots,J\}$, such that
	\begin{equation}
		L_{j}:=-\Delta+V_{j}+1\label{eq:schop}
	\end{equation}
	has $K_{j}$ strictly negative eigenvalues $-\nu^{2}$ with
	$\nu>0$ (for $k=1,\ldots,K_{j}$) and $M_{j}=\dim\ker L_{j}$.
	Let $\left(\phi\right)_{k=1,\ldots,K_{j}}$ and $\left(\phi_{j,m}^{0}\right)_{m=1,\ldots,M_{j}}$
	be orthonormal (in $L^{2}$) families such that
	\begin{equation}
		L_{j}\phi=-\nu^{2}\phi_{k},\,\,L_{j}\phi_{j,m}^{0}=0.\label{eq:modescalar}
	\end{equation}

	Now we introduce notations and recall
	exponential dichotomy for the Klein-Gordon equation from Chen-Jendrej
	\cite{CJ}. For an alternative approach, we refer to Metcalfe-Sterbenz-Tataru \cite{MST}.
	\fi	
	Consider the homogeneous version of our main equation \eqref{eq:maineq}
	\begin{equation}\label{eq:homomaineq}
		\partial_{t}^{2}u=\Delta u-u-\sum_{j=1}^{J}(V_{j})_{\beta_{j}(t)}(\cdot-y_{j}(t))u. 
	\end{equation}
	\fi
	To study the stability properties of multi-solitons, the linearzation of the equation around a multi-soliton is given as the following:
	denoting $V_\beta=- 3Q^2_{\beta}$ and \[\boldsymbol{u}\left(t\right)=\left(\begin{array}{c}
		u\\
		u_{t}
	\end{array}\right)\] then the lineration around $\hm{Q}(\hm{\beta}(t),\hm{y}(t))$, c.f. \eqref{eq:sumQ}, is given as
	\begin{equation}
		\partial_{t}\boldsymbol{u}\left(t\right)=JH\left(t\right)\boldsymbol{u}\left(t\right),\label{eq:dy1}
	\end{equation}
	where
	\[
	J:=\left(\begin{array}{cc}
		0 & 1\\
		-1 & 0
	\end{array}\right),\,\ \ H\left(t\right):=\left(\begin{array}{cc}
		-\Delta+1+\sum_{j=1}^{N} V_{\beta_{j}(t)}(\cdot-y_{j}(t)) & 0\\
		0 & 1
	\end{array}\right).
	\]
	Denote the evolution operator of the above system as $\mathcal{T}\left(t,s\right)$.

	%\smallskip
	From the spectral computations  on each potential, see \eqref{eq:negaeigen} and \eqref{eq:zeromodes}, following C\^ote-Martel \cite{CMart}, C\^ote-Martel \cite[Lemma 1]{CMu} and Chen-Jendrej \cite{CJ},  we now give explicit formulae for the stable, unstable
	and (iterated) null components of the flow.
	%\eqref{eq:kg-one-ham}. 
	First of all, we define
	\begin{align}
		\cY_{\beta}^-(x) &:= \eee^{\gamma\nu\beta\cdot x}(\phi -\gamma\beta\cdot\grad \phi-\gamma\nu\phi)_\beta(x), \label{eq:kg-Ym-def} \\
		\cY_{\beta}^+(x) &:= \eee^{-\gamma\nu\beta\cdot x}
		(\phi, -\gamma\beta\cdot\grad \phi+\gamma\nu\phi)_\beta(x),\label{eq:kg-Yp-def} 
		\\ \cY_{m, \beta}^0(x) &:= (\phi_{j,m}^0, -\gamma\beta\cdot\grad \phi_{j,m}^0)_\beta(x),\label{eq:kg-Y0-def} 
		\\ \cY_{m, \beta}^1(x) &:= ({-}(\beta \cdot x)\phi_{j,m}^0, \gamma \phi_{j,m}^0 + \gamma(\beta\cdot x)(\beta\cdot \grad \phi_{j,m}^0))_\beta(x),\label{eq:kg-Y1-def} 
		\\ \alpha_{\beta}^-(x) &:= J\cY_{k, \beta}^+(x) = \eee^{-\gamma\nu\beta\cdot x}({-}\gamma\beta\cdot\grad \phi + \gamma\nu\phi, -\phi)_\beta(x), \label{eq:kg-am-def} \\
		\alpha_{\beta}^+(x) &:= J\cY_{k, \beta}^-(x) = \eee^{\gamma\nu\beta\cdot x}({-}\gamma\beta\cdot\grad \phi - \gamma\nu\phi, -\phi)_\beta(x), \label{eq:kg-ap-def} \\
		\alpha_{m, \beta}^0(x) &:= J\cY_{m, \beta}^0(x) = ({-}\gamma \beta\cdot\grad \phi_{j,m}^0, -\phi_{j,m}^0)_\beta(x),\label{eq:kg-a0-def} 
		\\ \alpha_{m, \beta}^1(x) &:= J\cY_{m, \beta}^1(x) = (\gamma \phi_{j,m}^0 + \gamma(\beta\cdot x)(\beta\cdot \grad \phi_{j,m}^0), (\beta \cdot x)\phi_{j,m}^0)_\beta(x)\label{eq:kg-a1-def} 
	\end{align} 	where  $m \in \{1,2,3\}$. 
	In particular, using the vector forms and explicit  computations, one has	
	\begin{align}
		\mathcal{Y}_{m,\beta}^{0}(x)=\partial_{x_{m}}\hm{Q}_{\beta}(x),\\
		\mathcal{Y}_{m,\beta}^{1}(x)=\frac{1}{\gamma^{2}}\partial_{\beta^m}\hm{Q}_{\beta}(x).
	\end{align}
	Then we have the following important modes for the moving-potential problem \eqref{eq:dy1} above:
	\begin{align*}
		\cY_{j}^-(t) &:= \cY_{\beta_j(t)}^-(\cdot - y_j(t)), \\
		\cY_{j}^+(t) &:= \cY_{\beta_j(t)}^+(\cdot - y_j(t)), \\
		\cY_{j, m}^0(t) &:= \cY_{m, \beta_j(t)}^0(\cdot - y_j(t)), \\
		\cY_{j, m}^1(t) &:= \cY_{m, \beta_j(t)}^1(\cdot - y_j(t)), \\
		\alpha_{j}^-(t) &:= \alpha_{\beta_j(t)}^-(\cdot - y_j(t)), \\
		\alpha_{j}^+(t) &:= \alpha_{\beta_j(t)}^+(\cdot - y_j(t)), \\
		\alpha_{j, m}^0(t) &:= \alpha_{m, \beta_j(t)}^0(\cdot - y_j(t)), \\
		\alpha_{j, m}^1(t) &:= \alpha_{m, \beta_j(t)}^1(\cdot - y_j(t)),
		%\\
		%V(t) &:= \sum_{j=1}^J (V_j)_{\beta_j(t)}(\cdot - y_j(t)),
	\end{align*}
	To obtain dispersive estimates, naturally, one has to restrict flows onto the centre-stable space.  In this subspace, one needs to further remove  null components to avoid the polynomial growth.
	\begin{defn}\label{def:zeroproj}
		We define the projection $\pi_0(t)$ as the the projection onto the subspace spanned by $\cY_{j, m}^0(t)$ and $\cY_{j, m}^1(t)$.  We also define  $\pi_\pm(t)$ as projections onto the spans of $\cY_{j}^{\pm}(t)$ respectively.  Finally, we define $\pi_c(t)=1-\pi_0(t)-\pi_+(t)-\pi_-(t).$
	\end{defn}

	\subsubsection{Functional spaces}
	
	Next, we recall important functional spaces in our analysis. Setting
	\begin{equation}
		\mathcal{D}:=\sqrt{1-\Delta},\label{eq:D}
	\end{equation}
	firstly we define the space
	\begin{equation}
		\mathcal{D}^{\frac{\nu}{2}}L_{x}^{2}:=\left\{ g:\,g=\mathcal{D}^{\frac{\nu}{2}}f,\,\text{for some}\,f\in L_{x}^{2}\right\} .\label{eq:Dweightspace}
	\end{equation}
	Given $g\in\mathcal{D}^{\frac{\nu}{2}}L_{x}^{2}$ with $g=\mathcal{D}^{\frac{\nu}{2}}f$, 
	then norm of it given by
	\begin{equation}
		\text{\ensuremath{\left\Vert g\right\Vert }}_{\mathcal{D}^{\frac{\nu}{2}}L_{x}^{2}}:=\text{\ensuremath{\left\Vert \mathcal{D}^{-\frac{\nu}{2}}g\right\Vert }}_{L_{x}^{2}}=\text{\ensuremath{\left\Vert f\right\Vert }}_{L_{x}^{2}}.\label{eq:Dweightnorm}
	\end{equation}
	Secondly, we define the weighed $L^{2}$ space with the center $y$
	as
	\begin{equation}
		L_{x}^{2}\left\langle \cdot-y\right\rangle ^{\sigma}:=\left\{ g:g=\left\langle \cdot-y\right\rangle ^{-\alpha}f,\,\text{for some}\,f\in L_{x}^{2}\right\} .\label{eq:Xweightspace}
	\end{equation}
	Given $g\in L_{x}^{2}\left\langle \cdot-y\right\rangle ^{\sigma}$ with
	$g=\left\langle \cdot-y\right\rangle ^{-\alpha}f$ then norm of it given
	by
	\begin{equation}
		\text{\ensuremath{\left\Vert g\right\Vert }}_{L_{x}^{2}\left\langle \cdot-y\right\rangle ^{\sigma}}:=\text{\ensuremath{\left\Vert \left\langle \cdot-y\right\rangle ^{\alpha}g\right\Vert }}_{L_{x}^{2}}=\text{\ensuremath{\left\Vert f\right\Vert }}_{L_{x}^{2}}.\label{eq:Xweightnorm}
	\end{equation}
	Let $B_{p,q}^{s}(\mathbb{R}^{d})$ denote the inhomogeneous Besov
	space based on $L^{p}(\mathbb{R}^{d})$ for any $d\ge1$, $s\in\mathbb{R}$
	and $p,q\in[1,\infty]$. For brevity, we use the standard notation
	\[
	H^{s}:=B_{2,2}^{s},\ C^{s}:=B_{\infty,\infty}^{s}.
	\]
	The homogeneous versions are denoted by $\dot{B}_{p,q}^{s}$, $\dot{H}^{s}$
	and $\dot{C}^{s}$, respectively. For $s\in(0,1)$, we have the equivalent
	semi-norms by the difference, see Bergh-L\"ofstr\"om \cite[Exercice 7,  page 162]{BL} and  Bahouri-Chemin-Danchin \cite[Theorem 2.36]{BCD},
	\[
	\|\varphi\|_{\dot{B}_{p,q}^{s}}\simeq\|\sup_{|y|\le\sigma}\|\varphi(x)-\varphi(x-y)\|_{L^{p}}\|_{L^{q}(d\sigma/\sigma)}.
	\]
	Finally, the weighted $L^{2}$ space,  $L^{2,s}(\mathbb{R}^{d})$, is defined
	by the norm
	\[
	\|\varphi\|_{L^{2,s}\left(\mathbb{R}^{d}\right)}=\|\left\langle x\right\rangle ^{s}\varphi\|_{L^{2}(\mathbb{R}^{d})}
	\]
	for any $s\in\mathbb{R}$. Hence $L^{2,s}$ is the Fourier image
	of $H^{s}$.

	\subsection{Strichartz estimates}

	Motivated by the modulation and endpoint Strichartz estimates, we have the following assumptions on trajectories.
	
	\noindent\textbf{Assumption on trajectories.} 
	Let $y_{j}(t)$ be positions of the potentials. Denote $\beta_{j}(t)$
	the Lorentz factor for the potential $V_j$. We write $\boldsymbol{\beta}(t)=(\beta_{1}(t),\ldots,\beta_{N}(t))$,
	$\boldsymbol{y}(t)=(y_{1}(t),\ldots,y_{N}(t))$.
	Motivated by the modulation equations from the stability analysis, we impose that $y_{j}\left(0\right)=y_j,$
	$\beta_{j}\left(0\right)=\beta_{j}$ with $|\beta_j|<1$, $\beta_j\neq \beta_k$ for $k\neq j$ and
	\begin{equation}
		\left\Vert \beta_{j}'\left(t\right)\right\Vert _{L_{t}^{1}\bigcap L_{t}^{\infty}}+\left\Vert y'_{j}\left(t\right)-\beta_{j}\left(t\right)\right\Vert _{L_{t}^{1}\bigcap L_{t}^{\infty}}\ll1.\label{eq:maincond}
	\end{equation}	
	Now we recall Strichartz estimates for the associated Klein-Gordon
	models from \cite{CJ2}, see Theorem 1.4 there.  We reformulate the statement adapted to nonlinear applications in this paper.
	\begin{thm}[\cite{CJ2}]
		\label{thm:mainthmlinear}For $0<\zeta\ll1$ and $\sigma>14$, denote
		\[
		\mathfrak{W}:=L_{t}^{2}\left(\bigcap_{j=1}^{J}\mathcal{D}^{\frac{\zeta}{2}}L_{x}^{2}\left\langle \cdot-y_{j}\left(t\right)\right\rangle ^{-\sigma}\right)
		\]
		and the Strichartz norm
		\[
		S=L_{t}^{\infty}L_{x}^{2}\bigcap L_{t}^{2}B_{6,2}^{-5/6}.
		\]
		Using notations from \eqref{eq:dy1}, consider the system
		\begin{equation}\label{eq:equ}
			\partial_{t}\boldsymbol{u}\left(t\right)=JH\left(t\right)\boldsymbol{u}\left(t\right)+\boldsymbol{F}
		\end{equation}such that $$\pi_0(t)\bm{u}(t)=0,\,\,\forall t\in\mathbb{R}.$$
		Denote \begin{equation}\label{eq:linearcond}
			\begin{aligned}
				a_j^{\pm}(t):=\la \alpha_{ \beta_j(t)}^{\pm}(\cdot - y_j(t)), \bs u(t)\ra,\, 1 \leq j \leq N
			\end{aligned}
		\end{equation}
		\begin{equation}
			a^\ell_{j,m}:=\la \alpha^\ell_{m, \beta_j(t)}(t)(\cdot - y_j(t)), \bs u(t)\ra,\, 1 \leq j \leq N,\,1\leq m\leq 3,\,\ell=0,1.
		\end{equation}		
		We further set \begin{equation}
			\mathcal{B}(t):=\sum_{\ell=\pm}\sum_{j=1}^N	 \left | 	a_j^{\ell}(t) \right|  +\sum_{\ell=0}^1\sum_{j=1}^N\sum_{m=1}^3 \left | 	a^{\ell}_{m,j}(t) \right |.
		\end{equation}
		Then with the notations above, one has the local energy decay estimate and  Strichartz estimates
		\begin{equation}
			\left\Vert \mathcal{D}u\left(t\right)_{1}\right\Vert _{\mathfrak{W}\bigcap S}+\left\Vert u\left(t\right)_{2}\right\Vert _{\mathfrak{W}\bigcap S}\lesssim\left\Vert \hm{u}\left(0\right)\right\Vert _{\mathcal{H}}+\|\cD F_1\|_{(\mathfrak{W}\cap S)^{*}}+\| F_2\|_{(\mathfrak{W}\cap S)^{*}}+\left\Vert 	\mathcal{B} \right\Vert_{L^1_t\bigcap L^\infty_t}.\label{eq:mainstrichartz}
		\end{equation}
		Moreover, indeed $\boldsymbol{u}\left(t\right)$
		scatters to a free wave. There exists $\boldsymbol{\psi}_{+}\in \mathcal{H}$
		such that
		\begin{equation}
			\left\Vert \boldsymbol{u}\left(t\right)-e^{JH_{0}t}\boldsymbol{\psi}_{+}\right\Vert _{\mathcal{H}}\rightarrow0,\,\,t\rightarrow\infty.\label{eq:mainscatter}
		\end{equation}
	\end{thm}
	%The theorem above is established in the %proof of Theorem \ref{thm:mainu}. 
	%	This theorem  in particular shows that  the solution in the centre-stable space
	%	constructed in Chen-Jendrej \cite{CJ} indeed scatters to the free
	%	linear flow if it is orthogonal to zero modes. 

	\begin{rem}
In the statement of Strichartz estimates in \cite{CJ2}, we have spectral assumptions on the linear operator associated with each potential. Note that the theorem above is stated with potentials coming from the linearization around the ground state. The spectral assumptions in \cite{CJ2} hold in this setting, see \S \ref{subsec:spectral}, in particular this is no resonance on the bottom of the spectrum of each linear operator.
%	Note that the inhomogeneous term \bm{F} on the RHS of  the equation \eqref{eq:equ} can have nontrivial 
	\end{rem}
	\begin{rem}
		One can also restrict the time interval in the estimate \eqref{eq:mainstrichartz}
		to $\left[0,T\right]$. We can also take the initial data at $t_{0}=T$
		and solve the equation backwards. Then one has
		\[
		\left\Vert \mathcal{D}u\left(t\right)_{1}\right\Vert _{\mathfrak{W}\bigcap S}+\left\Vert u\left(t\right)_{2}\right\Vert _{\mathfrak{W}\bigcap S}\lesssim\left\Vert u\left(T\right)\right\Vert _{\mathcal{H}}+\|\cD F_1\|_{(\mathfrak{W}\cap S)^{*}}+\| F_2\|_{(\mathfrak{W}\cap S)^{*}}+\left\Vert 	\mathcal{B} \right\Vert_{L^1_t\bigcap L^\infty_t}.
		\]
	\end{rem}
	\begin{rem}
		In most applications, one always ensures  $	a^{\ell}_{m,j}(t)=0$ via modulation techniques.
	\end{rem}
	\subsection{Modulations}
	If a solution stays  close to the family of multi-solitons, then one can find a time-dependent path in the family of multi-solitons such that  for each time, the difference between  the solution and the multi-soliton parameterized by  this path satisfies the desired orthogonal conditions that Strichartz estimates require.
	
	%	We first recall some notations $\alpha_{j,\beta}^{0}(x)=J\mathcal{Y}_{j,\beta}^{0}(x)$,
	%	$\alpha_{j,\beta}^{1}(x)=J\mathcal{Y}_{j,\beta}^{1}(x)$.
	\begin{lem}\label{lem:mod} There exist $\kappa>0$ large  and $\eta>0$ small such that the following holds. 
		Let $\hm{\psi}(t)$ be a solution to \eqref{eq:dynkg} and
		suppose
		\begin{equation}
			\sup_{t\in\mathbb{R}^{+}}\inf_{\hm{\beta}\in\mathbb{R}^{3N},\hm{y}\in\mathfrak{S}_{\kappa}}\left\Vert \boldsymbol{\psi}(t)-\boldsymbol{Q}(\hm{\beta},\hm{y})\right\Vert _{\mathcal{H}}\lesssim\eta\label{eq:orbcond}
		\end{equation}where $\mathfrak{S}_{\kappa}:=
		\Large\{\hm{y}\in\mathbb{R}^{3N}| |y_j-y_k|\geq \kappa, j\neq k\Large\}$.
		Then there exist Lipschitz functions  $\beta_{j}(t)\in\mathbb{R}^{3}$ and $y_{j}(t)\in\mathbb{R}^{3}$
		such that writing $\hm{u}(t)=\boldsymbol{\psi}(t)-\boldsymbol{Q}\left(\hm{\beta}(t),\hm{y}(t)\right)$,
		one has
		\begin{equation}
			\left\langle \alpha_{m, \beta_j(t)}^1(\cdot - y_j(t)),\hm{u}(t)\right\rangle =\left\langle \alpha_{m, \beta_j(t)}^0(\cdot - y_j(t)),\hm{u}(t)\right\rangle =0\label{eq:modconds}
		\end{equation}
		for $t\in\mathbb{R}^{+}$ and $j=1\ldots N,\,m=1,2,3$.
		
		Moreover, we also have
		\begin{equation}
			|\dot{\beta}_{j}(t)|\lesssim\mathcal{O}\left((\left\Vert \hm{u}\right\Vert _{\mathcal{H}}+1)e^{-\kappa/2}\right)+\sum_{m=1}^{3}\left|\left\langle \alpha_{m,\beta_{j}(t)}^{0}(\cdot-y_{j}(t)),\hm {q}(t)\right\rangle \right|\label{eq:mod1lem}
		\end{equation}
		and
		\begin{equation}
			\left|y'_{j}(t)-\beta_{j}(t)\right|\lesssim\mathcal{O}\left((\left\Vert \hm{u}\right\Vert _{\mathcal{H}}+1)e^{-\kappa/2}\right)+\sum_{m=1}^{3}\left|\left\langle \alpha_{m,\beta_{j}(t)}^{1}(\cdot-y_{j}(t)),\hm {q}(t)\right\rangle \right|\label{eq:mod2lem}
		\end{equation}
		where $\hm {q}(t)=\Large(0,3\sum_{i=1}^{N}Q_{\beta_{i}}(\cdot-y_{i}(t))u^{2}\Large)^T. $	
	\end{lem}
	
	\begin{proof}
		Given condition \eqref{eq:orbcond} and the continuity of the Klein-Gordon
		flow in $\mathcal{H}$, one can find sets of parameters $\tilde{\hm{\beta}}(t)=(\tilde{\beta}_{1}(t),\ldots\tilde{\beta}_{J}(t))$
		and $\tilde{\hm{y}}(t)=(\tilde{y}_{1}(t),\ldots\tilde{y}_{J}(t))$
		such that $\tilde{\hm{{y}}}(t)\in \mathfrak{S}_{\kappa/2}:$
		\[
		\sup_{t\in\mathbb{R}^{+}}\left\Vert \boldsymbol{\psi}(t)-\boldsymbol{Q}\left(\tilde{\hm{\beta}}(t),\tilde{\hm{{y}}}(t)\right)\right\Vert _{\mathcal{H}}\lesssim\eta.
		\]
		Then orthogonality conditions \eqref{eq:modconds} follow from the classical
		implicit function theorem. Indeed, for a fixed $\hm{\psi}\in\mathcal{H}$,
		we define the map
		\[
		\mathfrak{F}:\mathbb{R}^{3N}\times\mathbb{R}^{3N}\rightarrow\mathbb{R}^{3N}\times\mathbb{R}^{3N}
		\]
		as
		\begin{equation}\label{eq:Fmap}
			\mathfrak{F}\left(\hm{\beta}, \hm{y}\right)\mapsto\left(\left\langle \alpha_{m,\beta_{j}}^{1}(\cdot-y_{j}),\hm{u}\right\rangle ,\left\langle \alpha_{m,\beta_{j}}^{0}(\cdot-y_{j}),\hm{u}\right\rangle \right)_{j=1\ldots N,\,m=1,2,3}
		\end{equation}
		where $\hm{u}=\boldsymbol{\psi}-\boldsymbol{Q}\left(\hm{\beta}, \hm{y}\right)$.
		Take a neighborhood of $(\tilde{\hm{\beta}},\tilde{\hm{y}})$ such
		that
		\[
		\left\Vert \boldsymbol{\psi}-\boldsymbol{Q}\left(\hm{\tilde{\beta}},\hm{\tilde{y}}\right)\right\Vert _{\mathcal{H}}\lesssim\eta.
		\]
		Then in this neighborhood, $|\mathfrak{F}\left(\hm{\beta}, \hm{y}\right)|\lesssim \eta.$ 
		Taking partial derivatives of $\mathfrak{F}$ with respect to $\hm{y}$
		and $\hm{\beta}$, given that $\eta$ is sufficiently small, the Jacobian
		is non-singular. From the implicit function theorem, then we can write find a unique $(\hm{\beta}, \hm{y})$
		such that with the decomposition,
		\[
		\hm{u}=\boldsymbol{\psi}-\boldsymbol{Q}\left(\hm{\beta}, \hm{y}\right)
		\]
		one has
		\[
		\left(\left\langle \alpha_{m,\beta_{j}}^{1}(\cdot-y_{j}),\hm{u}\right\rangle ,\left\langle \alpha_{m,\beta_{j}}^{0}(\cdot-y_{j}),\hm{u}\right\rangle \right)_{j=1\ldots N,\,m=1,2,3}=(0,0).
		\]
		We can apply the argument above for each $t\in\mathbb{R}^{+}$ to $\hm{\psi}(t)$ and $\hm{Q}\big(\tilde{\hm{\beta}}(t),\tilde{\hm{{y}}}(t)\big)$. With
		the Lipschitz continuity of $\hm{\psi}(t)$ in $t$, one
		find Lipschitz functions $\left(\hm{\beta}(t),\hm{y}(t)\right)$ such that \eqref{eq:modconds}
		hold.
		
		With $\left(\hm{\beta}(t),\hm{y}(t)\right)$ found above, given the decomposition $\hm{u}(t)=\boldsymbol{\psi}(t)-\boldsymbol{Q}\left(\hm{\beta}(t),\hm{y}(t)\right)$, the
		equation from $\hm{u}(t)$ can be written as 
		\begin{align*}
			\frac{d}{dt}\hm{u}(t) & =JH(t)\hm{u}(t)+\mathcal{I}(Q)+\mathfrak{I}_{1}(Q^{2},u)+\mathfrak{I}_{2}(Q,u^{2})+\hm{F}(u)\\
			& -\dot{\hm{\beta}}(t)\partial_{\hm{\beta}}\hm{Q}(\hm{\beta}(t),\hm{y}(t))-\left(\dot{\hm{y}}(t)-\hm{\beta}(t)\right)\partial_{\hm{y}}\hm{Q}(\hm{\beta}(t),\hm{y}(t))\\
			& =:JH(t)\hm{u}(t)+\mathcal{I}(Q)+\mathfrak{I}(Q,u)+\hm{F}(u)+\text{Mod}'(t)\nabla_{M}\hm{Q}(\hm{\beta}(t),\hm{y}(t))\\
			& =:JH(t)\hm{u}(t)+\mathcal{W}(x,t).
		\end{align*}
		where we introduced the notations
		\begin{align}
			\text{Mod}'(t)\nabla_{M}\hm{Q}(\hm{\beta}(t),\hm{y}(t)) & :=-\dot{\hm{\beta}}(t)\partial_{\hm{\beta}}\hm{Q}(\hm{\beta}(t),\hm{y}(t)) -\left(\dot{\hm{y}}(t)-\hm{\beta}(t)\right)\partial_{\hm{y}}\hm{Q}(\hm{\beta}(t),\hm{y}(t))\label{eq:modnotation} 
		\end{align}
		\begin{equation}
			\mathcal{I}(Q)=\left(\begin{array}{c}
				0\\
				\ensuremath{\left(\sum_{j=1}^{N}\sigma_{j}Q_{\beta_{j}(t)}(\cdot-y_{j}(t))\right)^{3}-\sum_{j=1}^{N}\left(\sigma_{j}Q_{\beta_{j}(t)}(\cdot-y_{j}(t))\right)^{3}}
			\end{array}\right).\label{eq:IQ}
		\end{equation}
		\begin{equation}
			\mathfrak{I}(Q,u)=\left(\begin{array}{c}
				0\\
				3\sum_{i\neq k}\sigma_i\sigma_k Q_{\beta_{i}}(\cdot-y_{i}(t))Q_{\beta_{k}}(\cdot-y_{k}(t))u+	3\ensuremath{\sum_{j=1}^{N}\sigma_{j}Q_{\beta_{j}(t)}(\cdot-y_{j}(t))u^{2}}
			\end{array}\right)\label{eq:JQu}
		\end{equation}
		\begin{equation}
			\hm{F}(u):=\left(\begin{array}{c}
				0\\
				u^{3}
			\end{array}\right)\label{eq:Fcubic}
		\end{equation}
		and sometime we use
		\[
		\mathcal{I}(Q)+\mathfrak{I}(Q,u)+\hm{F}(u)+\text{Mod}'(t)\nabla_{M}\hm{Q}(\hm{\beta}(t),\hm{y}(t))=:\mathcal{W}(x,t).
		\]
		Differentiating the second orthogonality condition in  \eqref{eq:modconds}, we get
		\begin{align*}
			\frac{d}{dt}\left\langle \alpha_{m,\beta_{j}(t)}^{0}(\cdot-y_{j}(t)),\hm{u}(t)\right\rangle  & =0\\
			& =\left\langle \frac{d}{dt}\left(\alpha_{m,\beta_{j}(t)}^{0}(\cdot-y_{j}(t))\right),\hm{u}(t)\right\rangle \\
			& +\left\langle \alpha_{m,\beta_{j}(t)}^{0}(\cdot-y_{j}(t)),\frac{d}{dt}\hm{u}(t)\right\rangle .
		\end{align*}
		First of all, using the equation for $\hm{u}(t)$, one has	
		%	We note that 
		%	\[
		%	\mathcal{Y}_{m,\beta_{j}}^{0}(x)=\partial_{x_{m}}\hm{Q}_{\beta_{j}}(x)
		%	\]
		%	and
		%	\[
		%	\mathcal{Y}_{m,\beta_{j}}^{1}(x)=\frac{1}{\gamma_{j}^{2}}\partial_{\beta_{j,m}}\hm{Q}_{\beta_{j}}(x).
		%	\]
		\begin{align*}
			\left\langle \alpha_{m,\beta_{j}(t)}^{0}(\cdot-y_{j}(t)),\frac{d}{dt}\hm{u}(t)\right\rangle  & =\left\langle \alpha_{m,\beta_{j}(t)}^{0}(\cdot-y_{j}(t)),JH(t)\hm{u}(t)\right\rangle \\
			& +\left\langle \alpha_{m,\beta_{j}(t)}^{0}(\cdot-y_{j}(t)),\mathcal{I}(Q)+\mathfrak{I}(Q,u)+\hm{F}(u)\right\rangle \\
			& \left\langle \alpha_{m,\beta_{j}(t)}^{0}(\cdot-y_{j}(t)),\text{Mod}'(t)\nabla_{M}\hm{Q}(\hm{\beta}(t),\hm{y}(t))\right\rangle \\
			& =\left\langle \alpha_{m,\beta_{j}(t)}^{0}(\cdot-y_{j}(t)),JH_{j}(t)\hm{u}(t)\right\rangle \\
			& +\sum_{i\neq j}\left\langle \alpha_{m,\beta_{j}(t)}^{0}(\cdot-y_{j}(t)),\mathcal{V}_{i}(t)\hm{u}(t)\right\rangle \\
			& +\left\langle \alpha_{m,\beta_{j}(t)}^{0}(\cdot-y_{j}(t)),\mathcal{I}(Q)+\mathfrak{I}(Q,u)+\hm{F}(u)\right\rangle \\
			& +\left\langle \alpha_{m,\beta_{j}(t)}^{0}(\cdot-y_{j}(t)),\text{Mod}'(t)\nabla_{M}\hm{Q}(\hm{\beta}(t),\hm{y}(t))\right\rangle .
		\end{align*}
		where	we used notations
		\begin{equation}\label{eq:notationstwo}
			\mathcal{V}_j(t):=\left(\begin{array}{cc}
				0 & 0\\
				-V_{\beta_j(t)}(x-y_j(t)) & 0
			\end{array}\right),\,JH_{j}(t):=JH_0+\mathcal{V}_j(t).
		\end{equation}	
		Note that
		{\small\begin{align*}
				\left\langle \alpha_{m,\beta_{j}(t)}^{0}(\cdot-y_{j}(t)),\text{Mod}'(t)\nabla_{M}\hm{Q}(\hm{\beta}(t),\hm{y}(t))\right\rangle  & =-\left\langle \alpha_{m,\beta_{j}(t)}^{0}(\cdot-y_{j}(t)),\beta'_{j}(t)\partial_{\beta}\hm{Q}_{j}(\hm{\beta}_{j}(t), \hm{y}_{j}(t))\right\rangle \\
				& -\sum_{i\neq j}\left\langle \alpha_{m,\beta_{j}(t)}^{0}(\cdot-y_{j}(t)),\beta'_{i}(t)\partial_{\beta}\hm{Q}_{i}(\hm{\beta}_{i}(t), \hm{y}_{i}(t))\right\rangle \\
				& -\sum_{i\neq j}\left\langle \alpha_{m,\beta_{j}(t)}^{0}(\cdot-y_{j}(t)),\left(y'_{i}(t)-\beta_{i}(t)\right)\partial_{y}\hm{Q}_{i}(\hm{\beta}_{i}(t), \hm{y}_{i}(t))\right\rangle 
		\end{align*}}
		Next using the fact \eqref{eq:kerQ},  we compute
		\begin{align*}
			\left\langle \frac{d}{dt}\left(\alpha_{m,\beta_{j}(t)}^{0}(\cdot-y_{j}(t))\right),\hm{u}(t)\right\rangle +\left\langle \alpha_{m,\beta_{j}(t)}^{0}(\cdot-y_{j}(t)),JH_{j}(t)\hm{u}(t)\right\rangle \\
			=-\left\langle \left(y'_{j}(t)-\beta_{j}(t)\right)\cdot\nabla\alpha_{m,\beta_{j}(t)}^{0}(\cdot-y_{j}(t)),\hm{u}(t)\right\rangle +\left\langle \beta_{j}'(t)\partial_{\beta}\alpha_{m,\beta_{j}(t)}^{0}(\cdot-y_{j}(t)),\hm{u}(t)\right\rangle  & .
		\end{align*}
		Therefore, putting two results above together, we can conclude that
		\begin{align}
			\left\langle \alpha_{m,\beta_{j}(t)}^{0}(\cdot-y_{j}(t)),\beta'_{j}(t)\partial_{\beta}\hm{Q}_{j}(\hm{\beta}_{j}(t), \hm{y}_{j}(t))\right\rangle
			+\left\langle \beta_{j}'(t)\partial_{\beta}\alpha_{m,\beta_{j}(t)}^{0}(\cdot-y_{j}(t)),\hm{u}(t)\right\rangle  \label{eq:modeq1}\\
			+\left\langle \left(y'_{j}(t)-\beta_{j}(t)\right)\cdot\nabla\alpha_{m,\beta_{j}(t)}^{0}(\cdot-y_{j}(t)),\hm{u}(t)\right\rangle \nonumber \\
			=\left\langle \alpha_{m,\beta_{j}(t)}^{0}(\cdot-y_{j}(t)),\mathcal{I}(Q)+\mathfrak{I}(Q,u)+\hm{F}(u)\right\rangle\nonumber\\
			+\mathcal{O}\left(\Large(\sum_{i\neq j}|y'_{i}(t)-\beta_{i}(t)|+|\beta'_{i}(t)|\Large)e^{-\kappa/2}\right).\nonumber 
		\end{align}
		provided that the trajectories satisfy the separation condition. 
		
		Similar computations can be applied to the other orthogonality condition.
		Given the equation for $\alpha_{m,\beta_{j}(t)}^{1}(\cdot-y_{j}(t))$, \eqref{eq:gkerQ},
		and $\left\langle \alpha_{m,\beta_{j}(t)}^{0}(\cdot-y_{j}(t)),\hm{u}(t)\right\rangle =0$,
		one has
		\begin{align}
			\left\langle \alpha_{m,\beta_{j}(t)}^{1}(\cdot-y_{j}(t)),(y'_{j}(t)-\beta_{j}(t))\partial_{y}\hm{Q}_{j}(\hm{\beta}_{j}(t), \hm{y}_{j}(t))\right\rangle \nonumber \\
			+\left\langle \left(y'_{j}(t)-\beta_{j}(t)\right)\cdot\nabla\alpha_{m,\beta_{j}(t)}^{1}(\cdot-y_{j}(t)),\hm{u}(t)\right\rangle \nonumber \\
			+\left\langle \beta_{j}'(t)\partial_{\beta}\alpha_{m,\beta_{j}(t)}^{1}(\cdot-y_{j}(t)),\hm{u}(t)\right\rangle \label{eq:modeq2}\\
			=\left\langle \alpha_{m,\beta_{j}(t)}^{1}(\cdot-y_{j}(t)),\mathcal{I}(Q)+\mathfrak{I}(Q,u)+\hm{F}(u)\right\rangle \nonumber\\
			+\mathcal{O}\left(\Large(\sum_{i\neq j}|y'_{i}(t)-\beta_{i}(t)|+|\beta'_{i}|(t)\Large)e^{-\kappa/2}\right).\nonumber
		\end{align}
		Note that $\mathfrak{I}$ can be further split as
		\[
		\mathfrak{I}_1(Q^{2},u)=\left(\begin{array}{c}
			0\\
			3\sum_{i\neq k}Q_{\beta_{i}}(\cdot-y_{i}(t))Q_{\beta_{k}}(\cdot-y_{k}(t))u
		\end{array}\right)
		\]
		and
		\[
		\mathfrak{I}_2(Q,u^{2})=\left(\begin{array}{c}
			0\\
			3\sum_{i=1}^{J}Q_{\beta_{i}}(\cdot-y_{i}(t))u^{2}
		\end{array}\right).
		\]
		Therefore, given the separation conditions with $\kappa$ large,  and $\left\Vert u\right\Vert _{\mathcal{H}}\lesssim\eta$
		for $\eta$ sufficiently small, one has
		\begin{align}
			|\beta'_{j}(t)|\lesssim \sum_{m=1}^{3}\left|\left\langle \alpha_{m,\beta_{j}(t)}^{0}(\cdot-y_{j}(t)),\mathcal{I}(Q)+\mathfrak{I}(Q,u)\right\rangle \right|\label{eq:mod1}
		\end{align}
		
		%	\end{equation}
		and
		\begin{equation}
			\left|y'_{j}(t)-\beta_{j}(t)\right|\lesssim \sum_{m=1}^{3}\left|\left\langle \alpha_{m,\beta_{j}(t)}^{1}(\cdot-y_{j}(t)),\mathcal{I}(Q)+\mathfrak{I}(Q,u)\right\rangle \right|\label{eq:mod2}
		\end{equation}which imply \eqref{eq:mod1lem} and \eqref{eq:mod2lem}.
	\end{proof}
	
	%	\newpage

	\section{Stability of multi-solitons}\label{sec:manifold}
	
	In this section, using tools from the previous section, we will illustrate that if a solution stay close to the family of multi-solitons, then actually the solution scatters to the family of the multi-solitons, i.e.,  the  solution converging
	in large time to a superposition of Lorentz-transformed solitons (with slightly modified velocities),
	and a radiation term which is at main order a free wave. In particular, this says that the orbital stability of multi-solitons implies their asymptotic stability.

	\begin{thm}\label{thm:orbitasy3}
		For every $\delta > 0$ there exist $\rho, \eta > 0$ such that
		the following holds. Let the initial parameters $(y_j^{in}, \beta_j^{in})_{j=1}^J$
		satisfy the separation condition in the sense of Definition \ref{def:sep}, and let
		\begin{equation}
			\|\bs \psi_0 - \bs Q( \hm{\beta}^{in}, \hm{y}^{in})\|_{\mathcal{H}} \leq \eta.
		\end{equation}
		Suppose that the solution $\bs \psi$ to \eqref{eq:dynkg} with initial data $\hm{\psi}_0$  stays close in  a neighborhood of multi-soliton family:
		\begin{equation}
			\sup_{t\in\mathbb{R}^{+}}\inf_{\hm{\beta}\in\mathbb{R}^{3N},\hm{y}\in\mathbb{R}^{3N}}\left\Vert \boldsymbol{\psi}(t)-\boldsymbol{Q}(\hm{\beta},\hm{y})\right\Vert _{\mathcal{H}}\lesssim\eta,\label{eq:orbcond1}
		\end{equation}
		then $\bs \psi$ scatters to the multi-soliton family:
		there exist $\beta_{j}\in\mathbb{R}^{3}$,
		functions  $y_{j}(t)\in\mathbb{R}^{3}$ and $\boldsymbol{\psi}_{+}\in \mathcal{H}$ with
		the property that $\dot{y}_{j}(t)\rightarrow\beta_{j}$ and
		\[
		\lim_{t\rightarrow\infty}\text{\ensuremath{\left\Vert \boldsymbol{\psi}(t)-\boldsymbol{Q}(\hm{\beta},\hm{y}(t))-e^{JH_{0}t}\boldsymbol{\psi}_{+}\right\Vert _{\mathcal{H}}}}=0.
		\]
	\end{thm}
	\begin{proof}
		
		From the orbital stability \eqref{thm:orbitasy3}, using the modulation computations, the first part of Lemma \ref{lem:mod}, we can find Lipschitz functions  $\hm{y}(t)$ and $\hm{\beta}(t)$ such that
		\[
		\sup_{t\in\mathbb{R}^{+}}\left\Vert \boldsymbol{\psi}(t)-\boldsymbol{Q}(\hm{y}(t),\hm{\beta}(t))\right\Vert _{\mathcal{H}}\lesssim\eta
		\]
		and setting
		\begin{equation}\label{eq:decomppsi}
			\hm{u}(t):=\boldsymbol{\psi}(t)-\boldsymbol{Q}(\hm{y}(t),\hm{\beta}(t)),
		\end{equation}
		one has for $m=1,2,3$ and $j=1,\ldots,N$
		\begin{equation}\label{eq:orbortho}
			\left\langle \alpha_{m,\beta_{j}(t)}^{1}(\cdot-y_{j}(t)),\hm{u}(t)\right\rangle =\left\langle \alpha_{m,\beta_{j}(t)}^{0}(\cdot-y_{j}(t),\hm{u}(t)\right\rangle =0
		\end{equation}
		and $\sup_{t\in\mathbb{R}^{+}}\left\Vert \boldsymbol{u}(t)\right\Vert _{\mathcal{H}}\lesssim\eta$.
		Next, we will analyze more quantitative estimates for $\large(\hm{u}(t),\hm{\beta}(t),\hm{y}(t)\large).$

		From the decomposition \eqref{eq:decomppsi} and using the equation for $\bs \psi(t)$, we obtain
		\begin{align}\label{eq:uorbit}
			\frac{d}{dt}\hm{u}(t) & =JH(t)\hm{u}(t)+\mathcal{I}(Q)+\mathfrak{I}_{1}(Q^{2},u)+\mathfrak{I}_{2}(Q,u^{2})+\hm{F}(u)\\
			& -\dot{\hm{\beta}}(t)\partial_{\hm{\beta}}\hm{Q}(\hm{\beta}(t),\hm{y}(t))-\left(\dot{\hm{y}}(t)-\hm{\beta}(t)\right)\partial_{\hm{y}}\hm{Q}(\hm{\beta}(t),\hm{y}(t))\nonumber\\
			& =:JH(t)\hm{u}(t)+\mathcal{I}(Q)+\mathfrak{I}(Q,u)+\hm{F}(u)+\text{Mod}'(t)\nabla_{M}\hm{Q}(\hm{\beta}(t),\hm{y}(t))\nonumber\\
			& =:JH(t)\hm{u}(t)+\mathcal{W}(x,t).\nonumber
		\end{align}
		From the notations from Definition \ref{def:sep-1} associated to the
		path $(\hm{\beta}(t),\hm{y}(t))$,  one can decompose $\hm{u}(t)$  as
		\[
		\hm{u}(t)=\pi_{c}(t)\hm{u}(t)+\text{\ensuremath{\pi_{+}}}(t)\hm{u}(t)+\pi_-(t)\hm{u}(t)
		\] where we used the fact \eqref{eq:orbortho}, $\pi_0(t)\hm{u}(t)=0,\,\forall t.$
		
		Denote \begin{equation}
			a_{j}^{\pm}(t)=\left\langle \alpha_{\beta_{j}(t)}^{\pm}(\cdot-y_{j}(t),\hm{u}(t)\right\rangle.
		\end{equation}
		
		We first compute $a_{j}^{+}(t)$. Taking the time derivative of $a_{j}^{+}(t)$, following similar computations in the proof of Lemma \ref{lem:mod} for the time derivatives of zero modes, we get
		\begin{align*}
			\frac{d}{dt}a_{j}^{+}(t)& =	\frac{d}{dt}\left\langle \alpha_{\beta_{j}(t)}^{+}(\cdot-y_{j}(t)),\hm{u}(t)\right\rangle  \\
			& =\left\langle \frac{d}{dt}\left(\alpha_{\beta_{j}(t)}^{+}(\cdot-y_{j}(t))\right),\hm{u}(t)\right\rangle +\left\langle \alpha_{\beta_{j}(t)}^{+}(\cdot-y_{j}(t)),\frac{d}{dt}\hm{u}(t)\right\rangle .
		\end{align*}
		Using the equation for $\hm{u}(t)$, we expand
		{\small \begin{align*}
				\left\langle \alpha_{\beta_{j}(t)}^{+}(\cdot-y_{j}(t)),\frac{d}{dt}\hm{u}(t)\right\rangle & =\left\langle \alpha_{\beta_{j}(t)}^{+}(\cdot-y_{j}(t))(\cdot-y_{j}(t)),JH(t)\hm{u}(t)\right\rangle \\
				& +\left\langle \alpha_{\beta_{j}(t)}^{+}(\cdot-y_{j}(t))(\cdot-y_{j}(t)),\mathcal{I}(Q)+\mathfrak{I}(Q,u)+\hm{F}(u)\right\rangle \\
				& \left\langle \alpha_{\beta_{j}(t)}^{+}(\cdot-y_{j}(t))(\cdot-y_{j}(t)),\text{Mod}'(t)\nabla_{M}\hm{Q}(\hm{\beta}(t),\hm{y}(t))\right\rangle \\
				& =\left\langle \alpha_{\beta_{j}(t)}^{+}(\cdot-y_{j}(t))(\cdot-y_{j}(t)),JH_{j}(t)\hm{u}(t)\right\rangle \\
				& +\sum_{i\neq j}\left\langle \alpha_{\beta_{j}(t)}^{+}(\cdot-y_{j}(t)),\mathcal{V}_{i}(t)\hm{u}(t)\right\rangle \\
				& +\left\langle \alpha_{\beta_{j}(t)}^{+}(\cdot-y_{j}(t)),\mathcal{I}(Q)+\mathfrak{I}(Q,u)+\hm{F}(u)\right\rangle \\
				& +\left\langle\alpha_{\beta_{j}(t)}^{+}(\cdot-y_{j}(t)),\text{Mod}'(t)\nabla_{M}\hm{Q}(\hm{\beta}(t),\hm{y}(t))\right\rangle
		\end{align*}}
		where	we used notations from \eqref{eq:notationstwo}.
		
		Next we note that\begin{align*}
			\left\langle \frac{d}{dt}\left(\alpha_{\beta_{j}(t)}^{+}(\cdot-y_{j}(t))\right),\hm{u}(t)\right\rangle +\left\langle \alpha_{\beta_{j}(t)}^{+}(\cdot-y_{j}(t)),JH_{j}(t)\hm{u}(t)\right\rangle \\
			=-\left\langle \left(y'_{j}(t)-\beta_{j}(t)\right)\cdot\nabla\alpha_{\beta_{j}(t)}^{+}(\cdot-y_{j}(t)),\hm{u}(t)\right\rangle +\left\langle \beta_{j}'(t)\partial_{\beta}\alpha_{\beta_{j}(t)}^{+}(\cdot-y_{j}(t))\right\rangle \\
			+\frac{\nu}{\gamma_{j}(t)}\left\langle \alpha_{\beta_{j}(t)}^{+}(\cdot-y_{j}(t)),\hm{u}(t)\right\rangle.
		\end{align*}
		Finally integrating the resulting ODE for  $a_{j}^{+}(t)$, one has 
		\begin{align*}
			a_{j}^{+}(t) & =a_{j}^{+}(0)\exp\left(\int_{0}^{t}\frac{\nu}{\gamma_{j}(\tau)}\,d\tau\right)+\int_{0}^{t}\exp\left(\int_{s}^{t}\frac{\nu}{\gamma_{j}(\tau)}\,d\tau\right)\left\langle \alpha_{\beta_{j}(s)}^{+}(\cdot-y_{j}(s),\mathcal{W}(\cdot,s)\right\rangle \,ds\\
			& +\sum_{k\neq j}\int_{0}^{t}\exp\left(\int_{s}^{t}\frac{\nu}{\gamma_{j}(\tau)}\,d\tau\right)\left\langle \alpha_{\beta_{j}(s)}^{+}\big(\cdot-y_{j}(s)\big),\mathcal{V}_{\beta_{k}(s)}\big(\cdot-y_{k}(s)\big)\hm{u}(s)\right\rangle \,ds\\
			& \text{+\ensuremath{\int_{0}^{t}}}\exp\left(\int_{s}^{t}\frac{\nu}{\gamma_{j}(\tau)}\,d\tau\right)\left\langle \left(y'_{j}(s)-\beta_{j}(s)\right)\cdot\nabla\alpha_{\beta_{j}(s)}^{\text{+}}(\cdot-y_{j}(s)),\hm{u}(s)\right\rangle \,ds\\
			& \text{+\ensuremath{\int_{0}^{t}}}\exp\left(\int_{s}^{t}\frac{\nu}{\gamma_{j}(\tau)}\,d\tau\right)\left\langle \beta_{j}'(s)\partial_{\beta}\alpha_{\beta_{j}(s)}^{+}(\cdot-y_{j}(s)),\hm{u}(s)\right\rangle \,ds.
		\end{align*}	
		%	Moreover, one also knows that
		%	\[
		%	\left\Vert \text{\ensuremath{\pi_{u}}}(t)\hm{u}(t)\right\Vert _{\mathcal{H}}\sim\sum_{j=1}^{J}\left|\left\langle \alpha_{\beta_{j}(t)}^{+}(\cdot-y_{j}(t),\hm{u}(t)\right\rangle \right|=:\sum_{j=1}^{N}\left|a_{j}^{+}(t)\right|.
		%	\]		
		Since $\sup_{t\in\mathbb{R}^+}\left\Vert \hm{u}(t)\right\Vert _{\mathcal{H}}\lesssim\eta$, it gives the unique choice for the initial data for $a_j^{+}(0)$, see for example Beceanu \cite{Bec2}, Krieger-Schlag \cite{KS2}, Schlag \cite{Sch1}, and leads to the following equation:
		\begin{align}
			a_{j}^{+}(t) & =-\int_{t}^{\infty}\exp\left(\int_{s}^{t}\frac{\nu}{\gamma_{j}(\tau)}\,d\tau\right)\left\langle \alpha_{\beta_{j}(s)}^{+}(\cdot-y_{j}(s),\mathcal{W}(\cdot,s)\right\rangle \,ds\label{eq:aorbit}\\
			& -\sum_{k\neq j}\int_{t}^{\infty}\exp\left(\int_{s}^{t}\frac{\nu}{\gamma_{j}(\tau)}\,d\tau\right)\left\langle \alpha_{\beta_{j}(s)}^{+}\big(\cdot-y_{j}(s)\big),\mathcal{V}_{\beta_{k}(s)}\big(\cdot-y_{k}(s)\big)\hm{u}(s)\right\rangle \,ds\nonumber\\
			& -\int_{t}^{\infty}\exp\left(\int_{s}^{t}\frac{\nu}{\gamma_{j}(\tau)}\,d\tau\right)\left\langle \left(y'_{j}(s)-\beta_{j}(s)\right)\cdot\nabla\alpha_{\beta_{j}(s)}^{\text{+}}(\cdot-y_{j}(s)),\hm{u}(s)\right\rangle \,ds\nonumber\\
			& -\int_{t}^{\infty}\exp\left(\int_{s}^{t}\frac{\nu}{\gamma_{j}(\tau)}\,d\tau\right)\left\langle \beta_{j}'(s)\partial_{\beta}\alpha_{\beta_{j}(s)}^{+}(\cdot-y_{j}(s)),\hm{u}(s)\right\rangle \,ds.\nonumber
		\end{align}
		%Then from \eqref{eq:uorbit}, \eqref{eq:aorbit} and \eqref{eq:orbortho}, we now conclude $(\hm{\beta}(t),\hm{y}(t),\hm{u}(t))$ is actually the fixed point of the map \eqref{eq:inputoutput} constructed in the previous proof. Then $(\hm{\beta}(t),\hm{y}(t),\hm{u}(t))$ has the same behavior as 
		%Theorem \ref{thm:manifold}.
		%Similar computations for the one potential setting appear als Beceanu, Schlag Krieger-Schlag. 
		By similar computations, for the stable modes, we have
		{\small\begin{align}
				a_{j}^{-}(t) & =a_{j}^{-}(0)\exp\left(-\int_{0}^{t}\frac{\nu}{\gamma_{j}(\tau)}\,d\tau\right)-\int_{0}^{t}\exp\left(-\int_{s}^{t}\frac{\nu}{\gamma_{j}(\tau)}\,d\tau\right)\left\langle \alpha_{\beta_{j}}^{-}(\cdot-\beta_{j}s-x_{j}),\mathcal{W}(\cdot,s)\right\rangle \,ds\label{eq:stable}\\
				& +\sum_{k\neq j}\int_{0}^{t}\exp\left(-\int_{s}^{t}\frac{\nu}{\gamma_{j}(\tau)}\,d\tau\right)\left\langle \alpha_{\beta_{j}}^{-}(\cdot-\beta_{j}s-x_{j}),\mathcal{V}_{\beta_{k}}\big(\cdot-\beta_{k}s-x_{k}\big)\hm{u}(s)\right\rangle \,ds\nonumber \\
				& +\int_{0}^{t}\exp\left(-\int_{s}^{t}\frac{\nu}{\gamma_{j}(\tau)}\,d\tau\right)\left\langle \left(y'_{j}(s)-\beta_{j}(s)\right)\cdot\nabla\alpha_{\beta_{j}(s)}^{-}(\cdot-y_{j}(s)),\hm{u}(s)\right\rangle \,ds\nonumber \\
				& +\int_{0}^{t}\exp\left(-\int_{s}^{t}\frac{\nu}{\gamma_{j}(\tau)}\,d\tau\right)\left\langle \beta_{j}'(s)\partial_{\beta}\alpha_{\beta_{j}(s)}^{-}(\cdot-y_{j}(s)),\hm{u}(s)\right\rangle \,ds.\nonumber 
		\end{align}}
		We also recall  notations from Theorem \ref{thm:mainthmlinear}: $\mathcal{B}(t):=\sum_{j=1}^N	\left( \left | 	a_j^{+}(t) \right |+\left | 	a_j^{-}(t) \right|  \right)$.

		\noindent{\bf Bootstrap:}
		We next derive refined quantitative estimates for  $\large(\hm{u}(t),\hm{\beta}(t),\hm{y}(t)\large)$ via a bootstrap argument.
		We claim that assuming for some $T>0$, and some constant $C_T$, we have
		\begin{align}\label{eq:bootassume}
			\left\Vert \hm{u}\right\Vert _{\mathcal{S}_{\mathcal{H}} [0,T]}^2+ 	\left\Vert |\hm{y}'(\cdot)-\hm{\beta}(\cdot)|+|\hm{\beta}'(\cdot)|\right\Vert _{L_{t}^{\infty} \bigcap L^1_t [0,T]}& \leq C_T \eta^2
		\end{align}
		then  one can improve estimates above to
		\begin{align}\label{eq:bootrecover}
			\left\Vert \hm{u}\right\Vert _{\mathcal{S}_{\mathcal{H}} [0,T]}^2+ 	\left\Vert |\hm{y}'(\cdot)-\hm{\beta}(\cdot)|+|\hm{\beta}'(\cdot)|\right\Vert _{L_{t}^{\infty} \bigcap L^1_t [0,T]}& \leq C\eta^2 + \frac{1}{2} C_T \eta^2
		\end{align} for some prescribed constant $C$.
		Notice that the assumption \eqref{eq:bootassume} is valid since we already know that from the construction that $\sup_{t\in\mathbb{R}^+}\left\Vert \hm{u}(t)\right\Vert _{\mathcal{H}}\lesssim\eta$ and the estimates for $\hm{\beta}(t),\hm{y}(t)$ can be obtained via \eqref{eq:mod1lem} and \eqref{eq:mod2lem} in Lemma \ref{lem:mod}. Indeed, 	using modulation equations again, Lemma \ref{lem:mod}, we also have
		\begin{equation}
			\left\Vert |\hm{y}'(\cdot)-\hm{\beta}(\cdot)|+|\hm{\beta}'(\cdot)|\right\Vert _{L_{t}^{\infty} } \lesssim \eta^2.
		\end{equation}	
		Now we prove \eqref{eq:bootrecover}.
		First of all,  from the bootstrap assumption, due to the separation properties of trajectories, the norms for $(\hm{\beta}(t),\hm{y}(t))$  above, we note that
		\begin{align}
			|y_j(t)-y_i(t)|&\geq \left|\int_0^t( y'_j(s)-y'_i(s))\,ds\right|+\rho\\&\geq \rho+\left|\int_0^t\left((\beta_{j,0}^{in}-\beta_{i,0}^{in})+ (\beta'_j(s)-\beta_{j,0}^{in})-(\beta'_i(s)-\beta_{i,0}^{in})\right)\,ds\right|\nonumber
			\\
			&- \left|\int_0^t\left( (y'_j(s)-\beta_j(s))-(y'_i(s)-\beta_i(s)\right)\,ds\right|\nonumber\\
			&\geq \rho+ \left|\int_0^t (\delta -2\eta^2 )\,ds\right|-2\eta^2\ \geq \frac{\delta t+ \rho}{2}\nonumber
		\end{align}provided that we picked $\eta\ll \delta$.

		Given  orthogonality conditions \eqref{eq:orbortho} above, we can apply Strichartz estimates in Theorem \ref{thm:mainthmlinear} to estimate $\pi_{c}\hm{u}(t)$ and  conclude that
		{\small	\begin{align}
				\left\Vert  \hm{u }\right\Vert _{\mathcal{S}_{\mathcal{H}} [0,T]} & \lesssim \eta +\left \Vert\mathcal{I}(Q)+\mathfrak{I}(Q,u)+\hm{F}(u)+\text{Mod}'(t)\nabla_{M}\hm{Q}(\hm{\beta}(t),\hm{y}(t))\right\Vert_{\mathcal{S}^*_{\mathcal{H}} [0,T]}+\left\Vert 	\mathcal{B} \right\Vert_{L^1_t\bigcap L^\infty_t[0,T]}\nonumber\\
				&\lesssim \eta + \frac{1}{\delta} e^{-\frac{\rho}{2}}+ \frac{1}{\delta} e^{-\frac{\rho}{2}}C_T\eta + (C_T\eta )^2+(C_T\eta )^3+\left\Vert 	\mathcal{B} \right\Vert_{L^1_t\bigcap L^\infty_t[0,T]}\label{eq:recover1}.
		\end{align}}
		Next, we estimate $\mathcal{B}(t)$.   We start with $a_j^-(t)$.  From the explicit formula \eqref{eq:stable},  due to the exponential decay of $\exp\left(\int_{s}^{t}-\frac{\nu}{\gamma_{j}(\tau)}\,d\tau\right)$
		by Young's inequality,  it suffices to bound the $L_{t}^{2}[0,T]$ and
		$L_{t}^{\infty}[0,T]$ norms of 
		\begin{align*}
			\left\langle \alpha_{\beta_{j}(s)}^{-}(\cdot-y_{j}(s)),\mathcal{W}(\cdot,s)\right\rangle +\sum_{k\neq j}\left\langle \alpha_{\beta_{j}(s)}^{-}\big(\cdot-y_{j}(s)\big),\mathcal{V}_{\beta_{k}(s)}\big(\cdot-y_{k}(s)\big)\hm{u}(s)\right\rangle \\
			\left\langle \left(y'_{j}(s)-\beta_{j}(s)\right)\cdot\nabla\alpha_{\beta_{j}(s)}^{\text{-}}(\cdot-y_{j}(s)),\hm{u}(s)\right\rangle +\left\langle \beta_{j}'(s)\partial_{\beta}\alpha_{\beta_{j}(s)}^{-}(\cdot-y_{j}(s)),\hm{u}(s)\right\rangle .
		\end{align*}
		Clearly, due to the separation of trajectories and the exponential
		decay of eigenfunctions, one has for $k\neq j$
		\[
		\left|\left\langle \alpha_{\beta_{j}(s)}^{-}(\cdot-y_{j}(s)),\mathcal{V}_{\beta_{k}(s)}(\cdot-y_{k}(s))\right\rangle \right|\lesssim e^{-(\frac{\delta s+ \rho}{2})}.
		\]
		Since eigenfunctions are smooth, we also have
		\begin{align*}
			\left|\left\langle \left(y'_{j}(s)-\beta_{j}(s)\right)\cdot\nabla\alpha_{\beta_{j}(s)}^{-}(\cdot-y_{j}(s)),\hm{u}(s)\right\rangle \right|\\
			+\left|\left\langle \beta_{j}'(s)\partial_{\beta}\alpha_{\beta_{j}(s)}^{-}(\cdot-y_{j}(s)),\hm{u}(s)\right\rangle \right|\\
			\lesssim\left(|\hm{y}'(s)-\hm{\beta}(s)|+|\hm{\beta}'(s)|\right)\left\Vert \hm{u}(s)\right\Vert _{\mathcal{H}}.
		\end{align*}
		One also can compute that
		\begin{align*}
			\left|\left\langle \alpha_{\beta_{j}(s)}^{-}(\cdot-y_{j}(s),\mathcal{W}(\cdot,s)\right\rangle \right| & \lesssim e^{-(\frac{\delta s+ \rho}{2})}+|\hm{y}'(s)-\hm{\beta}(s)|+|\hm{\beta}'(s)|\\
			& +\text{\ensuremath{\left(\left\Vert \hm{u}(s)\right\Vert _{\mathcal{H}}^{2}+\left\Vert \hm{u}(s)\right\Vert _{\mathcal{H}}\right)}}e^{-(\frac{\delta s+ \rho}{2})}\\
			&+	\left|\left\langle \alpha_{\beta_{j}(s)}^{-}(\cdot-y_{j}(s),\hm{F}(\hm{u)}\right\rangle \right|
		\end{align*}
		Therefore, we can conclude that given the separation condition, one
		has 
		\begin{align}
			\sum_{j=1}^N \left\Vert a_{j}^-\right\Vert_{L^\infty_t[0,T]\bigcap L^1_t[0,T]}& \lesssim\eta+ \frac{1}{\delta} e^{-(\frac{ \rho}{2})}+ \left\Vert |\hm{y}'(s)-\hm{\beta}(s)|+|\hm{\beta}'(s)|\right\Vert _{L_{t}^{\infty}\bigcap L^1_t[0,T]}\label{eq:aj-}\\
			& +\frac{1}{\delta} e^{-(\frac{ \rho}{2})}\sup_{s\in[0,T]}\left(\left\Vert \hm{u}(s)\right\Vert _{\mathcal{H}}^{2}+\left\Vert \hm{u}(s)\right\Vert _{\mathcal{H}}\right)+\left\Vert \hm{u}\right\Vert _{\mathcal{S}_{\mathcal{H}}[0,T]}^3.\nonumber\\
			& \lesssim  \eta+\frac{1}{\delta} e^{-\frac{\rho}{2}}+ \frac{1}{\delta} e^{-\frac{\rho}{2}}C_T\eta + (C_T\eta )^2+(C_T\eta )^3.\nonumber
		\end{align}
		\iffalse	
		Taking $\eta$ small and $\rho$ large, we recover \eqref{eq:bootrecover} for 	$\pi_{cs} (\cdot) \hm{u }$. 
		
		Next we estimate $\pi_{u} (\cdot) \hm{u }$. Since 	\[
		\left\Vert \text{\ensuremath{\pi_{u}}}(t)\hm{u}(t)\right\Vert _{\mathcal{H}}\sim\sum_{j=1}^{J}\left|\left\langle \alpha_{\beta_{j}(t)}^{+}(\cdot-y_{j}(t),\hm{u}(t)\right\rangle \right|=:\sum_{j=1}^{N}\left|a_{j}^{+}(t)\right|,
		\] 
		\fi
		Next to estimate $a_j^+(t)$, we split
		\begin{equation}
			a_j^+(t)=a_{j,0}^+(t)+a_{j,1}^+(t)
		\end{equation}
		where
		\begin{align}
			a_{j,0}^{+}(t) & =-\int_{t}^{T}\exp\left(\int_{s}^{t}\frac{\nu}{\gamma_{j}(\tau)}\,d\tau\right)\left\langle \alpha_{\beta_{j}(s)}^{+}(\cdot-y_{j}(s),\mathcal{W}(\cdot,s)\right\rangle \,ds\nonumber\\
			& -\sum_{k\neq j}\int_{t}^{T}\exp\left(\int_{s}^{t}\frac{\nu}{\gamma_{j}(\tau)}\,d\tau\right)\left\langle \alpha_{\beta_{j}(s)}^{+}\big(\cdot-y_{j}(s)\big),\mathcal{V}_{\beta_{k}(s)}\big(\cdot-y_{k}(s)\big)\hm{u}(s)\right\rangle \,ds\nonumber\\
			& -\int_{t}^{T}\exp\left(\int_{s}^{t}\frac{\nu}{\gamma_{j}(\tau)}\,d\tau\right)\left\langle \left(y'_{j}(s)-\beta_{j}(s)\right)\cdot\nabla\alpha_{\beta_{j}(s)}^{\text{+}}(\cdot-y_{j}(s)),\hm{u}(s)\right\rangle \,ds\nonumber\\
			& -\int_{t}^{T}\exp\left(\int_{s}^{t}\frac{\nu}{\gamma_{j}(\tau)}\,d\tau\right)\left\langle \beta_{j}'(s)\partial_{\beta}\alpha_{\beta_{j}(s)}^{+}(\cdot-y_{j}(s)),\hm{u}(s)\right\rangle \,ds\nonumber
		\end{align}
		and
		\begin{align}
			a_{j,1}^{+}(t) & =-\int_{T}^{\infty}\exp\left(\int_{s}^{t}\frac{\nu}{\gamma_{j}(\tau)}\,d\tau\right)\left\langle \alpha_{\beta_{j}(s)}^{+}(\cdot-y_{j}(s),\mathcal{W}(\cdot,s)\right\rangle \,ds\nonumber\\
			& -\sum_{k\neq j}\int_{T}^{\infty}\exp\left(\int_{s}^{t}\frac{\nu}{\gamma_{j}(\tau)}\,d\tau\right)\left\langle \alpha_{\beta_{j}(s)}^{+}\big(\cdot-y_{j}(s)\big),\mathcal{V}_{\beta_{k}(s)}\big(\cdot-y_{k}(s)\big)\hm{u}(s)\right\rangle \,ds\nonumber\\
			& -\int_{T}^{\infty}\exp\left(\int_{s}^{t}\frac{\nu}{\gamma_{j}(\tau)}\,d\tau\right)\left\langle \left(y'_{j}(s)-\beta_{j}(s)\right)\cdot\nabla\alpha_{\beta_{j}(s)}^{\text{+}}(\cdot-y_{j}(s)),\hm{u}(s)\right\rangle \,ds\nonumber\\
			& -\int_{T}^{\infty}\exp\left(\int_{s}^{t}\frac{\nu}{\gamma_{j}(\tau)}\,d\tau\right)\left\langle \beta_{j}'(s)\partial_{\beta}\alpha_{\beta_{j}(s)}^{+}(\cdot-y_{j}(s)),\hm{u}(s)\right\rangle \,ds.\nonumber
		\end{align}
		For $a_{j,0}^+(t)$, we can use the bootstrap assumption and Young's inequality in a similar manner to the analysis of  $a_j^-$. Due to the exponential decay of $\exp\left(\int_{s}^{t}\frac{\nu}{\gamma_{j}(\tau)}\,d\tau\right)$
		by Young's inequality,  it suffices to bound the $L_{t}^{2}[0,T]$ and
		$L_{t}^{\infty}[0,T]$ norms of 
		\begin{align*}
			\left\langle \alpha_{\beta_{j}(s)}^{+}(\cdot-y_{j}(s)),\mathcal{W}(\cdot,s)\right\rangle +\sum_{k\neq j}\left\langle \alpha_{\beta_{j}(s)}^{+}\big(\cdot-y_{j}(s)\big),\mathcal{V}_{\beta_{k}(s)}\big(\cdot-y_{k}(s)\big)\hm{u}(s)\right\rangle \\
			\left\langle \left(y'_{j}(s)-\beta_{j}(s)\right)\cdot\nabla\alpha_{\beta_{j}(s)}^{\text{+}}(\cdot-y_{j}(s)),\hm{u}(s)\right\rangle +\left\langle \beta_{j}'(s)\partial_{\beta}\alpha_{\beta_{j}(s)}^{+}(\cdot-y_{j}(s)),\hm{u}(s)\right\rangle .
		\end{align*}
		Due to the separation of trajectories and the exponential
		decay of eigenfunctions, one has for $k\neq j$
		\[
		\left|\left\langle \alpha_{\beta_{j}(s)}^{+}(\cdot-y_{j}(s)),\mathcal{V}_{\beta_{k}(s)}(\cdot-y_{k}(s))\right\rangle \right|\lesssim e^{-(\frac{\delta s+ \rho}{2})}.
		\]
		Since eigenfunctions are smooth, we also have
		\begin{align*}
			\left|\left\langle \left(y'_{j}(s)-\beta_{j}(s)\right)\cdot\nabla\alpha_{\beta_{j}(s)}^{\text{+}}(\cdot-y_{j}(s)),\hm{u}(s)\right\rangle \right|\\
			+\left|\left\langle \beta_{j}'(s)\partial_{\beta}\alpha_{\beta_{j}(s)}^{+}(\cdot-y_{j}(s)),\hm{u}(s)\right\rangle \right|\\
			\lesssim\left(|\hm{y}'(s)-\hm{\beta}(s)|+|\hm{\beta}'(s)|\right)\left\Vert \hm{u}(s)\right\Vert _{\mathcal{H}}.
		\end{align*}
		Finally, one can compute that
		\begin{align*}
			\left|\left\langle \alpha_{\beta_{j}(s)}^{+}(\cdot-y_{j}(s),\mathcal{W}(\cdot,s)\right\rangle \right| & \lesssim e^{-(\frac{\delta s+ \rho}{2})}+|\hm{y}'(s)-\hm{\beta}(s)|+|\hm{\beta}'(s)|\\
			& +\text{\ensuremath{\left(\left\Vert \hm{u}(s)\right\Vert _{\mathcal{H}}^{2}+\left\Vert \hm{u}(s)\right\Vert _{\mathcal{H}}\right)}}e^{-(\frac{\delta s+ \rho}{2})}\\
			&+	\left|\left\langle \alpha_{\beta_{j}(s)}^{+}(\cdot-y_{j}(s),\hm{F}(\hm{u)}\right\rangle \right|
		\end{align*}
		Therefore, we can conclude that given the separation condition, one
		has 
		\begin{align}
			\sum_{j=1}^N \left\Vert a_{j,0}^+\right\Vert_{L^\infty_t[0,T]\bigcap L^1_t[0,T]}& \lesssim\frac{1}{\delta} e^{-(\frac{ \rho}{2})}+ \left\Vert |\hm{y}'(s)-\hm{\beta}(s)|+|\hm{\beta}'(s)|\right\Vert _{L_{t}^{\infty}\bigcap L^1_t[0,T]}\label{eq:aj0}\\
			& +\frac{1}{\delta} e^{-(\frac{ \rho}{2})}\sup_{s\in[0,T]}\left(\left\Vert \hm{u}(s)\right\Vert _{\mathcal{H}}^{2}+\left\Vert \hm{u}(s)\right\Vert _{\mathcal{H}}\right)+\left\Vert \hm{u}\right\Vert _{\mathcal{S}_{\mathcal{H}}[0,T]}^3.\nonumber\\
			& \lesssim  \frac{1}{\delta} e^{-\frac{\rho}{2}}+ \frac{1}{\delta} e^{-\frac{\rho}{2}}C_T\eta + (C_T\eta )^2+(C_T\eta )^3.
		\end{align}
		For $a_{j,1}^+(t)$, we first note that by direct computations
		\begin{equation}\label{eq:aj1}
			\left\Vert \int_{T}^{\infty}\exp\left(\int_{s}^{t}\frac{\nu}{\gamma_{j}(\tau)}\,d\tau\right)\,ds\right\Vert_{L^\infty_t[0,T]\bigcap L^1_t[0,T]} \lesssim 1
		\end{equation} and then using $\sup_{t\in\mathbb{R}^+}\left\Vert \hm{u}(t)\right\Vert _{\mathcal{H}}\lesssim\eta$ and $
		\left\Vert |\hm{y}'(\cdot)-\hm{\beta}(\cdot)|+|\hm{\beta}'(\cdot)|\right\Vert _{L_{t}^{\infty} } \lesssim \eta^2$, one has 
		\begin{equation}
			\sum_{j=1}^N \left\Vert a_{j,1}^+\right\Vert_{L^\infty_t[0,T]\bigcap L^1_t[0,T]}\lesssim  \frac{1}{\delta} e^{-\frac{\rho}{2}}+ \frac{1}{\delta} e^{-\frac{\rho}{2}} \eta + \eta ^2 +\eta^3.
		\end{equation} 
		Therefore, taking $\eta$ large,  
		putting \eqref{eq:recover1}, \eqref{eq:aj-}, \eqref{eq:aj0} and \eqref{eq:aj1} together, we recast the bootstrap assumption for $\hm{u}(t)$.   
		
		Finally, for modulation parameters, 
		we differentiate the orthogonality conditions \eqref{eq:orbortho}. From the second part of Lemma \ref{lem:mod}, we have
		\begin{equation}
			|\dot{\beta}_{j}(t)|\lesssim \left(\left\Vert \hm{u}\right\Vert _{\mathcal{H}}+1\right)e^{-(\frac{\delta t+ \rho}{2})}+\sum_{m=1}^{3}\left|\left\langle \alpha_{m,\beta_{j}(t)}^{0}(\cdot-y_{j}(t)),\hm {q}(t)\right\rangle \right|
		\end{equation}
		and
		\begin{equation}
			\left|y'_{j}(t)-\beta_{j}(t)\right|\lesssim\left(\left\Vert \hm{u}\right\Vert _{\mathcal{H}}+1\right)e^{-(\frac{\delta t+ \rho}{2})}+\sum_{m=1}^{3}\left|\left\langle \alpha_{m,\beta_{j}(t)}^{1}(\cdot-y_{j}(t)),\hm {q}(t)\right\rangle \right|
		\end{equation}
		where $\hm {q}(t)=\Large(0,3\sum_{i=1}^{N}Q_{\beta_{i}}(\cdot-y_{i}(t))u^{2}\Large)^T. $
		In the computations above, we also used for $\iota=0,1 $
		\begin{equation}
			\left|\left\langle \alpha_{m,\beta_{j}(t)}^{\iota}(\cdot-y_{j}(t)),\mathcal{V}_{i}(t)\hm{u}(t)\right\rangle \right|\lesssim \exp\left\{-\frac{|y_j(t)-y_i(t)|}{2}\right\}\left\Vert \hm{u}\right\Vert _{\mathcal{H}}
		\end{equation} due to the exponential decay of eigenfunctions and potentials. 	 We also used similar estimates for
		\begin{equation}
			\left|\left\langle \alpha_{m,\beta_{j}(t)}^{\iota}(\cdot-y_{j}(t)),\mathfrak{I}_1(Q,u)\right\rangle \right|\lesssim e^{-(\frac{\delta t+ \rho}{2})}\left\Vert \hm{u}\right\Vert _{\mathcal{H}}
		\end{equation}
		and 
		\begin{equation}
			\left|\left\langle \alpha_{m,\beta_{j}(t)}^{\iota}(\cdot-y_{j}(t)),	\mathcal{I}(Q)\right\rangle \right|\lesssim e^{-(\frac{\delta t+ \rho}{2})}
		\end{equation}
		where we recall that
		\begin{equation}
			\mathcal{I}(Q)=\left(\begin{array}{c}
				0\\
				\ensuremath{\left(\sum_{j=1}^{N}\sigma_{j}Q_{\beta_{j}(t)}(\cdot-y_{j}(t))\right)^{3}-\sum_{j=1}^{N}\left(\sigma_{j}Q_{\beta_{j}(t)}(\cdot-y_{j}(t))\right)^{3}}
			\end{array}\right)
		\end{equation}
		\begin{equation}
			\mathfrak{I}_1(Q,u)=\left(\begin{array}{c}
				0\\
				3\sum_{i\neq k}\sigma_i\sigma_k Q_{\beta_{i}}(\cdot-y_{i}(t))Q_{\beta_{k}}(\cdot-y_{k}(t))u
			\end{array}\right).
		\end{equation}
		Next, by direct estimates and integration in $t$, we conclude that
		\begin{equation}
			\left\Vert \hm{\beta}'\right\Vert _{L_{t}^{1}\bigcap L_{t}^{\infty}[0,T]}+\left\Vert \hm{y}'-\hm{\beta}\right\Vert _{L_{t}^{1}\bigcap L_{t}^{\infty}[0,T]}\lesssim \frac{1}{\delta} e^{-(\frac{ \rho}{2})}\left\Vert \hm{u}\right\Vert _{\mathcal{S}_{\mathcal{H}}[0,T]}+\frac{1}{\delta} e^{-(\frac{ \rho}{2})}+\left\Vert \hm{u}\right\Vert _{\mathcal{S}_{\mathcal{H}}[0,T]}^2.\label{eq:tildemod}
		\end{equation}	
		Taking $\rho$ large, since $\left\Vert \hm{u}\right\Vert _{\mathcal{S}_{\mathcal{H}}[0,T]}$ recovers the bootstrap assumptions, whence, modulation parameters also recover the bootstrap assumptions. 
		
		After recasting the bootstrap assumptions for all pieces, we then conclude that \eqref{eq:bootassume} holds with some constant independent of $T$. So we can pass $T$ to $\infty$ and obtain
		\begin{equation}\label{eq:conclude}
			\sqrt{\left\Vert \hm{\beta}'\right\Vert _{L_{t}^{1}\bigcap L_{t}^{\infty}}}+\sqrt{\left\Vert \hm{y}'-\hm{\beta}\right\Vert _{L_{t}^{1}\bigcap L_{t}^{\infty}}}+\left\Vert \hm{u}\right\Vert _{\mathcal{S}_{\mathcal{H}}}\lesssim\eta.
		\end{equation}
		
		%The asymptotic behaviors of $\hm{u}(t)$ and modulation parameters follows in the same manner.
		%as the last part of the proof of Theorem \ref{thm:manifold}.
		
		\noindent{\bf Scattering}: 	
		From \eqref{eq:conclude}
		it follows that there exist $\beta_{j}\in\mathbb{R}^{3}$, 
		with the property that $\dot{y}_{j}(t)\rightarrow\beta_{j}$  $\beta_j(t)\rightarrow\beta_j. $

		The scattering of $\hm{u}$ is achieved by the standard argument via
		the Yang-Feldman equation. We expand $\hm{u}(t)$ using Duhamel formula
		with respect to the free evolution:
		\begin{align*}
			\hm{u}(t) & =e^{JH_{0}t}\hm{u}(0)+\int_{0}^{t}e^{JH_{0}(t-s)}\mathcal{V}(s)\hm{u}(s)\,ds\\
			& +\int_{0}^{t}e^{JH_{0}(t-s)}\left(\mathcal{I}(Q)+\mathfrak{I}(Q,u(s))+\hm{F}(u(s))\right)\,ds\\
			& +\int_{0}^{t}e^{JH_{0}(t-s)}\text{Mod}'(s)\nabla_{M}\hm{Q}(\hm{y}(s),\hm{\beta}(s))\,ds.
		\end{align*}
		Then we set $\hm{\psi}_{+}$ as
		\begin{align*}
			\hm{\psi}_{+} & :=\hm{u}(0)+\int_{0}^{\infty}e^{-sJH_{0}}\mathcal{V}(s)\hm{u}(s)\,ds\\
			& +\int_{0}^{\infty}e^{-sJH_{0}}\left(\mathcal{I}(Q)+\mathfrak{I}(Q,u(s))+\hm{F}(u(s))\right)\,ds\\
			& +\int_{0}^{\infty}e^{-sJH_{0}}\text{Mod}'(s)\nabla_{M}\hm{Q}(\hm{y}(s),\hm{\beta}(s))\,ds.
		\end{align*}
		Next we show $\hm{\psi}_{+}$ in $\mathcal{H}$. We first apply the
		endpoint Strichartz estimates for the free evolution, see \cite{BC,DaF,IMN},
		\begin{align*}
			\left\Vert \int_{0}^{\infty}e^{-sJH_{0}}\mathcal{V}(s)\hm{u}(s)\,ds\right\Vert _{\mathcal{H}} & \lesssim\left\Vert \mathcal{V}(s)\hm{u}(s)\right\Vert _{S_{2}^{*}}\lesssim\left\Vert \hm{u}(s)\right\Vert _{S_{2}}  \lesssim\left\Vert \hm{u}\right\Vert _{\mathcal{S}_{\mathcal{H}}}
		\end{align*}
		where $\mathcal{S}_{2}$ and $\mathcal{S}_{2}^{*}$ denote the endpoint
		Strichartz estimates. Using Strichartz estimates again, we get 
		\begin{align*}
			\left\Vert \int_{0}^{\infty}e^{-sJH_{0}}\left(\mathcal{I}(Q)+\mathfrak{I}(Q,u(s))+\hm{F}(u(s))\right)\,ds\right\Vert _{\mathcal{H}} & \lesssim\left\Vert \mathcal{I}(Q)+\mathfrak{I}(Q,u(s))+\hm{F}(u(s))\right\Vert _{L^{1}\mathcal{H}}\\
			& \lesssim\frac{1}{\delta} e^{-(\frac{ \rho}{2})}+\left\Vert \hm{u}\right\Vert _{\mathcal{S}_{\mathcal{H}}}^{2}+\left\Vert \hm{u}\right\Vert _{\mathcal{S}_{\mathcal{H}}}^{3}
		\end{align*}
		and
		\begin{align*}
			\left\Vert \int_{0}^{\infty}e^{-sJH_{0}}\left(\mathcal{I}(Q)+\mathfrak{I}(Q,u(s))+\hm{F}(u(s))\right)\,ds\right\Vert _{\mathcal{H}} & \lesssim\left\Vert \text{Mod}'(s)\nabla_{M}\hm{Q}(\hm{y}(s),\hm{\beta}(s))\right\Vert _{L^{1}\mathcal{H}}\\
			& \lesssim\left\Vert \hm{y}'(\cdot)-\hm{\beta}(\cdot)\right\Vert _{L_{t}^{1}}+\left\Vert \hm{\beta}'(\cdot)\right\Vert _{L_{t}^{1}}.
		\end{align*}
		Therefore, one has
		\begin{align*}
			\left\Vert \hm{\psi}_{+}\right\Vert _{\mathcal{H}} & \lesssim \frac{1}{\delta} e^{-(\frac{ \rho}{2})}+\left\Vert \hm{u}(0)\right\Vert _{\mathcal{H}}+\left\Vert \hm{y}'(\cdot)-\hm{\beta}(\cdot)\right\Vert _{L_{t}^{1}}+\left\Vert \hm{\beta}'(\cdot)\right\Vert _{L_{t}^{1}}\\
			& +\left\Vert \hm{u}\right\Vert _{\mathcal{S}_{\mathcal{H}}}+\left\Vert \hm{u}\right\Vert _{\mathcal{S}_{\mathcal{H}}}^{2}+\left\Vert \hm{u}\right\Vert _{\mathcal{S}_{\mathcal{H}}}^{3}\lesssim\eta
		\end{align*}
		and moreover, by construction, it follows
		\[
		\lim_{t\rightarrow\infty}\text{\ensuremath{\left\Vert \text{\ensuremath{\hm{u}(t)}}-e^{JH_{0}t}\boldsymbol{\psi}_{+}\right\Vert _{\mathcal{H}}}}=0
		\]
		and returning to the solution to the original equation $\hm{\psi}(t)$,
		one has
		\[
		\lim_{t\rightarrow\infty}\text{\ensuremath{\left\Vert \boldsymbol{\psi}(t)-\boldsymbol{Q}(y(t),\beta(t))-e^{JH_{0}t}\boldsymbol{\psi}_{+}\right\Vert _{\mathcal{H}}}}=0.
		\]
		Finally using the limiting properties of Lorentz parameters, we conclude	
		\[
		\lim_{t\rightarrow\infty}\text{\ensuremath{\left\Vert \boldsymbol{\psi}(t)-\boldsymbol{Q}(y(t),\beta)-e^{JH_{0}t}\boldsymbol{\psi}_{+}\right\Vert _{\mathcal{H}}}}=0
		\]
		which gives the desired result.

	\end{proof}

	\section{Construction of centre-stable manifolds}\label{sec:CSmanifold}
	
	In this section, using tools from previous section, we will construct a centre-stable around well-separated multi-soliton family.  %Then we will illustrate that if a solution stay close to the family of multi-solitons, then actually the solution scatters to the family of the multi-solitons, i.e.,  the  solution converging
	%	in large time to a superposition of Lorentz-transformed solitons (with slightly modified velocities),
	%	and a radiation term which is at main order a free wave. In particular, this says that the orbital stability of multi-solitons implies their asymptotic stability.
	
	\subsection{Construction of centre-stable manifolds}
	We first construct the centre-stable manifold in a small neighborhood of  one single  well-separated multi-soliton  in the energy space.

	Given a multi-soliton $\hm{Q}(\hm{\beta},\hm{y})$, we define the linear stable subspace associated to it as
	{\small\begin{equation}\label{eq:stablelinear}
			\mathcal{N}_{L}\left(\hm{Q}(\hm{\beta},\hm{y}),\eta\right):=\left\{ \begin{array}{c}
				\hm{R}\in\mathcal{H}|\left\langle \alpha_{j,m,\beta_{j}}^{1}(\cdot-y_{j}),\hm{R}\right\rangle =\left\langle \alpha_{j,m,\beta_{j}}^{0}(\cdot-y_{j}),\hm{R}\right\rangle =\left\langle \alpha_{j,\beta}^{+}(\cdot-y_{j}),\hm{R}\right\rangle =0\\
				j=1,\ldots,N,\,m=1,2,3,\,\left\Vert \hm{R}\right\Vert _{\mathcal{H}}<\eta
			\end{array}\right\} 
	\end{equation}} which is of codimension $7N$.
	
	\begin{thm}\label{thm:manifold1}
		
		For every $\delta>0$, there exist $\rho>0$ and $\eta>0$ such that
		the following holds. Let $\big(\hm{\beta}_0,\hm{y}_0\big)$ satisfy the
		separation condition  in the sense of Definition \ref{def:sep} with parameters $\delta$ and $\rho$. 
		There exists a map $\hm{\Phi}$ such that
		\[
		\hm{\Phi}:\mathcal{N}_{L}\left(\hm{Q}(\hm{\beta}_{0},\hm{y}_{0}),\eta\right)\rightarrow\mathbb{R}^{N}
		\]
		\begin{equation}\label{eq:boundphi}
			\left|\hm{\Phi}(\hm{R}_{0})\right|\lesssim\frac{1}{\delta}e^{-\rho}+\frac{1}{\delta}e^{-\rho}\left\Vert \hm{R}_{0}\right\Vert _{\mathcal{H}}+\left\Vert \hm{R}_{0}\right\Vert _{\mathcal{H}}^{2},\,\hm{R}_{0}\in\mathcal{N}_{L}\left(\hm{Q}(\hm{\beta}_{0},\hm{y}_{0})\right)
		\end{equation}
		\begin{equation}\label{eq:diffphi}
			\left|\hm{\Phi}(\hm{R}_{0})-\hm{\Phi}(\tilde{\hm{R}}_{0})\right|\lesssim\frac{1}{\delta}e^{-\rho}\left\Vert \hm{R}_{0}-\tilde{\hm{R}}_{0}\right\Vert _{\mathcal{H}}+\eta\left\Vert \hm{R}_{0}-\tilde{\hm{R}}_{0}\right\Vert _{\mathcal{H}},\,\hm{R}_{0},\tilde{\hm{R}}_{0}\in\mathcal{N}_{L}\left(\hm{Q}(\hm{\beta}_{0},\hm{y}_{0}),\eta\right)
		\end{equation}
		and so that $\forall\hm{R}_{0}\in\mathcal{N}_{L}\left(\hm{Q}(\hm{\beta}_{0},\hm{y}_{0}),\eta\right)$,
		the solution $\hm{\psi}(t)$ to \eqref{eq:dynkg} with initial data
		\[
		\hm{\psi}(0)=\hm{Q}\left(\hm{\beta}_0,\hm{y}_0\right)+\hm{R}_{0}+\hm{\Phi}(\hm{R}_{0})\cdot\hm{\mathcal{Y}^{+}}(\hm{\beta}_0,\hm{y}_0)
		\]
		where $\hm{\Phi}(\hm{R}_{0})\cdot\hm{\mathcal{Y}^{+}}(\hm{\beta}_0,\hm{y}_0)=\sum_{j=1}^N \phi_j(\hm{R_0})\mathcal{Y}^+_{\beta_{0,j}}(\cdot-y_{0,j})$
		exists globally, and it scatters to the multi-soliton family: there
		exist $\beta_{j}\in\mathbb{R}^{3}$, paths $y_{j}(t)\in\mathbb{R}^{3}$
		and $\boldsymbol{\psi}_{+}$ with the property that $\dot{y}_{j}(t)\rightarrow\beta_{j}$
		and
		\[
		\lim_{t\rightarrow\infty}\text{\ensuremath{\left\Vert \boldsymbol{\psi}(t)-\boldsymbol{Q}(\hm{\beta},\hm{y}(t))-e^{JH_{0}t}\boldsymbol{\psi}_{+}\right\Vert _{\mathcal{H}}}}=0.
		\]
		
		%
		%		
		%		Moreover, for $\boldsymbol{\psi}_{0}\in\mathcal{M}$ such that
		%		\[
		%		\left\Vert \boldsymbol{\psi}_{0}-\boldsymbol{Q}(y^{in},\beta^{in})\right\Vert %_{\mathcal{H}}\leq\eta
		%		\]
		%		and let $\hm{\psi}$ be the corresponding solution to \eqref{eq:dynkg}, then there exist
		%		$\beta_{j}\in\mathbb{R}^{3}$, paths $y_{j}(t)\in\mathbb{R}^{3}$
		%		and $\boldsymbol{\psi}_{+}$ with the property that $\dot{y}_{j}(t)\rightarrow\beta_{j}$
		%		and
		%		\[
		%		\lim_{t\rightarrow\infty}\text{\ensuremath{\left\Vert \boldsymbol{\psi}(t)-\boldsymbol{Q}(y(t),\beta)-e^{JH_{0}t}\boldsymbol{\psi}_{+}\right\Vert _{\mathcal{H}}}}=0.
		%		\]
	\end{thm}
	
	\begin{proof}
		%The proof here is similar to the proof of the theorem above. The key
		%difference now is that in the problem above, the path for solitons
		%$\left(\hm{\beta}(t),\hm{y}(t)\right)\in\mathbb{R}^{3N}\times\mathbb{R}^{3N}$
		%is given but now we need to find paths of solitons together with solving
		%the equation for the radiation.
		We show the theorem above via two steps:  assuming the desired scattering solution exists first, we compute the equations it needs to satisfy, and then we use contraction and iteration to construct the solution. 
		
		Assume there is a solution $\hm{\psi}$ scatters to the multi-soliton family. 
		We set $\hm{\psi}(t)-\hm{Q}\big(\hm{\beta}(t),\hm{y}(t)\big)=:\hm{u}(t)$ and take the initial data of the form
		\begin{equation}\label{eq:initial}
			\boldsymbol{\psi}_{0}-\boldsymbol{Q}(\hm{\beta}_{0},\hm{\hm{y}}_{0})=\hm{u}(0)=\hm{R}_0+\hm{\Phi}\cdot\hm{\mathcal{Y}}^+(\hm{\beta}_0,\hm{y}_0)
		\end{equation}
		for some vector $\hm{\Phi}$, and then by construction one has that	the orthogonality conditions
		\[
		\left\langle \alpha_{j,m,\beta_{j,0}}^{0}(\cdot-y_{j,0}),\hm{u}(0)\right\rangle =\left\langle \alpha_{j,m,\beta_{j,0}}^{1}(\cdot-y_{j,0}),\hm{u}(0)\right\rangle =0
		\]
		hold for $j=1,\ldots,N$ and $m=1,2,3$. 
		%Let $\bs \psi(t)$ be the solution to \eqref{eq:dynkg} with initial data $\bs\psi_0$. 
		In order to  ensure the scattering behavior of $\hm{u}$, we solve the error term $\hm{u}(t)$
		and parameters $\left(\hm{\beta}(t),\hm{y}(t)\right)\in\mathbb{R}^{3N}\times\mathbb{R}^{3N}$
		satisfying the following system of equation
		\begin{align}
			\frac{d}{dt}\hm{u}(t) & =JH(t)\hm{u}(t)+\mathcal{I}(Q)+\mathfrak{I}_{1}(Q^{2},u)+\mathfrak{I}_{2}(Q,u^{2})+\hm{F}(u)\nonumber \\
			& -\dot{\hm{\beta}}(t)\partial_{\hm{\beta}}\hm{Q}(\hm{\beta}(t),\hm{y}(t))-\left(\dot{\hm{y}}(t)-\hm{\beta}(t)\right)\partial_{\hm{y}}\hm{Q}(\hm{\beta}(t),\hm{y}(t))\nonumber \\
			& =:JH(t)\hm{u}(t)+\mathcal{I}(Q)+\mathfrak{I}(Q,u)+\hm{F}(u)+\text{Mod}'(t)\nabla_{M}\hm{Q}(\hm{\beta}(t),\hm{y}(t))\label{eq:manieq1}\\
			& =:JH(t)\hm{u}(t)+\mathcal{W}(x,t)\nonumber 
		\end{align}
		and  modulation equations		
		\begin{equation}
			\frac{d}{dt}\left\langle \alpha_{m,\beta_{j}(t)}^{0}(\cdot-y_{j}(t)),\hm{u}(t)\right\rangle =\frac{d}{dt}\left\langle \alpha_{m,\beta_{j}(t)}^{1}(\cdot-y_{j}(t)),\hm{u}(t)\right\rangle =0\label{eq:manieq2}
		\end{equation}
		with initial data $\hm{u}(0)$ and $\left(\hm{\beta}(t),\hm{y}(t)\right)=\left(\hm{\beta}_0,\hm{y}_0\right)$.		
		
		Using projections associated  to $(\hm{\beta}(t),\hm{y}(t))$,  one can decompose $\hm{u}(t)$ into
		\[
		\hm{u}(t)=\pi_{c}(t)\hm{u}(t)+\pi_+(t)\hm{u}(t)+\pi_-(t)\hm{u}(t).
		\]
		To control $\pi_+(t)\hm{u}(t)$, we compute $a_{j}^{+}(t)$. Following the same computations as in the proof of Theorem \ref{thm:orbitasy3}, one has 
		\begin{align*}
			a_{j}^{+}(t) & =a_{j}^{+}(0)\exp\left(\int_{0}^{t}\frac{\nu}{\gamma_{j}(\tau)}\,d\tau\right)+\int_{0}^{t}\exp\left(\int_{s}^{t}\frac{\nu}{\gamma_{j}(\tau)}\,d\tau\right)\left\langle \alpha_{\beta_{j}(s)}^{+}(\cdot-y_{j}(s),\mathcal{W}(\cdot,s)\right\rangle \,ds\\
			& +\sum_{k\neq j}\int_{0}^{t}\exp\left(\int_{s}^{t}\frac{\nu}{\gamma_{j}(\tau)}\,d\tau\right)\left\langle \alpha_{\beta_{j}(s)}^{+}\big(\cdot-y_{j}(s)\big),\mathcal{V}_{\beta_{k}(s)}\big(\cdot-y_{k}(s)\big)\hm{u}(s)\right\rangle \,ds\\
			& \text{+\ensuremath{\int_{0}^{t}}}\exp\left(\int_{s}^{t}\frac{\nu}{\gamma_{j}(\tau)}\,d\tau\right)\left\langle \left(y'_{j}(s)-\beta_{j}(s)\right)\cdot\nabla\alpha_{\beta_{j}(s)}^{\text{+}}(\cdot-y_{j}(s)),\hm{u}(s)\right\rangle \,ds\\
			& \text{+\ensuremath{\int_{0}^{t}}}\exp\left(\int_{s}^{t}\frac{\nu}{\gamma_{j}(\tau)}\,d\tau\right)\left\langle \beta_{j}'(s)\partial_{\beta}\alpha_{\beta_{j}(s)}^{+}(\cdot-y_{j}(s)),\hm{u}(s)\right\rangle \,ds.
		\end{align*}
		Moreover, since the solution scatters to the multi-soliton family,
		the system above is stabilized by the following conditions for the
		unstable modes:
		\begin{align}
			a_{j}^{+}(t) & =\left\langle \alpha_{\beta_{j}(t)}^{+}(\cdot-y_{j}(t)),\hm{u}(t)\right\rangle \label{eq:manieq3}\\
			& =-\int_{t}^{\infty}\exp\left(\int_{s}^{t}\frac{\nu}{\gamma_{j}(\tau)}\,d\tau\right)\left\langle \alpha_{\beta_{j}(s)}^{+}\big(\cdot-y_{j}(s)\big),\mathcal{W}(\cdot,s)\right\rangle \,ds\nonumber \\
			& -\sum_{k\neq j}\int_{t}^{\infty}\exp\left(\int_{s}^{t}\frac{\nu}{\gamma_{j}(\tau)}\,d\tau\right)\left\langle \alpha_{\beta_{j}(s)}^{+}\big(\cdot-y_{j}(s)\big),\mathcal{V}_{\beta_{k}(s)}\big(\cdot-y_{k}(s)\big)\hm{u}(s)\right\rangle \,ds\nonumber \\
			& -\int_{t}^{\infty}\exp\left(\int_{s}^{t}\frac{\nu}{\gamma_{j}(\tau)}\,d\tau\right)\left\langle \left(y'_{j}(s)-\beta_{j}(s)\right)\cdot\nabla\alpha_{\beta_{j}(s)}^{\text{+}}\big(\cdot-y_{j}(s)\big),\hm{u}(s)\right\rangle \,ds\nonumber \\
			& -\int_{t}^{\infty}\exp\left(\int_{s}^{t}\frac{\nu}{\gamma_{j}(\tau)}\,d\tau\right)\left\langle \beta_{j}'(s)\partial_{\beta}\alpha_{\beta_{j}(s)}^{+}\big(\cdot-y_{j}(s)\big),\hm{u}(s)\right\rangle \,ds.\nonumber 
		\end{align}
		By similar computations, we have the following formula for stable modes:
		{\small\begin{align}
				a_{j}^{-}(t) & =a_{j}^{-}(0)\exp\left(-\int_{0}^{t}\frac{\nu}{\gamma_{j}(\tau)}\,d\tau\right)-\int_{0}^{t}\exp\left(-\int_{s}^{t}\frac{\nu}{\gamma_{j}(\tau)}\,d\tau\right)\left\langle \alpha_{\beta_{j}}^{-}(\cdot-\beta_{j}s-x_{j}),\mathcal{W}(\cdot,s)\right\rangle \,ds\label{eq:stablecs}\\
				& +\sum_{k\neq j}\int_{0}^{t}\exp\left(-\int_{s}^{t}\frac{\nu}{\gamma_{j}(\tau)}\,d\tau\right)\left\langle \alpha_{\beta_{j}}^{-}(\cdot-\beta_{j}s-x_{j}),\mathcal{V}_{\beta_{k}}\big(\cdot-\beta_{k}s-x_{k}\big)\hm{u}(s)\right\rangle \,ds\nonumber \\
				& +\int_{0}^{t}\exp\left(-\int_{s}^{t}\frac{\nu}{\gamma_{j}(\tau)}\,d\tau\right)\left\langle \left(y'_{j}(s)-\beta_{j}(s)\right)\cdot\nabla\alpha_{\beta_{j}(s)}^{-}(\cdot-y_{j}(s)),\hm{u}(s)\right\rangle \,ds\nonumber \\
				& +\int_{0}^{t}\exp\left(-\int_{s}^{t}\frac{\nu}{\gamma_{j}(\tau)}\,d\tau\right)\left\langle \beta_{j}'(s)\partial_{\beta}\alpha_{\beta_{j}(s)}^{-}(\cdot-y_{j}(s)),\hm{u}(s)\right\rangle \,ds.\nonumber 
		\end{align}}
		With computations above, to find a solution which scatters to the multi-soliton family, it is reduced to find the solution to the system with equations
		\eqref{eq:manieq1},
		\eqref{eq:manieq2},  \eqref{eq:manieq3} and \eqref{eq:stablecs}. We will achieve this goal by the fixed point theorem
		via a contraction map.   Assuming the unique solution can be found, then the initial values for unstable modes are uniquely determined: $\hm{\Phi}=(\Phi_j)_{j=1}^N=(a_j^+(0))_{j=1}^N$ from \eqref{eq:manieq3}. This in particular  defines$$\hm{\Phi}(\hm{R}_0)=(a_j^+(0))_{j=1}^N.$$

		\noindent{\bf Basic setting of the contraction map:} Given a set of data $\left(\hm{\beta}(t),\hm{y}(t),\hm{u}(t)\right),$
		we consider the following map
		\begin{equation}
			\mathscr{F}\text{\ensuremath{\left(\left(\hm{\beta}(t),\hm{y}(t),\hm{u}(t)\right)\right)}=\ensuremath{\left(\tilde{\hm{\beta}}(t),\tilde{\hm{y}}(t),\tilde{\hm{u}}(t)\right)}}\label{eq:inputoutput}
		\end{equation}
		where $\left(\tilde{\hm{\beta}}(t),\tilde{\hm{y}}(t),\tilde{\hm{u}}(t)\right)$
		are defined by solving
		\begin{align*}
			\frac{d}{dt}\tilde{\hm{u}}(t) & =JH(t)\tilde{\hm{u}}(t)+\mathcal{I}(Q)+\mathfrak{I}_{1}(Q^{2},u)+\mathfrak{I}_{2}(Q,u^{2})+\hm{F}(u)\\
			& -\dot{\tilde{\hm{\beta}}}(t)\partial_{\hm{\beta}}\hm{Q}(\hm{\beta}(t),\hm{y}(t))-\left(\dot{\tilde{\hm{y}}}(t)-\tilde{\hm{\beta}}(t)\right)\partial_{\hm{y}}\hm{Q}(\hm{\beta}(t),\hm{y}(t))\\
			& =:JH(t)\tilde{\hm{u}}(t)+\mathcal{I}(Q)+\mathfrak{I}(Q,u)+\hm{F}(u)+\widetilde{\text{Mod}'}(t)\nabla_{M}\hm{Q}(\hm{\beta}(t),\hm{y}(t))\\
			& =:JH(t)\tilde{\hm{u}}(t)+\mathcal{W}(t,x)
			%(\hm{u}(t),\hm{\beta}(t),\hm{y}(t))+\widetilde{\text{Mod}'}(t)\nabla_{M}\hm{Q}(\hm{\beta}(t),\hm{y}(t))
		\end{align*}
		and the modulation equations
		\begin{equation}\label{eq:modmani}
			\frac{d}{dt}\left\langle \alpha_{m,\beta_{j}(t)}^{0}(\cdot-y_{j}(t)),\tilde{\hm{u}}(t)\right\rangle =\frac{d}{dt}\left\langle \alpha_{m,\beta_{j}(t)}^{1}(\cdot-y_{j}(t)),\tilde{\hm{u}}(t)\right\rangle =0
		\end{equation}
		with the same initial data for the modulation parameters, i.e., $\left(\tilde{\hm{\beta}}(0),\tilde{\hm{y}}(0)\right)=\left(\hm{\beta}_{0},\hm{\hm{y}}_{0}\right)$, and 
		\begin{equation}
			\tilde{\hm{u}}(0)=\hm{R}_0+\hm{\Phi}\cdot\hm{\mathcal{Y}}^+(\hm{\beta}_0,\hm{y}_0) 
		\end{equation} for some $\hm{\Phi}$.

		\iffalse
		Note that the solution to the original problem is given by the fixed point of
		\begin{equation}
			\mathscr{F}\text{\ensuremath{\left(\left(\hm{\beta}(t),\hm{y}(t),\hm{u}(t)\right)\right)}=\ensuremath{\left(\hm{\beta}(t),\hm{y}(t),\hm{u}(t)\right)}}.
		\end{equation}
		
		We will firstly perform {\it a priori} estimates in a strong topology. Then to show the contraction, we will employ a weaker topology in order to handle the comparison of different paths.
		\fi	
		\noindent{\bf {\it {\bf \emph{a priori}}} estimates:}
		We perform {\it a priori} estimates in
		\[
		\mathcal{A}_{\eta}=:\left\{ \left(\hm{\beta}(t),\hm{y}(t),\hm{u}(t)\right)|\left\Vert \hm{\beta}'\right\Vert _{L_{t}^{1}\bigcap L_{t}^{\infty}}+\left\Vert \hm{y}'-\hm{\beta}\right\Vert _{L_{t}^{1}\bigcap L_{t}^{\infty}}+\left\Vert \hm{u}\right\Vert^2 _{\mathcal{S}_{\mathcal{H}}}\leq A\eta^2\right\} 
		\]
		with the norm
		\begin{equation}\label{eq:Aetanorm}
			\left\Vert \left(\hm{\beta},\hm{y},\hm{u}\right)\right\Vert _{\mathcal{A}_{\eta}}=\sqrt{\left\Vert \hm{\beta}'\right\Vert _{L_{t}^{1}\bigcap L_{t}^{\infty}}}+\sqrt{\left\Vert \hm{y}'-\hm{\beta}\right\Vert _{L_{t}^{1}\bigcap L_{t}^{\infty}}}+\left\Vert \hm{u}\right\Vert _{\mathcal{S}_{\mathcal{H}}}.
		\end{equation}
		Supposing  $\left(\hm{\beta}(t),\hm{y}(t),\hm{u}(t)\right)\in\mathcal{A}_{\eta}$, we claim that there is a unique choice of $\hm{\Phi}:=\hm{\Phi}(\hm{R}_0,\hm{\beta},\hm{y},\hm{u})$ such that $\left(\tilde{\hm{\beta}}(t),\tilde{\hm{y}}(t),\tilde{\hm{u}}(t)\right)\in \mathcal{A}_\eta$ and we can estimate the size of the norm in terms of $\left(\hm{\beta}(t),\hm{y}(t),\hm{u}(t)\right)$.
		\iffalse
		will estimate the corresponding norms for $\left(\tilde{\hm{\beta}}(t),\tilde{\hm{y}}(t),\tilde{\hm{u}}(t)\right)$.
		\fi		
		We start with modulation equations \eqref{eq:modmani}. By the computations in Lemma \ref{lem:mod}		
		\begin{equation}
			|\tilde{\beta}'_{j}(t)|\lesssim \left(\left\Vert \tilde{\hm{u}}(t)\right\Vert _{\mathcal{H}}+\left\Vert \hm{u}(t)\right\Vert _{\mathcal{H}}+1\right)e^{-(\frac{\delta t+ \rho}{2})}+\sum_{m=1}^{3}\left|\left\langle \alpha_{m,\beta_{j}(t)}^{0}(\cdot-y_{j}(t)),\hm {q}(t)\right\rangle \right|
		\end{equation}
		and
		\begin{equation}
			\left|\tilde{y}'_{j}(t)-\tilde{\beta}_{j}(t)\right|\lesssim\left(\left\Vert \tilde{\hm{u}}(t)\right\Vert _{\mathcal{H}}+\left\Vert \hm{u}(t)\right\Vert _{\mathcal{H}}+1\right)e^{-(\frac{\delta t+ \rho}{2})}+\sum_{m=1}^{3}\left|\left\langle \alpha_{m,\beta_{j}(t)}^{1}(\cdot-y_{j}(t)),\hm {q}(t)\right\rangle \right|
		\end{equation}
		where $\hm {q}(t)=\Large(0,3\sum_{i=1}^{N}Q_{\beta_{i}}(\cdot-y_{i}(t))u^{2}\Large)^T. $
		In the computations above, we  used for $\iota=0,1 $
		\begin{equation}
			\left|\left\langle \alpha_{m,\beta_{j}(t)}^{\iota}(\cdot-y_{j}(t)),\mathcal{V}_{i}(t)\tilde{\hm{u}}(t)\right\rangle \right|\lesssim \exp\left\{-\frac{|y_j(t)-y_i(t)|}{2}\right\}\left\Vert \tilde{\hm{u}}\right\Vert _{\mathcal{H}}
		\end{equation} due to the exponential decay of eigenfunction.  We also used the separation properties of trajectories which are ensured by the size of  norms for $(\hm{\beta}(t),\hm{y}(t))$ and
		\begin{align}
			|y_j(t)-y_i(t)|&\geq \left|\int_0^t( y'_j(s)-y'_i(s))\,ds\right|+\rho\\&\geq \rho+\left|\int_0^t\left((\beta_{j,0}-\beta_{i,0})+ (\beta'_j(s)-\beta_{j,0})-(\beta'_i(s)-\beta_{i,0})\right)\,ds\right|\nonumber
			\\
			&- \left|\int_0^t\left( (y'_j(s)-\beta_j(s))-(y'_i(s)-\beta_i(s)\right)\,ds\right|\nonumber\\
			&\geq \rho+ \left|\int_0^t (\delta -2\eta^2 )\,ds\right|-2\eta^2\ \geq \frac{\delta t+ \rho}{2}\nonumber
		\end{align}provided that we picked $\eta\ll \delta$.  We also used similar estimates for
		\begin{equation}
			\left|\left\langle \alpha_{m,\beta_{j}(t)}^{\iota}(\cdot-y_{j}(t)),\mathfrak{I}_1(Q,u)\right\rangle \right|\lesssim e^{-(\frac{\delta t+ \rho}{2})}\left\Vert \hm{u}\right\Vert _{\mathcal{H}}
		\end{equation}
		and 
		\begin{equation}
			\left|\left\langle \alpha_{m,\beta_{j}(t)}^{\iota}(\cdot-y_{j}(t)),	\mathcal{I}(Q)\right\rangle \right|\lesssim e^{-(\frac{\delta t+ \rho}{2})}
		\end{equation}
		where we recall that
		\begin{equation}
			\mathcal{I}(Q)=\left(\begin{array}{c}
				0\\
				\ensuremath{\left(\sum_{j=1}^{N}\sigma_{j}Q_{\beta_{j}(t)}(\cdot-y_{j}(t))\right)^{3}-\sum_{j=1}^{N}\left(\sigma_{j}Q_{\beta_{j}(t)}(\cdot-y_{j}(t))\right)^{3}}
			\end{array}\right)
		\end{equation}
		\begin{equation}
			\mathfrak{I}_1(Q,u)=\left(\begin{array}{c}
				0\\
				3\sum_{i\neq k}\sigma_i\sigma_k Q_{\beta_{i}}(\cdot-y_{i}(t))Q_{\beta_{k}}(\cdot-y_{k}(t))u
			\end{array}\right).
		\end{equation}
		Next, by direct estimates and integration in $t$, we conclude that
		\begin{equation}
			\left\Vert \tilde{\hm{\beta}}'\right\Vert _{L_{t}^{1}\bigcap L_{t}^{\infty}}+\left\Vert \tilde{\hm{y}}'-\tilde{\hm{\beta}}\right\Vert _{L_{t}^{1}\bigcap L_{t}^{\infty}}\lesssim \frac{1}{\delta} e^{-(\frac{ \rho}{2})}\left\Vert \hm{u}\right\Vert _{\mathcal{S}_{\mathcal{H}}}+\frac{1}{\delta} e^{-(\frac{ \rho}{2})}+\frac{1}{\delta} e^{-(\frac{ \rho}{2})}\left\Vert \tilde{\hm{u}}\right\Vert _{\mathcal{S}_{\mathcal{H}}}+\left\Vert \hm{u}\right\Vert _{\mathcal{S}_{\mathcal{H}}}^2.\label{eq:tildemodcs}
		\end{equation}
		Applying Strichartz estimates, see Theorem \ref{thm:mainthmlinear}, we get
		{\small\begin{equation}
				\left\Vert \tilde{\hm{u}}\right\Vert _{\mathcal{S}_{\mathcal{H}}}\lesssim\left\Vert \tilde{\hm{u}}(0)\right\Vert _{\mathcal{H}}+\left\Vert \tilde{\hm{\beta}}'\right\Vert _{L_{t}^{1}\bigcap L_{t}^{\infty}}+\left\Vert \tilde{\hm{y}}'-\tilde{\hm{\beta}}\right\Vert _{L_{t}^{1}\bigcap L_{t}^{\infty}}+\frac{1}{\delta} e^{-(\frac{ \rho}{2})}+\frac{1}{\delta} e^{-(\frac{ \rho}{2})}\left\Vert \hm{u}\right\Vert _{\mathcal{S}_{\mathcal{H}}}+\left\Vert \hm{u}\right\Vert _{\mathcal{S}_{\mathcal{H}}}^2+\Vert \tilde{B}\Vert_{L^1_t\bigcap L^\infty_t}\label{eq:tildecs}
		\end{equation}}where
		\begin{equation}
			\tilde{B}(t):=\sum_{j=1}^N \Large ( |	\tilde{a}_{j}^{+}(t)|+	|\tilde{a}_{j}^{-}(t)|\large).
		\end{equation}
		Requiring 	$\left(\tilde{\hm{\beta}}(t),\tilde{\hm{y}}(t),\tilde{\hm{u}}(t)\right)\in \mathcal{A}_\eta$, it gives the stabilization conditions
		\begin{align}
			\tilde{a}_{j}^{+}(t) & :=\left\langle \alpha_{\beta_{j}(t)}^{+}(\cdot-y_{j}(t)),\tilde{\hm{u}}(t)\right\rangle \label{eq:manieq3-1}\\
			& =-\int_{t}^{\infty}\exp\left(\int_{s}^{t}\frac{\nu}{\gamma_{j}(\tau)}\,d\tau\right)\left\langle \alpha_{\beta_{j}(s)}^{+}\big(\cdot-y_{j}(s)\big),\mathcal{W}(\cdot,s)\right\rangle \,ds\nonumber \\
			& -\sum_{k\neq j}\int_{t}^{\infty}\exp\left(\int_{s}^{t}\frac{\nu}{\gamma_{j}(\tau)}\,d\tau\right)\left\langle \alpha_{\beta_{j}(s)}^{+}\big(\cdot-y_{j}(s)\big),\mathcal{V}_{\beta_{k}(s)}\big(\cdot-y_{k}(s)\big)\tilde{\hm{u}}(s)\right\rangle \,ds\nonumber \\
			& -\int_{t}^{\infty}\exp\left(\int_{s}^{t}\frac{\nu}{\gamma_{j}(\tau)}\,d\tau\right)\left\langle \left(y'_{j}(s)-\beta_{j}(s)\right)\cdot\nabla\alpha_{\beta_{j}(s)}^{\text{+}}\big(\cdot-y_{j}(s)\big),\tilde{\hm{u}}(s)\right\rangle ds\nonumber \\
			& -\int_{t}^{\infty}\exp\left(\int_{s}^{t}\frac{\nu}{\gamma_{j}(\tau)}\,d\tau\right)\left\langle \beta_{j}'(s)\partial_{\beta}\alpha_{\beta_{j}(s)}^{+}\big(\cdot-y_{j}(s)\big),\tilde{\hm{u}}(s)\right\rangle \,ds\nonumber .
		\end{align}
		Then for stable mode, we have		
		{\small\begin{align}
				\tilde{a}_{j}^{-}(t) & =a_{j}^{-}(0)\exp\left(-\int_{0}^{t}\frac{\nu}{\gamma_{j}(\tau)}\,d\tau\right)-\int_{0}^{t}\exp\left(-\int_{s}^{t}\frac{\nu}{\gamma_{j}(\tau)}\,d\tau\right)\left\langle \alpha_{\beta_{j}}^{-}(\cdot-\beta_{j}s-x_{j}),\mathcal{W}(\cdot,s)\right\rangle \,ds\label{eq:stableclass22}\\
				& +\sum_{k\neq j}\int_{0}^{t}\exp\left(-\int_{s}^{t}\frac{\nu}{\gamma_{j}(\tau)}\,d\tau\right)\left\langle \alpha_{\beta_{j}}^{-}(\cdot-\beta_{j}s-x_{j}),\mathcal{V}_{\beta_{k}}\big(\cdot-\beta_{k}s-x_{k}\big)\tilde{\hm{u}}(s))\right\rangle \,ds\nonumber \\
				& +\int_{0}^{t}\exp\left(-\int_{s}^{t}\frac{\nu}{\gamma_{j}(\tau)}\,d\tau\right)\left\langle \left(y'_{j}(s)-\beta_{j}(s)\right)\cdot\nabla\alpha_{\beta_{j}(s)}^{-}(\cdot-y_{j}(s)),\tilde{\hm{u}}(s)\right\rangle \,ds\nonumber \\
				& +\int_{0}^{t}\exp\left(-\int_{s}^{t}\frac{\nu}{\gamma_{j}(\tau)}\,d\tau\right)\left\langle \beta_{j}'(s)\partial_{\beta}\alpha_{\beta_{j}(s)}^{-}(\cdot-y_{j}(s)),\tilde{\hm{u}}(s)\right\rangle \,ds.\nonumber 
		\end{align}}We analyze stable/unstable modes and control the $L^\infty_t\bigcap L^1_t$ norm of $\tilde{a}_j^{\pm}(t)$. From \eqref{eq:manieq3-1}, due to the exponential decay of $\exp\left(\int_{s}^{t}\frac{\nu}{\gamma_{j}(\tau)}\,d\tau\right)$ and the exponential decay of $\exp\left(-\int_{s}^{t}\frac{\nu}{\gamma_{j}(\tau)}\,d\tau\right)$ in \eqref{eq:stableclass22},
		by Young's inequality,  it suffices to bound the $L_{t}^{167tr}$ and
		$L_{t}^{\infty}$ norms of 
		\begin{align*}
			\left\langle \alpha_{\beta_{j}(s)}^{\pm}(\cdot-y_{j}(s)),\mathcal{W}(\cdot,s)\right\rangle +\sum_{k\neq j}\left\langle \alpha_{\beta_{j}(s)}^{\pm}\big(\cdot-y_{j}(s)\big),\mathcal{V}_{\beta_{k}(s)}\big(\cdot-y_{k}(s)\big)\hm{u}(s)\right\rangle \\
			\left\langle \left(y'_{j}(s)-\beta_{j}(s)\right)\cdot\nabla\alpha_{\beta_{j}(s)}^{\pm}(\cdot-y_{j}(s)),\hm{u}(s)\right\rangle +\left\langle \beta_{j}'(s)\partial_{\beta}\alpha_{\beta_{j}(s)}^{\pm}(\cdot-y_{j}(s)),\hm{u}(s)\right\rangle .
		\end{align*}
		Clearly, due to the separation of trajectories and the exponential
		decay of eigenfunctions, one has for $k\neq j$
		\[
		\left|\left\langle \alpha_{\beta_{j}(s)}^{\pm}(\cdot-y_{j}(s)),\mathcal{V}_{\beta_{k}(s)}(\cdot-y_{k}(s))\right\rangle \right|\lesssim e^{-(\frac{\delta s+ \rho}{2})}.
		\]
		Since eigenfunctions are smooth, we also have
		\begin{align*}
			\left|\left\langle \left(y'_{j}(s)-\beta_{j}(s)\right)\cdot\nabla\alpha_{\beta_{j}(s)}^{\pm}(\cdot-y_{j}(s)),\tilde{\hm{u}}(s)\right\rangle \right|\\
			+\left|\left\langle \beta_{j}'(s)\partial_{\beta}\alpha_{\beta_{j}(s)}^{\pm}(\cdot-y_{j}(s)),\tilde{\hm{u}}\right\rangle \right|\\
			\lesssim\left(|\hm{y}'(s)-\hm{\beta}(s)|+|\hm{\beta}'(s)|\right)\left\Vert \tilde{\hm{u}}(s)\right\Vert _{\mathcal{H}}.
		\end{align*}
		Finally, one can compute that
		\begin{align*}
			\left|\left\langle \alpha_{\beta_{j}(s)}^{\pm}(\cdot-y_{j}(s),\mathcal{W}(\cdot,s)\right\rangle \right| & \lesssim e^{-(\frac{\delta s+ \rho}{2})}+|\tilde{\hm{y}}'(s)-\tilde{\hm{\beta}}(s)|+|\tilde{\hm{\beta}}'(s)|\\
			& +\text{\ensuremath{\left(\left\Vert \hm{u}(s)\right\Vert _{\mathcal{H}}^{2}+\left\Vert \hm{u}(s)\right\Vert _{\mathcal{H}}\right)}}e^{-(\frac{\delta s+ \rho}{2})}\\
			&+	\left|\left\langle \alpha_{\beta_{j}(s)}^{+}(\cdot-y_{j}(s),\hm{F}(\hm{u)}\right\rangle \right|
		\end{align*}
		Therefore, we can conclude that given the separation condition, one
		has 
		\begin{align}
			\sum_{j=1}^{N}\left\Vert a_{j}^{+}(\cdot)\right\Vert _{L_{t}^{\infty}\bigcap L_{t}^{1}} & \lesssim\frac{1}{\delta} e^{-(\frac{ \rho}{2})}+ \left\Vert |\tilde{\hm{y}}'(s)-\tilde{\hm{\beta}}(s)|+|\tilde{\hm{\beta}}'(s)|\right\Vert _{L_{t}^{\infty}\bigcap L^1_t}\label{eq:tildeustable}\\
			&+\left\Vert |\hm{y}'(s)-\hm{\beta}(s)|+|\hm{\beta}'(s)|\right\Vert_{L^\infty_t\bigcap L^1_t}\nonumber\\
			& +\frac{1}{\delta} e^{-(\frac{ \rho}{2})}\sup_{s\in\mathbb{R}^{+}}\left(\left\Vert \hm{u}(s)\right\Vert _{\mathcal{H}}^{2}+\left\Vert \hm{u}(s)\right\Vert _{\mathcal{H}}+\left\Vert \tilde{\hm{u}}(s)\right\Vert _{\mathcal{H}}\right)+\left\Vert \hm{u}\right\Vert _{\mathcal{S}_{\mathcal{H}}}^3\nonumber
		\end{align}
		and
		\begin{align}
			\sum_{j=1}^{N}\left\Vert a_{j}^{-}(\cdot)\right\Vert _{L_{t}^{\infty}\bigcap L_{t}^{1}} & \lesssim \Vert \hm{R}_0\Vert_{\mathcal{H}}+\frac{1}{\delta} e^{-(\frac{ \rho}{2})}+ \left\Vert |\tilde{\hm{y}}'(s)-\tilde{\hm{\beta}}(s)|+|\tilde{\hm{\beta}}'(s)|\right\Vert _{L_{t}^{\infty}\bigcap L^1_t}\label{eq:tildestable}\\
			&+\left\Vert |\hm{y}'(s)-\hm{\beta}(s)|+|\hm{\beta}'(s)|\right\Vert_{L^\infty_t\bigcap L^1_t}\nonumber\\
			& +\frac{1}{\delta} e^{-(\frac{ \rho}{2})}\sup_{s\in\mathbb{R}^{+}}\left(\left\Vert \hm{u}(s)\right\Vert _{\mathcal{H}}^{2}+\left\Vert \tilde{\hm{u}}(s)\right\Vert _{\mathcal{H}}+\left\Vert \hm{u}(s)\right\Vert _{\mathcal{H}}\right)+\left\Vert \hm{u}\right\Vert _{\mathcal{S}_{\mathcal{H}}}^3.\nonumber
		\end{align}
		Putting \eqref{eq:tildemodcs}, \eqref{eq:tildecs},  \eqref{eq:tildeustable} and \eqref{eq:tildestable}
		together, taking $\rho$ large and $\eta$ small enough depending
		on $\delta$ and prescribed constants, we conclude that $\left(\tilde{\hm{\beta}}(t),\tilde{\hm{y}}(t),\tilde{\hm{u}}(t)\right)\in\mathcal{A}_{\eta}$.
		In particular, the unique choice of $\hm{\Phi}=(\tilde{a}^+_j(0))_{j=1}^N$ satisfying
		\begin{align}
			|\hm{\Phi}(\hm{R}_0,\hm{\beta},\hm{y},\hm{u})|&\lesssim \frac{1}{\delta} e^{-(\frac{ \rho}{2})}
			+\frac{1}{\delta} e^{-(\frac{ \rho}{2})}\left\Vert \hm{R}_0\right\Vert _{\mathcal{H}}+\frac{1}{\delta} e^{-(\frac{ \rho}{2})}\left\Vert \hm{u}\right\Vert _{\mathcal{S}_{\mathcal{H}}}\\
			&+\left\Vert \hm{u}\right\Vert _{\mathcal{S}_{\mathcal{H}}}^2+\left\Vert |\hm{y}'(s)-\hm{\beta}(s)|+|\hm{\beta}'(s)|\right\Vert_{L^\infty_t\bigcap L^1_t}\nonumber
		\end{align}
		The claim and the {\it a priori} estimates are proved.
		
		\noindent{\bf Contraction:}	 
		Next we will perform the contraction  using weaker norms in order
		to handle potential divergence of paths.
		
		Let $\kappa>0$ be a small fixed number. We define the weighted norm for a function
		of $t$, 
		\[
		\left\Vert f\right\Vert _{G}:=\sup_{t\in\mathbb{R}^{+}}e^{-\kappa t}|f(t)|
		\]
		and for a space-time function, we define weighted Strichartz norms as
		\[
		\left\Vert v\right\Vert _{G\mathcal{S}_{\mathcal{H}}}:=\sup_{t\in\mathbb{R}^{+}}e^{-\kappa t}\left\Vert \hm{v}\right\Vert _{\mathcal{S}_{\mathcal{H}}[0,t]},\,	\left\Vert \hm{v}\right\Vert _{G\mathcal{S}^{*}_{\mathcal{H}}}:=\sup_{t\in\mathbb{R}^{+}}e^{-\kappa t}\left\Vert \hm{v}\right\Vert _{\mathcal{S}^{*}_{\mathcal{H}}[0,t]}
		\]
		We will show the contraction in the following space
		{\small	\begin{equation}\label{eq:AGeta}
				\mathcal{A}_{G,\eta}=:\left\{ \left(\hm{\beta}(t),\hm{y}(t),\hm{u}(t)\right)|\left\Vert \hm{\beta}\right\Vert _{L_{t}^{1}\bigcap L_{t}^{\infty}}+\left\Vert \hm{y}'-\hm{\beta}\right\Vert _{L_{t}^{1}\bigcap L_{t}^{\infty}}+\left\Vert \hm{u}\right\Vert^2 _{\mathcal{S}_{\mathcal{H}}}\leq B\eta^2,\left(\hm{\beta}(0),\hm{y}(0)\right)=\left(\hm{\beta}_{0},\hm{y}_{0}\right)\right\} \end{equation}}
		endowed with the norm
		\begin{equation}\label{eq:AGetanorm}
			\left\Vert \left(\hm{\beta},\hm{y},\hm{u}\right)\right\Vert _{\mathcal{A}_{G,\eta}}=\sqrt{\left\Vert \hm{\beta}'\right\Vert _{G}}+\sqrt{\left\Vert \hm{y}'-\hm{\beta}\right\Vert _{G}}+\left\Vert \hm{u}\right\Vert _{G\mathcal{S}_{\mathcal{H}}}.
		\end{equation}
		To see contraction, we consider two given sets of data $\left(\hm{\beta}_{1}(t),\hm{y}_{1}(t),\hm{u}_{1}(t)\right)$
		and $\left(\hm{\beta}_{2}(t),\hm{y}_{2}(t),\hm{u}_{2}(t)\right)$
		and then we estimate the difference given by two outputs given by
		\eqref{eq:inputoutput}:  for $i=1,2$
		\begin{equation}
			\mathscr{F}\text{\ensuremath{\left(\left(\hm{\beta}_i(t),\hm{y}_i(t),\hm{u}_i(t)\right)\right)}=\ensuremath{\left(\tilde{\hm{\beta}}_i(t),\tilde{\hm{y}}_i(t),\tilde{\hm{u}}_i(t)\right)}}
		\end{equation}
		We define
		\[
		\delta\hm{y}=\hm{y}_{1}-\hm{y}_{2},\,\delta\hm{\beta}=\hm{\beta}_{1}-\hm{\beta}_{2},\,\delta\hm{u}=\hm{u}_{1}-\hm{u}_{2}
		\]
		and
		\[
		\delta\tilde{\hm{y}}=\tilde{\hm{y}}_{1}-\tilde{\hm{y}}_{2},\,\delta\tilde{\hm{\beta}}=\tilde{\hm{\beta}}_{1}-\tilde{\hm{\beta}}_{2},\,\delta\tilde{\hm{u}}=\tilde{\hm{u}}_{1}-\tilde{\hm{u}}_{2}.
		\]
		For modulation parameters, we note that
		\[
		\left\Vert \delta\tilde{\hm{\beta}}\right\Vert _{G}=\frac{1}{\kappa}\left\Vert \delta\tilde{\hm{\beta}}'\right\Vert _{G},\,\left\Vert \delta\tilde{\hm{y}}\right\Vert _{G}=\frac{1}{\kappa}\left\Vert \delta\tilde{\hm{y}}'\right\Vert _{G}\lesssim\frac{1}{\kappa}\left\Vert \delta\tilde{\hm{y}}'-\delta\tilde{\hm{\beta}}\right\Vert _{G}+\frac{1}{\kappa^{2}}\left\Vert \delta\tilde{\hm{\beta}}'\right\Vert _{G}
		\]
		since these parameters have the same initial conditions. 
		These estimates can be used to the difference of two sets of trajectories
		directly. In particular, one can estimate the differences of potentials
		and discrete modes given by two different trajectories. To illustrate
		the idea, we compute
		{\small	\begin{align}
				\left\Vert \left(\mathcal{V}_{\beta_{j,1}(t)}\left(\cdot-y_{j,1}(t)\right)-\mathcal{V}_{\beta_{j,2}(t)}\left(\cdot-y_{j,2}(t)\right)\right)\hm{u}_{2}(t)\right\Vert _{GS^{*}} & \lesssim\left(\left\Vert \delta\hm{\beta}\right\Vert _{G}+\left\Vert \delta\hm{y}\right\Vert _{G}\right)\left\Vert \hm{u}_{2}\right\Vert_{\mathcal{S}_{\mathcal{H}}}\nonumber \\
				& \lesssim\left(\frac{1}{\kappa}\left\Vert \delta\hm{y}'-\delta\hm{\beta}\right\Vert _{G}+\frac{1}{\kappa^{2}}\left\Vert \delta\hm{\beta}'\right\Vert _{G}\right)\left\Vert \hm{u}_{2}\right\Vert_{\mathcal{S}_{\mathcal{H}}}.\label{eq:excompG}
		\end{align}}
		Same computations can be applied to estimate inner products between
		radiation terms $\hm{u}_{j}$ with discrete modes, and interactions
		caused by multiple potentials. 
		\begin{align*}
			\left\Vert \left\langle \alpha_{\beta_{j,1}(s)}^{\iota}\big(\cdot-y_{j,1}(s)\big),\hm{M}_{1}\right\rangle -\left\langle \alpha_{\beta_{j,2}(s)}^{\iota}\big(\cdot-y_{j,2}(s)\big),\hm{M}_{2}\right\rangle \right\Vert _{G}\\
			\text{\ensuremath{\lesssim}\ensuremath{\left\Vert \left\langle \alpha_{\beta_{j,1}(s)}^{\iota}\big(\cdot-y_{j,1}(s)\big),\hm{M}_{1}\right\rangle -\left\langle \alpha_{\beta_{j,2}(s)}^{\iota}\big(\cdot-y_{j,2}(s)\big),\hm{M}_{1}\right\rangle \right\Vert _{G}}}\\
			+\text{\ensuremath{\left\Vert \left\langle \alpha_{\beta_{j,1}(s)}^{\iota}\big(\cdot-y_{j,1}(s)\big),\hm{M}_{1}\right\rangle -\left\langle \alpha_{\beta_{j,2}(s)}^{\iota}\big(\cdot-y_{j,2}(s)\big),\hm{M}_{2}\right\rangle \right\Vert _{G}}}
		\end{align*}
		where $\iota=\pm,0,1$ and $\hm{M}_{j}$, $j=1,2$ denote one of terms appearing
		in \eqref{eq:manieq3-1} and modulation equations.
		
		The first term can be estimated as
		\begin{align*}
			\left\Vert \left\langle \alpha_{\beta_{j,1}(s)}^{\iota}\big(\cdot-y_{j,1}(s)\big),\hm{M}_{1}\right\rangle -\left\langle \alpha_{\beta_{j,2}(s)}^{\iota}\big(\cdot-y_{j,2}(s)\big),\hm{M}_{1}\right\rangle \right\Vert _{G}\\
			\lesssim\left(\left\Vert \delta\hm{\beta}\right\Vert _{G}+\left\Vert \delta\hm{y}\right\Vert _{G}\right)\text{\ensuremath{\left(e^{-(\frac{ \rho}{2})}+\left\Vert \delta\hm{y}'-\delta\hm{\beta}\right\Vert _{G}+\left\Vert \delta\hm{\beta}'\right\Vert _{G}+\left\Vert \hm{u}_{1}\right\Vert _{\mathcal{S}_{\mathcal{H}}}^{2}+\left\Vert \tilde{\hm{u}}_{1}\right\Vert _{\mathcal{S}_{\mathcal{H}}}\right)}}
		\end{align*}
		and the second term can be estimated in a regular manner as in the
		estimates for the central direction which we will compute later on.
		
		Then using ideas above and modulation equations for $\tilde{\hm{\beta}}'_{i}$
		and $\tilde{\hm{y}}'_{i}-\hm{\beta}_{i}$, and then taking the difference
		of those equations, one has
		\begin{equation}
			\left\Vert \delta\tilde{\hm{\beta}}'\right\Vert _{G}\lesssim\frac{1}{\kappa^{2}}B\eta\left(\left\Vert \delta\hm{u}\right\Vert _{G\mathcal{S}_{\mathcal{H}}}+\left\Vert \delta\hm{\beta}\right\Vert _{G}+\left\Vert \delta\hm{y}'-\delta\hm{\beta}\right\Vert _{G}\right)\label{eq:diffmodG1}
		\end{equation}
		\begin{align}
			\left\Vert \delta\tilde{\hm{y}}'\right\Vert _{G} & \lesssim\left\Vert \delta\tilde{\hm{y}}'-\delta\tilde{\hm{\beta}}\right\Vert _{G}+\frac{1}{\kappa}\left\Vert \delta\tilde{\hm{\beta}}'\right\Vert _{G}\label{eq:diffmodG2}\\
			& \lesssim\frac{B\eta}{\kappa^{3}}\left(\left\Vert \delta\hm{u}\right\Vert _{G\mathcal{S}_{\mathcal{H}}}+\left\Vert \delta\hm{\beta}\right\Vert _{G}+\left\Vert \delta\hm{y}'-\delta\hm{\beta}\right\Vert _{G}\right).\nonumber 
		\end{align}
		To estimate the difference of the radiation term $\delta\tilde{\hm{u}},$
		we look at the equation for $\tilde{\hm{u}}_{i}$
		\begin{align}
			\frac{d}{dt}\tilde{\hm{u}}_{i}(t) & =JH_{i}(t)\tilde{\hm{u}}_{i}(t)+\mathcal{I}_{i}(Q)+\mathfrak{I}_{1,i}(Q^{2},u_{i})+\mathfrak{I}_{2,i}(Q,u_{i}^{2})+\hm{F}(u_{i})\label{eq:tildeui}\\
			& -\tilde{\hm{\beta}}'_{i}(t)\partial_{\hm{\beta}}\hm{Q}(\hm{\beta}_{i}(t), \hm{y}_{i}(t))-\left(\tilde{\hm{y}_{i}}'(t)-\tilde{\hm{\beta}}_{i}(t)\right)\partial_{\hm{y}}\hm{Q}(\hm{\beta}_{i}(t), \hm{y}_{i}(t))\nonumber \\
			& =:JH_{i}(t)\tilde{\hm{u}}_{i}(t)+\mathcal{I}_{i}(Q)+\mathfrak{I}_{i}(Q,u_{i})+\hm{F}(u_{i})+\widetilde{\text{Mod}'}_{i}(t)\nabla_{M}\hm{Q}(\hm{\beta}_{i}(t), \hm{y}_{i}(t))\nonumber 
			%\\
			%	& =JH_{i}(t)\tilde{\hm{u}}_{i}(t)+\hm{W}(\hm{u}_{i}(t),\hm{\beta}_{i}(t),\hm{y}_{i}(t))+\widetilde{\text{Mod}'}_{i}(t)\nabla_{M}\hm{Q}(\hm{\beta}_{i}(t), \hm{y}_{i}(t))
		\end{align}
		Taking the difference of RHS above for $\tilde{\hm{u}}_{1}$ and $\tilde{\hm{u}}_{2}$,
		we get
		\begin{align}
			\frac{d}{dt}\delta\tilde{\hm{u}}(t) &=	JH_{1}(t)\delta\tilde{\hm{u}}+J(H_{1}-H_{2})\tilde{\hm{u}}_{2}
			+\mathcal{I}_{1}(Q)-\mathcal{I}_{2}(Q)\label{eq:diff}\\
			&+\mathfrak{I}_{1}(Q,u_{1})-\mathfrak{I}_{2}(Q,u_{2})
			+\hm{F}(u_{1})-\hm{F}(u_{2})\nonumber\\
			&+\widetilde{\text{Mod}'}_{1}(t)\nabla_{M}\hm{Q}(\hm{y}_{1}(t),\hm{\beta}_{1}(t))-\widetilde{\text{Mod}'}_{2}(t)\nabla_{M}\hm{Q}(\hm{y}_{2}(t),\hm{\beta}_{2}(t)) \nonumber\\
			&=:JH_{1}(t)\delta\tilde{\hm{u}}+\hm{\Delta}\nonumber
			.
		\end{align}
		We first note that
		\[
		\left\Vert \mathcal{I}_{1}(Q)-\mathcal{I}_{2}(Q)\right\Vert _{G\mathcal{S}^{*}_{\mathcal{H}}}\lesssim \frac{1}{\delta} e^{-(\frac{ \rho}{2})}\left\Vert \delta\hm{\beta}\right\Vert _{G}+\frac{1}{\delta} e^{-(\frac{ \rho}{2})}\left\Vert \delta\hm{y}\right\Vert _{G}
		\]
		\begin{align*}
			\left\Vert \mathfrak{I}_{1}(Q,u_{1})-\mathfrak{I}_{2}(Q,u_{2})\right\Vert _{G\mathcal{S}^{*}_{\mathcal{H}}}&\lesssim
			\frac{1}{\delta} e^{-(\frac{ \rho}{2})}\left\Vert \delta\hm{u}\right\Vert _{G\mathcal{S}_{\mathcal{H}}}	+	(\left\Vert \hm{u}_{1}\right\Vert _{\mathcal{S}_{\mathcal{H}}}+\left\Vert \hm{u}_{2}\right\Vert _{\mathcal{S}_{\mathcal{H}}})\left\Vert \delta\hm{u}\right\Vert _{G\mathcal{S}_{\mathcal{H}}}\\
			&+\left(\left\Vert \delta\hm{\beta}\right\Vert _{G}+\left\Vert \delta\hm{y}\right\Vert _{G}\right)\left\Vert \hm{u}_{2}\right\Vert _{\mathcal{S}_{\mathcal{H}}}
		\end{align*}
		\[
		\left\Vert \hm{F}(u_{1})-\hm{F}(u_{2})\right\Vert _{G\mathcal{S}^{*}_{\mathcal{H}}}\lesssim\left(\left\Vert \hm{u}_{1}\right\Vert _{\mathcal{S}_{\mathcal{H}}}^{2}+\left\Vert \hm{u}_{2}\right\Vert _{\mathcal{S}_{\mathcal{H}}}^{2}\right)\left\Vert \delta\hm{u}\right\Vert _{G\mathcal{S}_{\mathcal{H}}}
		\]
		\begin{align*}
			\left\Vert \widetilde{\text{Mod}'}_{1}(t)\nabla_{M}\hm{Q}(\hm{y}_{1}(t),\hm{\beta}_{1}(t))-\widetilde{\text{Mod}'}_{2}(t)\nabla_{M}\hm{Q}(\hm{y}_{2}(t),\hm{\beta}_{2}(t))\right\Vert _{G\mathcal{S}^{*}_{\mathcal{H}}}\,\,\,\,\,\, \,\,\,\,\,\,\,\,\,\\
			\lesssim\left\Vert \delta\tilde{\hm{\beta}}\right\Vert _{G}+\left\Vert \delta\tilde{\hm{y}}'-\delta\tilde{\hm{\beta}}\right\Vert _{G}
			+\left(\left\Vert \tilde{\hm{\beta}}'\right\Vert _{L_{t}^{1}\bigcap L_{t}^{\infty}}+\left\Vert \tilde{\hm{y}}'-\tilde{\hm{\beta}}\right\Vert _{L_{t}^{1}\bigcap L_{t}^{\infty}}\right)\left(\left\Vert \delta\hm{\beta}\right\Vert _{G}+\left\Vert \delta\hm{y}\right\Vert _{G}\right).
		\end{align*}
		Let $\pi_{\ell,i}(t)$, $\ell=\pm,0,c,\,i=1,2$ be  projections
		with respect to $JH_{i}(t)$ given by the set of paths $\left(\hm{\beta}_{i}(t),\hm{y}_{i}(t)\right).$
		Let $\pi_{0,i}(t)$ be the projection on the generalized kernel with
		respect to $JH_{i}(t)$. 		
		To estimate the weighted Strichartz norm, from equation \eqref{eq:diff}, we write
		\begin{equation}
			\delta\tilde{\hm{u}}(t)=\mathcal{T}_1(t,0)\delta\tilde{\hm{u}}(0)+\int_0^t \mathcal{T}_1(t,s) \hm{\Delta}\,ds
		\end{equation} where $\mathcal{T}_1(t,s)$ denotes the propagator with respect to $J\mathcal{H}_1(t)$. Applying Strichartz norms on both sides, it follows
		\begin{align}
			\left\Vert \delta\tilde{\hm{u}}\right\Vert _{\mathcal{S}_{\mathcal{H}}[0,t]} &\lesssim  \left\Vert \delta\tilde{\hm{u}}(0)\right\Vert _{\mathcal{H}}+ \left\Vert \hm{\Delta }\right\Vert _{\mathcal{S}^*_{\mathcal{H}}[0,t]}+\left\Vert \pi_{0,1}(\cdot) \delta\tilde{\hm{u}}\right\Vert _{\mathcal{S}^*_{\mathcal{H}}[0,t]}\\
			&+\left\Vert \pi_{+,1}(\cdot) \delta\tilde{\hm{u}}\right\Vert _{\mathcal{S}^*_{\mathcal{H}}[0,t]}
			+\left\Vert \pi_{-,1}(\cdot) \delta\tilde{\hm{u}}\right\Vert _{\mathcal{S}^*_{\mathcal{H}}[0,t]},
		\end{align}
		then the multiplying both sides by $e^{-\kappa t}$ and taking supremun, we get
		\begin{align}
			\left\Vert \delta\tilde{\hm{u}}\right\Vert _{G\mathcal{S}_{\mathcal{H}}} &\lesssim  \left\Vert \delta\tilde{\hm{u}}(0)\right\Vert _{\mathcal{H}}+ \left\Vert \hm{\Delta }\right\Vert _{G\mathcal{S}^*_{\mathcal{H}}}+\left\Vert \pi_{0,1}(\cdot) \delta\tilde{\hm{u}}\right\Vert _{G\mathcal{S}^*_{\mathcal{H}}}\label{eq:deltau}\\
			&+\left\Vert \pi_{+,1}(\cdot) \delta\tilde{\hm{u}}\right\Vert _{G\mathcal{S}^*_{\mathcal{H}}}
			+\left\Vert \pi_{-,1}(\cdot) \delta\tilde{\hm{u}}\right\Vert _{G\mathcal{S}^*_{\mathcal{H}}}.\nonumber\
		\end{align}
		Now we estimate terms appearing on the RHS above one by one.

		Starting from zero modes, by construction, one has
		\[
		\pi_{0,1}(t)(\delta\tilde{\hm{u}})=-\pi_{0,1}(t)\tilde{\hm{u}}_{2}=\text{\ensuremath{\left(\pi_{0,2}(t)-\pi_{0,1}(t)\right)\tilde{\hm{u}}_{2}}}.
		\]
		Then by a similar computation to \eqref{eq:excompG}, using the decay of zero modes, we get
		\begin{align}
			\left\Vert \pi_{0,1}(\cdot)(\delta\tilde{\hm{u}})\right\Vert _{G\mathcal{S}^*_{\mathcal{H}}} & =\left\Vert \left(\pi_{0,2}(t)-\pi_{0,1}(t)\right)\tilde{\hm{u}}_{2}\right\Vert_{G\mathcal{S}^*_{\mathcal{H}}} \nonumber \\
			& \lesssim\left(\frac{1}{\kappa}\left\Vert \delta\hm{y}'-\delta\hm{\beta}\right\Vert _{G}+\frac{1}{\kappa^{2}}\left\Vert \delta\hm{\beta}'\right\Vert _{G}\right)\left\Vert \tilde{\hm{u}}_{2}\right\Vert_{\mathcal{S}_{\mathcal{H}}}\nonumber \\
			& \lesssim\frac{B\eta}{\kappa^{2}}\left(\left\Vert \delta\hm{y}'-\delta\hm{\beta}\right\Vert _{G}+\left\Vert \delta\hm{\beta}'\right\Vert _{G}\right).\label{eq:diffzeroG}
		\end{align}
		Next from the explicit terms of $\hm{\Delta}$, one has		
		\begin{align}
			\left\Vert \delta\hm{\Delta}\right\Vert _{G\mathcal{S}^*_{\mathcal{H}}} & \lesssim \left\Vert \delta H(\cdot)\right\Vert _{G}\left\Vert \tilde{\hm{u}}_{2}\right\Vert_{\mathcal{S}_{\mathcal{H}}}+\frac{1}{\delta} e^{-(\frac{ \rho}{2})}\left(\frac{1}{\kappa}\left\Vert \delta\hm{y}'-\delta\hm{\beta}\right\Vert _{G}+\frac{1}{\kappa^{2}}\left\Vert \delta\hm{\beta}'\right\Vert _{G}\right)\label{eq:diffcsG}\\
			& +\left(\frac{1}{\kappa}\left\Vert \delta\hm{y}'-\delta\hm{\beta}\right\Vert _{G}+\frac{1}{\kappa^{2}}\left\Vert \delta\hm{\beta}'\right\Vert _{G}\right)\left(\left\Vert \hm{u}_{1}\right\Vert_{\mathcal{S}_{\mathcal{H}}}+\left\Vert \hm{u}_{1}\right\Vert _{S_\mathcal{H}}\right)\nonumber \\
			& +\frac{1}{\delta} e^{-(\frac{ \rho}{2})}\left\Vert \delta\hm{u}\right\Vert _{G\mathcal{S}_{\mathcal{H}}}+\left(\left\Vert \hm{u}_{1}\right\Vert_{\mathcal{S}_{\mathcal{H}}}^{2}+\left\Vert \hm{u}_{2}\right\Vert_{\mathcal{S}_{\mathcal{H}}}^{2}\right)\left\Vert \delta\hm{u}\right\Vert _{G\mathcal{S}_{\mathcal{H}}}\nonumber \\
			& \lesssim\left\Vert \delta\hm{u}(0)\right\Vert _{\mathcal{H}}+\frac{B\eta}{\kappa^{2}}\left(\left\Vert \delta\hm{y}'-\delta\hm{\beta}\right\Vert _{G}+\left\Vert \delta\hm{\beta}'\right\Vert _{G}+\left\Vert \delta\hm{u}\right\Vert _{G\mathcal{S}_{\mathcal{H}}}\right)\nonumber \\
			& +\frac{1}{\delta} e^{-(\frac{ \rho}{2})}\left(\frac{1}{\kappa}\left\Vert \delta\hm{y}'-\delta\hm{\beta}\right\Vert _{G}+\frac{1}{\kappa^{2}}\left\Vert \delta\hm{\beta}'\right\Vert _{G}+\left\Vert \delta\hm{u}\right\Vert _{G\mathcal{S}_{\mathcal{H}}}\right).\nonumber 
		\end{align}
		Finally, we estimate $\pi_{\pm,1}(t)(\delta\tilde{\hm{u}})$. 
		This is similar to the argument \eqref{eq:diffcsG} above. The only
		difference is that here we need to take $\kappa$ small enough such
		that
		\[
		\int_{t}^{\infty}\left|e^{(s-t)\kappa}\exp\left(\int_{s}^{t}\frac{\nu}{\gamma_{j,1}(\tau)}\,d\tau\right)\right|ds\lesssim1.
		\]
		Note that for $\pi_{-,1}(t)(\delta \tilde{\hm{u}})$, automatically, we have
		\[
		\int_{0}^{t}\left|e^{(s-t)\kappa}\exp\left(-\int_{s}^{t}\frac{\nu}{\gamma_{j,1}(\tau)}\,d\tau\right)\right|ds\lesssim1.
		\]
		Then we can bound
		\begin{equation}
			\left\Vert \pi_{\pm,1}(t)(\delta\tilde{\hm{u}})\right\Vert _{G\mathcal{S}^*_{\mathcal{H}}}\lesssim\left\Vert \delta\hm{u}(0)\right\Vert _{\mathcal{H}}+\frac{1}{\kappa^{2}}\big(B\eta+\frac{1}{\delta} e^{-(\frac{ \rho}{2})}\big)\left(\left\Vert \delta\hm{y}'-\delta\hm{\beta}\right\Vert _{G}+\left\Vert \delta\hm{\beta}'\right\Vert _{G}+\left\Vert \delta\hm{u}\right\Vert _{G\mathcal{S}_{\mathcal{H}}}\right).\label{eq:diffuG}
		\end{equation}
		Therefore, from \eqref{eq:diffmodG1}, \eqref{eq:diffmodG2},
		\eqref{eq:deltau},
		\eqref{eq:diffzeroG},
		\eqref{eq:diffcsG} and \eqref{eq:diffuG}, we conclude that
		\[
		\left\Vert \left(\delta\tilde{\hm{\beta}},\delta\tilde{\hm{y}},\delta\tilde{\hm{u}}\right)\right\Vert _{\mathcal{A}_{G,\eta}}\lesssim\left\Vert \delta\hm{u}(0)\right\Vert _{\mathcal{H}}+\frac{1}{\kappa^{2}}\big(B\eta+\frac{1}{\delta} e^{-(\frac{ \rho}{2})}\big)\left\Vert \left(\delta\hm{\beta},\delta\hm{y},\delta\hm{u}\right)\right\Vert _{\mathcal{A}_{G,\eta}}
		\]
		which results in that $\mathfrak{F}$ from \eqref{eq:inputoutput} a contraction provided $\eta$ is
		small enough and $\rho>0$ large enough.  Therefore, with the {\it a priori} estimates, there is a unique
		fixed point $\left(\hm{\beta},\hm{y},\hm{u}\right)$ for the map \eqref{eq:inputoutput} in $\mathcal{A}_{G,\eta}$.
		Moreover, by construction, with respect to the norm of $\mathcal{A}_{G,\eta}$,
		the fixed point $\left(\hm{\beta}(t),\hm{y}(t),\hm{u}(t)\right)$
		is a Lipchitz function of  $\pi_{c}(0)\hm{u}(0)=\hm{R}_0$.  In particular
		$\hm{u}(0)$ of the fixed point is a Lipchitz  function of  $\pi_{c}(0)\hm{u}(0)$.
		Note that the contraction topology is irrelevant since $\pi_{+}(0)\hm{u}(0)$
		is of finite-dimension.
		
		Hence, we obtain a Lipchitz  map $\hm{\Phi}:\mathcal{N}_{L}\left(\hm{Q}(\hm{\beta}_{0},\hm{y}_{0}),\eta\right)\rightarrow\mathbb{R}^{N}$ as claimed.  The desired estimates of $\hm{\Phi}$ follow from the {\it a priori} estimates  and difference estimates above.
		
		\smallskip
		
		\noindent{\bf Scattering}: 	
		From the construction, since the fixed point $\left(\hm{\beta}(t),\hm{y}(t),\hm{u}(t)\right)$
		is bounded with respect to the norm of $\mathcal{A}_{\eta}$, i.e.,
		\[
		\sqrt{\left\Vert \hm{\beta}'\right\Vert _{L_{t}^{1}\bigcap L_{t}^{\infty}}}+\sqrt{\left\Vert \hm{y}'-\hm{\beta}\right\Vert _{L_{t}^{1}\bigcap L_{t}^{\infty}}}+\left\Vert \hm{u}\right\Vert _{\mathcal{S}_{\mathcal{H}}}\lesssim\eta.,
		\]
		it follows that there exist $\beta_{j}\in\mathbb{R}^{3}$, 
		with the property that $\dot{y}_{j}(t)\rightarrow\beta_{j}$  $\beta_j(t)\rightarrow\beta_j. $

		The scattering of $\hm{u}$ is achieved by the standard argument via
		the Yang-Feldman equation as in the scattering analysis in Theorem \ref{thm:orbitasy3}. 
		We conclude	
		\[
		\lim_{t\rightarrow\infty}\text{\ensuremath{\left\Vert \boldsymbol{\psi}(t)-\boldsymbol{Q}(y(t),\beta)-e^{JH_{0}t}\boldsymbol{\psi}_{+}\right\Vert _{\mathcal{H}}}}=0
		\]
		which gives the desired result.
	\end{proof}
	\begin{rem}
		Our contraction argument presented above is slightly different from Nakanishi-Schlag \cite{NSch}. We performed our analysis with moving potentials. But in \cite{NSch}, they made a change of variable to make the potential static which is not applicable in our multiple moving potential problem. Moreover, in \cite{NSch}, the estimate for the difference of the centra part is based on an energy estimate at $H^{-1}$ level. In our problem, we used weighted Strichartz estimates.
		%presented below. These weighted Strichartz estimates are also useful in the classification of pure multi-solitons in the next section.
	\end{rem}
	\begin{rem}\label{rem:sizebeta} Note that the size of neighborhood $\delta$ above also depends on $\max_j |\beta_{0,j}|$ via constants from Strichartz estimates and size of eigenfunctions. We notice that as  $\max_j |\beta_{0,j}|\rightarrow1$, $\delta\rightarrow0$.
		
	\end{rem}
	Note that the construction $\hm{\Phi}(\hm{R}_0)$ above clearly depends on $\boldsymbol{Q}(\hm{\beta}_0,\hm{y}_0)$. So in those settings which we need to emphasize the dependence on the multi-soliton, we will write it as $\hm{\Phi}\left(\hm{R}_0;\boldsymbol{Q}(\hm{\beta}_0,\hm{y}_0)\right).$
	
	Using implicit function theorem and modulation techniques, locally we can gain $6N$ dimension back in the neighborhood of $\boldsymbol{Q}(\hm{\beta}_0,\hm{y}_0)$.
	\begin{thm}\label{thm:localmani} For every $\delta>0$, there exist $\rho>0$ and $\eta>0$ such that
		the following holds. Let $\big(\hm{\beta}_0,\hm{y}_0\big)$ satisfy the
		separation condition  in the sense of Definition \ref{def:sep} with parameters $\delta$ and $\rho$. 
		Then there exists a Lipschitz manifold $\mathcal{N}\left(\boldsymbol{Q}(\hm{\beta}_{0},\hm{\hm{y}}_{0})\right)$
		inside the small ball $B_{\eta}\left(\boldsymbol{Q}(\hm{\beta}_{0},\hm{\hm{y}}_{0})\right)\subset\mathcal{H}$
		of  codimension $N$  with 
		the following property: for any choice of initial data $\hm{\psi}(0)\in\mathcal{N}\left(\boldsymbol{Q}(\hm{\beta}_{0},\hm{\hm{y}}_{0})\right)$,
		the solution $\hm{\psi}(t)$  to \eqref{eq:dynkg} with initial data $\hm{\psi}(0)$
		exists globally, and it scatters to the multi-soliton family there
		exist $\beta_{j}\in\mathbb{R}^{3}$, paths $y_{j}(t)\in\mathbb{R}^{3}$
		and $\boldsymbol{\psi}_{+}\in\mathcal{H}$ with the property that $\dot{y}_{j}(t)\rightarrow\beta_{j}$
		and
		\[
		\lim_{t\rightarrow\infty}\text{\ensuremath{\left\Vert \boldsymbol{\psi}(t)-\boldsymbol{Q}(\hm{\beta},\hm{y}(t))-e^{JH_{0}t}\boldsymbol{\psi}_{+}\right\Vert _{\mathcal{H}}}}=0.
		\]
	\end{thm}
	\begin{proof}
		The analysis is based on the construction above of Theorem 
		\ref{thm:manifold1}. We first introduce some notations again.
		\iffalse
		Given arbitrary $\hm{R}_0\in\mathcal{N}_L(\boldsymbol{Q}(\hm{\beta}_0,\hm{y}_0)$, we define the map
		\begin{equation}
			\mathfrak{F}_{\boldsymbol{Q}(\hm{\beta}_0,\hm{y}_0)}\left(\hm{R}_0\right):=\boldsymbol{Q}(\hm{\beta}_0,\hm{y}_0)+\hm{R}_0+\hm{\Phi}\left(\hm{R}_0;\boldsymbol{Q}(\hm{\beta}_0,\hm{y}_0)\right)\cdot \hm{\mathcal{Y}}^+\left(\hm{\beta}_0,\hm{y}_0\right)
		\end{equation}which is Lipschitz in $\hm{R}_0$.

		We denote 
		\begin{equation}
			\mathcal{N}_N (\boldsymbol{Q}(\hm{\beta}_0,\hm{y}_0)):=  \mathfrak{F}_{\boldsymbol{Q}(\hm{\beta}_0,\hm{y}_0)}\left(\mathcal{N}_L(\boldsymbol{Q}(\hm{\beta}_0,\hm{y}_0))\right).
		\end{equation}
		\fi
		
		Given a multi-soliton $\hm{Q}(\hm{\beta},\hm{y})$, we let $\Pi_{j,0}\big(\hm{Q}(\hm{\beta},\hm{y})\big)$ be the projection onto the span of $\cY_{m, \beta_{0,j}}^0(\cdot - y_{j})$ and $ \cY_{m, \beta_{j}}^1(\cdot - y_{j})$. And set $\Pi_{j,\pm}\left(\hm{Q}(\hm{\beta},\hm{y})\right)$ be projections onto $\mathcal{Y}^{\pm}_{\beta_{j}}(\cdot-y_{j})$ respectively.
		
		%We introduce the exponential map from the tangent space of one single soliton.
		%\begin{equation}\label{eq:expmap}
		%   \hm{Q}_{j,0}(\hm{R}_0):=\exp_{\hm{Q}_j
		%}\left(\Pi_{j,0}\left( \hm{Q}_{j,0}(\hm{R}_0)\right)\hm{R}_0\right)
		%\end{equation}
		%where $\hm{Q}_j=\hm{Q}_{\beta_j}(\cdot-y_j)$.

		%Maybe we can try something different.
		Given any $\hm {R}_0\in B_{\eta}\left(\boldsymbol{Q}(\hm{\beta}_{0},\hm{\hm{y}}_{0})\right)$, we can find \begin{equation}
			\hm{Q}_{\hm{R}_0}:=\mathfrak{M}\left(\hm{R}_0,\hm{Q}(\hm{\beta}_0,\hm{y}_0)\right)
		\end{equation} via the implicit function theorem, see Lemma \ref{lem:mod},  such that
		\begin{equation}
			\Pi_{j,0}\big(\hm{Q}_{\hm{R}_0}\big)\left(\hm{R}_0+\hm{Q}(\hm{\beta}_0,\hm{y}_0)-\hm{Q}_{\hm{R}_0}\right)=0.
		\end{equation}From the implicit function theorem, $\mathfrak{M}\left(\hm{R}_0,\hm{Q}(\hm{\beta}_0,\hm{y}_0)\right)$ is a Lipschitz  function in $\hm{R}_0$.
		%In fact, it should be a smooth function in $\hm{R}_0$ but we will not pursue this here.
		
		Further denote
		\begin{equation}
			\mathfrak{R}(\hm{R}_0):=\hm{R}_0+\hm{Q}(\hm{\beta}_0,\hm{y}_0)-\hm{Q}_{\hm{R}_0}.
		\end{equation}
		Then we define a codimension $N$ subspace as
		\begin{equation}
			\mathfrak{N}\left(\hm{Q}(\hm{\beta}_0,\hm{y}_0)\right):=\left\{\hm{R}_0\in\mathcal{H}|\,\left\Vert \hm{R}_0\right\Vert_\mathcal{H}\leq \eta,\Pi_{j,+}\big(\hm{Q}_{\hm{R}_0}\big)\left(\mathfrak{R}(\hm{R}_0)\right)=0,\,j=1,\ldots,N\right\}
		\end{equation}
		and then the codimension $N$ manifold of initial data is given by
		\begin{equation}
			\mathcal{N}\left(\boldsymbol{Q}(\hm{\beta}_{0},\hm{\hm{y}}_{0})\right):=\left\{\hm{Q}_{\hm{R}_0}+\mathfrak{R}(\hm{R}_0)+\hm{\Phi}\left( \mathfrak{R}(\hm{R}_0);\hm{Q}_{\hm{R}_0}\right)\cdot \hm{\mathcal{Y}}^+\left(\hm{Q}_{\hm{R}_0}\right),\,\hm{R}_0\in  \mathfrak{N}\left(\hm{Q}(\hm{\beta}_0,\hm{y}_0)\right)\right\}
		\end{equation} where $\hm{\Psi}$ is constructed from Theorem \ref{thm:manifold1}.
		%More generally,
		
		Finally, we define
		\begin{equation}
			\mathfrak{F}_{\boldsymbol{Q}(\hm{\beta}_0,\hm{y}_0)}\left(\hm{R}_0\right):=\hm{Q}_{\hm{R}_0}+\mathfrak{R}(\hm{R}_0)+\hm{\Phi}\left( \mathfrak{R}(\hm{R}_0);\hm{Q}_{\hm{R}_0}\right)\cdot \hm{\mathcal{Y}}^+\left(\hm{Q}_{\hm{R}_0}\right)
		\end{equation}which can be written as
		\begin{equation}
			\hm{Q}_{\hm{R}_0}+\mathfrak{R}(\hm{R}_0)+\hm{\Phi}\left( \mathfrak{R}(\hm{R}_0);\hm{Q}_{\hm{R}_0}\right)\cdot \hm{\mathcal{Y}}^+\left(\hm{Q}_{\hm{R}_0}\right)=\hm{R}_0+\hm{Q}(\hm{\beta}_0,\hm{y}_0)+\hm{\Phi}\left( \mathfrak{R}(\hm{R}_0);\hm{Q}_{\hm{R}_0}\right)\cdot \hm{\mathcal{Y}}^+\left(\hm{Q}_{\hm{R}_0}\right).
		\end{equation}
		We check that $\mathrm{d}\mathfrak{F}_{\boldsymbol{Q}(\hm{\beta}_0,\hm{y}_0)}(0)\neq 0$.
		By construction $\mathcal{N}\left(\boldsymbol{Q}(\hm{\beta}_{0},\hm{\hm{y}}_{0})\right)$ is the image of $ \mathfrak{N}\left(\hm{Q}(\hm{\beta}_0,\hm{y}_0)\right)$ under the bi-Lipschitiz invertibale map $\mathfrak{F}_{\boldsymbol{Q}(\hm{\beta}_0,\hm{y}_0)}$.
		
		By construction and Theorem \ref{thm:manifold1}, for any solution $\hm{\psi}(t)$ with data $\hm{\psi}(0)\in\mathcal{N}\left(\boldsymbol{Q}(\hm{\beta}_{0},\hm{\hm{y}}_{0})\right)$, it scatters to the multi-soliton family as claimed.
	\end{proof}
	
	Recall that notation of the well-separated multi-soliton family $\mathfrak{S}_{\delta,\rho}$ from Definition \ref{def:sepfamily}. For a fixed $\delta>0$, picking $\rho$ large depending on $\delta$, we take the union
	\begin{equation}
		\mathcal{N}:= \bigcup_{\left(\hm{\beta},\hm{y}\right)\in \mathfrak{S}_{\delta,\rho}} \mathcal{N}\left(\boldsymbol{Q}(\hm{\beta},\hm{\hm{y}})\right).
	\end{equation} We claim that the union above  gives the centre-stable manifold over the well-separated multi-soliton family.
	
	Since each $\mathcal{N}\left(\boldsymbol{Q}(\hm{\beta},\hm{\hm{y}})\right)$ is the image of $\mathfrak{N}\left(\boldsymbol{Q}(\hm{\beta},\hm{\hm{y}})\right)$ under the bi-Lipschitz map $\mathfrak{F}_{\boldsymbol{Q}(\hm{\beta},\hm{y})}$,  it suffices to check the following set
	\begin{equation}
		\mathfrak{N}:=\bigcup_{\left(\hm{\beta},\hm{y}\right)\in \mathfrak{S}_{\delta,\rho}} \mathfrak{N}\left(\boldsymbol{Q}(\hm{\beta},\hm{\hm{y}})\right)
	\end{equation}has a Liphschtiz manifold structure. 
	
	In a neighborhood of $\hm{Q}(\hm{\beta}_0,\hm{y}_0)$, around $\mathfrak{N}\left(\hm{Q}(\hm{\beta}_0,\hm{y}_0)\right)$, a local chart can be defined as
	\begin{equation}
		\mathrm{Char}_{\hm{Q}(\hm{\beta}_0,\hm{y}_0)} (\hm{R}_0)=\hm{R}_0+\sum_{j=1}^N\left\{\Pi_{j,+}\big(\hm{Q}_{\hm{R}_0}\big)\left(\mathfrak{R}(\hm{R}_0)\right) -\Pi_{j,+}\big(\hm{Q}(\hm{\beta}_0,\hm{y}_0)\big)\left(\hm{R}_0\right)\right\}
	\end{equation}
	where $\mathfrak{R}(\hm{R}_0)$ and $\mathfrak{R}(\hm{R}_0)$ are constructed using the implicit function theorem in the neighborhood of $\hm{Q}(\hm{\beta}_0,\hm{y}_0)$. Since both $\mathfrak{R}(\hm{R}_0)$ and $\mathfrak{R}(\hm{R}_0)$ are Lipschitz in $\hm{R}_0$, the chart above is also Lipschitz in $\hm{R}_0$. Notice we can also compute the differential
	$\mathrm{d}\mathrm{Char}_{\hm{Q}(\hm{\beta}_0,\hm{y}_0)}(0)\neq 0$. Therefore, in a small enough neighborhood, the inverse of $\mathrm{Char}_{\hm{Q}(\hm{\beta}_0,\hm{y}_0)}$ exists and is also Lipschitz.
	
	To prove the compatibility of local charts for $\mathfrak{N}$, consider two fixed multi-solitons $\hm{Q}(\hm{\beta}_1,\hm{y}_1)$ and $\hm{Q}(\hm{\beta}_2,\hm{y}_2)$
	and  local charts $\mathrm{Char}_{\hm{Q}(\hm{\beta}_1,\hm{y}_1)}$ and $\mathrm{Char}_{\hm{Q}(\hm{\beta}_2,\hm{y}_2)}$, for $i=1,2$
	\begin{equation}
		\mathrm{Char}_{\hm{Q}(\hm{\beta}_i,\hm{y}_i)} (\hm{R}_0)=\hm{R}_0+\sum_{j=1}^N\left\{\Pi_{j,+}\big(\hm{Q}_{i,\hm{R}_0}\big)\left(\mathfrak{R}_i(\hm{R}_0)\right) -\Pi_{j,+}\big(\hm{Q}(\hm{\beta}_i,\hm{y}_i)\big)\left(\hm{R}_0\right)\right\}
	\end{equation}
	where $\hm{Q}_{i,\hm{R}_0}$ and $\mathfrak{R}_i(\hm{R}_0)$ are constructed via the implicit function theorem as above in the neighborhood of $\hm{Q}(\hm{\beta}_i,\hm{y}_i)$.  
	%f1.R0/ WD .I C P C.WR0 /  P C.W1//R0;
	%f2.R0/ WD .I C P C.WR0 /  P C.W2//R0:
	These functions extend to a small neighborhood of $\mathfrak{N}$ in $\mathcal{H}$. They are Lipschitz
	and their differentials are not singular at $0$ , so they are locally invertible and the inverses are also Lipschitz. So in the overlapped region $\mathrm{Char}_{\hm{Q}(\hm{\beta}_1,\hm{y}_1)}\circ \mathrm{Char}_{\hm{Q}(\hm{\beta}_2,\hm{y}_2)}^{-1}$ is locally Lipschitz which results in the  compatibility of local
	charts.
	%Thus f1 ? f 1
	%2 is real analytic, showing the compatibility of local
	%charts for S.?
	
	Therefore, we have constructed a centre-stable manifold around the well-separated multi-soliton family.
	\begin{thm}\label{thm:globalmani}
		Fixed a natural number $N$, given $\delta>0$, there exists  $\rho>0$ large, such that there exists a codimension $N$ Lipschitz centre-stable manifold $\mathcal{N}$ around the well-separated multi-soliton family $\mathfrak{S}_{\delta,\rho}$ which is invariant for $t\geq0$ such that for any choice of initial data $\hm{\psi}(0)\in\mathcal{N}$,
		the solution $\hm{\psi}(t)$  to \eqref{eq:dynkg} with initial data $\hm{\psi}(0)$
		exists globally, and it scatters to the multi-soliton family there
		exist $\beta_{j}\in\mathbb{R}^{3}$, paths $y_{j}(t)\in\mathbb{R}^{3}$
		and $\boldsymbol{\psi}_{+}\in\mathcal{H}$ with the property that $\dot{y}_{j}(t)\rightarrow\beta_{j}$
		and
		\[
		\lim_{t\rightarrow\infty}\text{\ensuremath{\left\Vert \boldsymbol{\psi}(t)-\boldsymbol{Q}(\hm{\beta},\hm{y}(t))-e^{JH_{0}t}\boldsymbol{\psi}_{+}\right\Vert _{\mathcal{H}}}}=0.
		\]
	\end{thm}
	\begin{rem}
		Notice that by Remark \ref{rem:sizebeta}, the size of the centre-stable manifold above is not uniform.
	\end{rem}
	\begin{rem}
		It is possible to show that the manifold constructed above is real analytic for the cubic Klein-Gordon equation. We refer to Beceanu \cite{Bec2} for the real analytic manifold constructed around the single soliton for the cubic Schr\"odinger equation in $\mathbb{R}^3$.
	\end{rem}
	It is clear that by construction, for any initial data $\hm{\Psi}\in\mathcal{N}$, the corresponding solution $\hm{\Psi}(t)$ enjoys the orbital stability since it scatters to the multi-soliton family. The converse is also sure which is the following corollary.
	\begin{cor}\label{cor:orbitalmanifold}
		For every $\delta > 0$ there exist $\rho, \eta > 0$ such that
		the following holds. Let the initial parameters $(y_j^{in}, \beta_j^{in})_{j=1}^J$
		satisfy the separation condition in the sense of Definition \ref{def:sep}, and let
		\begin{equation}
			\|\bs \psi_0 - \bs Q( \hm{\beta}^{in}, \hm{y}^{in})\|_{\mathcal{H}} \leq \eta.
		\end{equation}
		Suppose that the solution $\bs \psi$ to \eqref{eq:dynkg} with initial data $\hm{\psi}_0$  stays close in  a neighborhood of multi-soliton family:
		\begin{equation}
			\sup_{t\in\mathbb{R}^{+}}\inf_{\hm{\beta}\in\mathbb{R}^{3N},\hm{y}\in\mathbb{R}^{3N}}\left\Vert \boldsymbol{\psi}(t)-\boldsymbol{Q}(\hm{\beta},\hm{y})\right\Vert _{\mathcal{H}}\lesssim\eta,
		\end{equation}
		then $\bs \psi(t)\in\mathcal{N}$.
	\end{cor}
	\begin{proof}
		Using the modulation computations, the first part of Lemma \ref{lem:mod}, we can find Lipschitz functions  $\hm{y}(t)$ and $\hm{\beta}(t)$ such that
		\[
		\sup_{t\in\mathbb{R}^{+}}\left\Vert \boldsymbol{\psi}(t)-\boldsymbol{Q}(\hm{y}(t),\hm{\beta}(t))\right\Vert _{\mathcal{H}}\lesssim\eta
		\]
		and setting
		\begin{equation}
			\hm{u}(t):=\boldsymbol{\psi}(t)-\boldsymbol{Q}(\hm{y}(t),\hm{\beta}(t)),
		\end{equation}
		one has for $m=1,2,3$ and $j=1,\ldots,N$
		\begin{equation}
			\left\langle \alpha_{m,\beta_{j}(t)}^{1}(\cdot-y_{j}(t)),\hm{u}(t)\right\rangle =\left\langle \alpha_{m,\beta_{j}(t)}^{0}(\cdot-y_{j}(t),\hm{u}(t)\right\rangle =0
		\end{equation}
		and $\sup_{t\in\mathbb{R}^{+}}\left\Vert \boldsymbol{u}(t)\right\Vert _{\mathcal{H}}\lesssim\eta$.

		From the decomposition above and using the equation for $\bs \psi(t)$, we obtain
		\begin{align}
			\frac{d}{dt}\hm{u}(t) & =JH(t)\hm{u}(t)+\mathcal{I}(Q)+\mathfrak{I}_{1}(Q^{2},u)+\mathfrak{I}_{2}(Q,u^{2})+\hm{F}(u)\\
			& -\dot{\hm{\beta}}(t)\partial_{\hm{\beta}}\hm{Q}(\hm{\beta}(t),\hm{y}(t))-\left(\dot{\hm{y}}(t)-\hm{\beta}(t)\right)\partial_{\hm{y}}\hm{Q}(\hm{\beta}(t),\hm{y}(t))\nonumber\\
			& =:JH(t)\hm{u}(t)+\mathcal{I}(Q)+\mathfrak{I}(Q,u)+\hm{F}(u)+\text{Mod}'(t)\nabla_{M}\hm{Q}(\hm{\beta}(t),\hm{y}(t))\nonumber\\
			& =:JH(t)\hm{u}(t)+\mathcal{W}(x,t).\nonumber
		\end{align}
		From the notations from Definition \ref{def:sep-1} associated to the
		path $(\hm{\beta}(t),\hm{y}(t))$,  one can decompose $\hm{u}(t)$  as
		\[
		\hm{u}(t)=\pi_{c}(t)\hm{u}(t)+\text{\ensuremath{\pi_{+}}}(t)\hm{u}(t)+\pi_-(t)\hm{u}(t)
		\] where by construction, $\pi_0(t)\hm{u}(t)=0,\,\forall t.$
		
		Denote \begin{equation}
			a_{j}^{\pm}(t)=\left\langle \alpha_{\beta_{j}(t)}^{\pm}(\cdot-y_{j}(t),\hm{u}(t)\right\rangle.
		\end{equation}
		%	Moreover, one also knows that
		%	\[
		%	\left\Vert \text{\ensuremath{\pi_{u}}}(t)\hm{u}(t)\right\Vert _{\mathcal{H}}\sim\sum_{j=1}^{J}\left|\left\langle \alpha_{\beta_{j}(t)}^{+}(\cdot-y_{j}(t),\hm{u}(t)\right\rangle \right|=:\sum_{j=1}^{N}\left|a_{j}^{+}(t)\right|.
		%	\]		
		Since $\sup_{t\in\mathbb{R}^+}\left\Vert \hm{u}(t)\right\Vert _{\mathcal{H}}\lesssim\eta$, it gives the unique choice for the initial data for $a_j^{+}(0)$ and leads to the following equation:
		\begin{align}
			a_{j}^{+}(t) & =-\int_{t}^{\infty}\exp\left(\int_{s}^{t}\frac{\nu}{\gamma_{j}(\tau)}\,d\tau\right)\left\langle \alpha_{\beta_{j}(s)}^{+}(\cdot-y_{j}(s),\mathcal{W}(\cdot,s)\right\rangle \,ds\\
			& -\sum_{k\neq j}\int_{t}^{\infty}\exp\left(\int_{s}^{t}\frac{\nu}{\gamma_{j}(\tau)}\,d\tau\right)\left\langle \alpha_{\beta_{j}(s)}^{+}\big(\cdot-y_{j}(s)\big),\mathcal{V}_{\beta_{k}(s)}\big(\cdot-y_{k}(s)\big)\hm{u}(s)\right\rangle \,ds\nonumber\\
			& -\int_{t}^{\infty}\exp\left(\int_{s}^{t}\frac{\nu}{\gamma_{j}(\tau)}\,d\tau\right)\left\langle \left(y'_{j}(s)-\beta_{j}(s)\right)\cdot\nabla\alpha_{\beta_{j}(s)}^{\text{+}}(\cdot-y_{j}(s)),\hm{u}(s)\right\rangle \,ds\nonumber\\
			& -\int_{t}^{\infty}\exp\left(\int_{s}^{t}\frac{\nu}{\gamma_{j}(\tau)}\,d\tau\right)\left\langle \beta_{j}'(s)\partial_{\beta}\alpha_{\beta_{j}(s)}^{+}(\cdot-y_{j}(s)),\hm{u}(s)\right\rangle \,ds.\nonumber
		\end{align}
		%Then from \eqref{eq:uorbit}, \eqref{eq:aorbit} and \eqref{eq:orbortho}, we now conclude $(\hm{\beta}(t),\hm{y}(t),\hm{u}(t))$ is actually the fixed point of the map \eqref{eq:inputoutput} constructed in the previous proof. Then $(\hm{\beta}(t),\hm{y}(t),\hm{u}(t))$ has the same behavior as 
		%Theorem \ref{thm:manifold}.
		%Similar computations for the one potential setting appear als Beceanu, Schlag Krieger-Schlag. 
		For the stable modes, we have
		{\small\begin{align}
				a_{j}^{-}(t) & =a_{j}^{-}(0)\exp\left(-\int_{0}^{t}\frac{\nu}{\gamma_{j}(\tau)}\,d\tau\right)-\int_{0}^{t}\exp\left(-\int_{s}^{t}\frac{\nu}{\gamma_{j}(\tau)}\,d\tau\right)\left\langle \alpha_{\beta_{j}}^{-}(\cdot-\beta_{j}s-x_{j}),\mathcal{W}(\cdot,s)\right\rangle \,ds\label{eq:stablecsinv}\\
				& +\sum_{k\neq j}\int_{0}^{t}\exp\left(-\int_{s}^{t}\frac{\nu}{\gamma_{j}(\tau)}\,d\tau\right)\left\langle \alpha_{\beta_{j}}^{-}(\cdot-\beta_{j}s-x_{j}),\mathcal{V}_{\beta_{k}}\big(\cdot-\beta_{k}s-x_{k}\big)\hm{u}(s)\right\rangle \,ds\nonumber \\
				& +\int_{0}^{t}\exp\left(-\int_{s}^{t}\frac{\nu}{\gamma_{j}(\tau)}\,d\tau\right)\left\langle \left(y'_{j}(s)-\beta_{j}(s)\right)\cdot\nabla\alpha_{\beta_{j}(s)}^{-}(\cdot-y_{j}(s)),\hm{u}(s)\right\rangle \,ds\nonumber \\
				& +\int_{0}^{t}\exp\left(-\int_{s}^{t}\frac{\nu}{\gamma_{j}(\tau)}\,d\tau\right)\left\langle \beta_{j}'(s)\partial_{\beta}\alpha_{\beta_{j}(s)}^{-}(\cdot-y_{j}(s)),\hm{u}(s)\right\rangle \,ds.\nonumber 
		\end{align}}
		By the same argument as the orbital stability implying the asymptotic stability, we can write
		\begin{equation}
			\boldsymbol{\psi}(0)=\boldsymbol{Q}(\hm{y}(0),\hm{\beta}(0))+\hm{R}_0+\tilde{\hm{\Phi}}\cdot\hm{\mathcal{Y}}^+(\hm{y}(0),\hm{\beta}(0))
		\end{equation}for $\hm{R}_0\in\mathcal{N}_{L}\left(\boldsymbol{Q}(\hm{y}(0),\hm{\beta}(0)),
	\eta\right)$.
		We can also define
		\begin{equation}
			\title{\boldsymbol{\psi}}(0)=\boldsymbol{Q}(\hm{y}(0),\hm{\beta}(0))+\hm{R}_0+\hm{\Phi}\left(\hm{R}_0;\boldsymbol{Q}(\hm{y}(0),\hm{\beta}(0))\right)\cdot\hm{\mathcal{Y}}^+(\hm{y}(0),\hm{\beta}(0))
		\end{equation} where $\hm{\Phi}\left(\hm{R}_0;\boldsymbol{Q}(\hm{y}(0),\hm{\beta}(0))\right)$ is constructed via Theorem \ref{thm:manifold1}.  By construction, this initial data produces a solution $\tilde{\hm{\Psi}}(t)\in\mathcal{N}$ which scatters to the multi-soliton family.
		
		Notice that both $\large(\tilde{\hm{\beta}}(t), \tilde{\hm{y}}(t),\tilde{\hm{\Psi}}(t)\Large)$ and $\large(\hm{\beta}(t), \hm{y}(t),\hm{\Psi}(t)\Large)$ are fixed points of the map \eqref{eq:inputoutput}. 
		By the same argument as the contraction argument in Theorem \ref{thm:manifold1} and using notations there, we conclude that
		\begin{align}
			\left\Vert \left(\tilde{\hm{\beta}}(\cdot)-\hm{\beta}(\cdot),\tilde{\hm{y}}(\cdot)-\hm{y}(\cdot),\tilde{\hm{\Psi}}(\cdot)-\hm{\Psi}(\cdot)\right)\right\Vert _{\mathcal{A}_{G,\eta}}+\left|\tilde{\hm{\Phi}}-\hm{\Phi}\left(\hm{R}_0;\boldsymbol{Q}(\hm{y}(0),\hm{\beta}(0))\right)\right|\\\lesssim \frac{1}{\kappa^{2}}\big(B\eta+\frac{1}{\delta} e^{-(\frac{ \rho}{2})}\big)\left\Vert \left(\tilde{\hm{\beta}}(\cdot)-\hm{\beta}(\cdot),\tilde{\hm{y}}(\cdot)-\hm{y}(\cdot),\tilde{\hm{\Psi}}(\cdot)-\hm{\Psi}(\cdot)\right)\right\Vert _{\mathcal{A}_{G,\eta}}.
		\end{align}
		Taking $\eta$ small and $\rho$ large, one has
		\begin{equation}
			\left(\tilde{\hm{\beta}}(\cdot),\tilde{\hm{y}}(\cdot),\tilde{\hm{\Psi}}(\cdot)\right)=\left(\hm{\beta}(\cdot),\hm{y}(\cdot),\hm{\Psi}(\cdot)\right)
		\end{equation}which implies $\hm{\Psi}(t)\in\mathcal{N}$ as desired.
	\end{proof}

	\section{Classification of pure multi-solitons}\label{sec:cla}
	In this section, we analyze pure multi-solitons which are special solutions scattering the well-separated soliton family with zero radiation fields.  Due to the exponential interaction among solitons, one should expect a better behavior of the error term. With Strichartz estimates, it will be shown that actually the error term decays exponentially and modulation parameters have refined information.  
	\begin{thm}\label{thm:exp}
		Suppose $\hm{\psi}(t)$ is a pure multi-soliton in the sense of Definition \ref{def:puremulti}. Then actually there exists $\hm{x}_0\in\mathbb{R}^{3N}$ such that one has
		\begin{equation}\label{eq:expconv}
			\text{\text{\ensuremath{\left\Vert \boldsymbol{\psi}(t)-\boldsymbol{Q}(\hm{\beta},\hm{\beta}t+\hm{x}_{0})\right\Vert _{\mathcal{H}}}}\ensuremath{\lesssim e^{-\rho t}}}
		\end{equation}
		for small $\rho>0$. 
	\end{thm}

	\begin{rem}
		Using the definition from Chen-Jendrej \cite{CJ3}, the theorem above
		tells that every pure multi-soliton solution is an exponential multi-soliton
		solution.
	\end{rem}
	
	\begin{proof}
		Denote $\delta:=\inf_{i\neq k}|\beta_j-\beta_k|$. Given a solution $\hm{\psi}$ satisfying \eqref{eq:puremultidef}, by construction,
		for any $\eta>0$, there exits $T_{\bs \psi}$ such that
		\begin{equation}
			\sup_{t\geq T_{\bs \psi}} {\left\Vert \boldsymbol{\psi}(t)-\boldsymbol{Q}(\hm{\beta},\hm{y}_{p}(t))\right\Vert _{\mathcal{H}}}\leq \eta.
		\end{equation} Now we fix a $\eta\ll1$ small enough depending on $\delta$ as in Theorem \ref{thm:orbitasy3}.

		For $t\geq T_{\hm{\psi}}$, by the
		modulation analysis, Lemma \eqref{lem:mod}, we can find a set of modulation parameters $(\hm{\beta}(t),\hm{y}(t))$
		such that one can decompose the original solution as
		\[
		\boldsymbol{\psi}(t)=\boldsymbol{Q}(\hm{\beta}(t),\hm{y}(t))+\boldsymbol{u}(t)
		\]
		with $\boldsymbol{u}$ satisfies orthogonality conditions
		\[
		\left\langle \alpha_{m,\beta_{j}(t)}^{0}(\cdot-y_{j}(t)),\hm{u}(t)\right\rangle =\left\langle \alpha_{m,\beta_{j}(t)}^{1}(\cdot-y_{j}(t)),\hm{u}(t)\right\rangle =0
		\]
		for $m=1,2,3$ and $j=1,\ldots N$. By construction and the implicit function theorem, one also has
		\begin{equation}
			|y_j(t)-y_k(t)|\geq \frac{L}{2}
		\end{equation}where $L$ is from Definition \ref{def:puremulti}.
		\iffalse
		$\left\Vert \hm{u}(t)\right\Vert _{\mathcal{H}}\rightarrow0,$
		and $\hm{\beta}(t)\rightarrow\hm{\beta}$
		as $t\rightarrow\infty$. Taking time large enough, for any $\eta>0$, there exits $T_{\bs \psi}$ such that
		\begin{equation}
			\sup_{t\geq T_{\bs \psi}} \left\Vert \hm{u}(t)\right\Vert _{\mathcal{H}}\leq \eta.
		\end{equation}

		%Using modulation equations again, Lemma \ref{lem:mod}, we also have
		%\begin{equation}
		%   \left\Vert |\hm{y}'(\cdot)-\hm{\beta}(\cdot)|+|\hm{\beta}'(\cdot)|\right\Vert _{L_{t}^{\infty} } \lesssim \eta^2.
		%\end{equation}	

		Without loss of generality, we can assume that
		\begin{equation}
			|\beta_jt+x_{0,j}-(\beta_kt+x_{0,k})|\geq \delta t+L,\, j\neq k
		\end{equation}
		with $\delta$ small and $L$ large. We can always achieve this by restricting onto large time independent of $\bs \psi$. 
		
		By construction, $\left\Vert \hm{u}(t)\right\Vert _{\mathcal{H}}\rightarrow0,$
		and $\hm{\beta}(t)\rightarrow\hm{\beta}$, $\hm{y}(t)\rightarrow\hm{\beta}t+\hm{x}_{0}$
		as $t\rightarrow\infty$. Therefore, for any $\eta>0$, there exits $T_{\bs \psi}$ such that
		\begin{equation}
			\sup_{t\geq T_{\bs \psi}} \left\Vert \hm{u}(t)\right\Vert _{\mathcal{H}}\leq \eta.
		\end{equation}
		\fi
		
		%We will show by a bootstrap argument that actually, one can make $T_{\bs \psi}=0$ independent of $\bs \psi (t)$.
		Using the equation for $\hm{\psi}(t)$, the equation for $\hm{u}(t)$ is again given by
		\begin{align*}
			\frac{d}{dt}\hm{u}(t) & =JH(t)\hm{u}(t)+\mathcal{I}(Q)+\mathfrak{I}_{1}(Q^{2},u)+\mathfrak{I}_{2}(Q,u^{2})+\hm{F}(u)\\
			& -\dot{\hm{\beta}}(t)\partial_{\hm{\beta}}\hm{Q}(\hm{\beta}(t),\hm{y}(t))-\left(\dot{\hm{y}}(t)-\hm{\beta}(t)\right)\partial_{\hm{y}}\hm{Q}(\hm{\beta}(t),\hm{y}(t))\\
			& =:JH(t)\hm{u}(t)+\mathcal{I}(Q)+\mathfrak{I}(Q,u)+\hm{F}(u)+\text{Mod}'(t)\nabla_{M}\hm{Q}(\hm{\beta}(t),\hm{y}(t))\\
			& =:JH(t)\hm{u}(t)+\mathcal{W}(x,t).
		\end{align*}
		Following a similar argument to Theorem \ref{thm:orbitasy3}, we write
		\[
		\hm{u}(t)=\pi_{+}(t)\hm{u}(t)+\pi_{-}(t)\hm{u}(t)
		%+\pi_{s}(t)\hm{u}(t)
		+\text{\ensuremath{\pi_{c}}}(t)\hm{u}(t).
		\]
		%	We also know
		%	\[
		%	\left\Vert \text{\ensuremath{\pi_{u}}}(t)\hm{u}(t)\right\Vert _{\mathcal{H}}\sim\sum_{j=1}^{N}\left|\left\langle \alpha_{\beta_{j}(t)}^{+}(\cdot-y_{j}(t),\hm{u}(t)\right\rangle \right|=:\sum_{j=1}^{N}\left|a_{j}^{+}(t)\right|.
		%	\]
		%		and
		%		\[
		%		\left\Vert \text{\ensuremath{\pi_{S}}}(t)\hm{u}(t)\ri%ght\Vert _{\mathcal{H}}\sim\sum_{j=1}^{N}\left|\left\langle \alpha_{\beta_{j}(t)}^{-}(\cdot-y_{j}(t),\hm{u}(t)\right\rangle \right|=:\sum_{j=1}^{N}\left|a_{j}^{-}(t)\right|
		%		\]
		Since $\left\Vert \hm{u}(t)\right\Vert _{\mathcal{H}}\rightarrow0$,
		one still has the stabilization conditions
		\begin{align*}
			a_{j}^{+}(t) & =-\int_{t}^{\infty}\exp\left(\int_{s}^{t}\frac{\nu}{\gamma_{j}(\tau)}\,d\tau\right)\left\langle \alpha_{\beta_{j}(s)}^{+}(\cdot-y_{j}(s),\mathcal{W}(\cdot,s)\right\rangle \,ds\\
			& -\sum_{k\neq j}\int_{t}^{\infty}\exp\left(\int_{s}^{t}\frac{\nu}{\gamma_{j}(\tau)}\,d\tau\right)\left\langle \alpha_{\beta_{j}(s)}^{+}\big(\cdot-y_{j}(s)\big),\mathcal{V}_{\beta_{k}(s)}\big(\cdot-y_{k}(s)\big)\hm{u}(s)\right\rangle \,ds\\
			& -\int_{t}^{\infty}\exp\left(\int_{s}^{t}\frac{\nu}{\gamma_{j}(\tau)}\,d\tau\right)\left\langle \left(y'_{j}(s)-\beta_{j}(s)\right)\cdot\nabla\alpha_{\beta_{j}(s)}^{\text{+}}(\cdot-y_{j}(s)),\hm{u}(s)\right\rangle \,ds\\
			& -\int_{t}^{\infty}\exp\left(\int_{s}^{t}\frac{\nu}{\gamma_{j}(\tau)}\,d\tau\right)\left\langle \beta_{j}'(s)\partial_{\beta}\alpha_{\beta_{j}(s)}^{+}(\cdot-y_{j}(s)),\hm{u}(s)\right\rangle \,ds.
		\end{align*}
		Again, for the stable modes, we have
		{\small\begin{align}
				a_{j}^{-}(t) & =a_{j}^{-}(0)\exp\left(-\int_{0}^{t}\frac{\nu}{\gamma_{j}(\tau)}\,d\tau\right)-\int_{0}^{t}\exp\left(-\int_{s}^{t}\frac{\nu}{\gamma_{j}(\tau)}\,d\tau\right)\left\langle \alpha_{\beta_{j}}^{-}(\cdot-\beta_{j}s-x_{j}),\mathcal{W}(\cdot,s)\right\rangle \,ds\label{eq:stableclass}\\
				& +\sum_{k\neq j}\int_{0}^{t}\exp\left(-\int_{s}^{t}\frac{\nu}{\gamma_{j}(\tau)}\,d\tau\right)\left\langle \alpha_{\beta_{j}}^{-}(\cdot-\beta_{j}s-x_{j}),\mathcal{V}_{\beta_{k}}\big(\cdot-\beta_{k}s-x_{k}\big)\hm{u}(s)\right\rangle \,ds\nonumber \\
				& +\int_{0}^{t}\exp\left(-\int_{s}^{t}\frac{\nu}{\gamma_{j}(\tau)}\,d\tau\right)\left\langle \left(y'_{j}(s)-\beta_{j}(s)\right)\cdot\nabla\alpha_{\beta_{j}(s)}^{-}(\cdot-y_{j}(s)),\hm{u}(s)\right\rangle \,ds\nonumber \\
				& +\int_{0}^{t}\exp\left(-\int_{s}^{t}\frac{\nu}{\gamma_{j}(\tau)}\,d\tau\right)\left\langle \beta_{j}'(s)\partial_{\beta}\alpha_{\beta_{j}(s)}^{-}(\cdot-y_{j}(s)),\hm{u}(s)\right\rangle \,ds.\nonumber 
		\end{align}}By the same argument as Theorem \ref{thm:orbitasy3}, the unique solution $\big(\hm{\beta (t),\hm{y}(t),\hm{u}(t)}\big)$ satisfies the following estimates: $\forall t\geq T_{\bs \psi}$, one has 
		\[
		\left\Vert \hm{\beta}'\right\Vert _{L_{t}^{1}\bigcap L_{t}^{\infty}[t,\infty]}+\left\Vert \hm{y}'-\hm{\beta}\right\Vert _{L_{t}^{1}\bigcap L_{t}^{\infty}[t,\infty]}\lesssim e^{-\frac{\delta}{2}t}+\left\Vert \hm{u}\right\Vert _{\mathcal{S}_{\mathcal{H}}[t,\infty]}^{2}
		\]
		from the modulation equations, see Lemma \ref{lem:mod} and
		\begin{equation}
			\left\Vert\hm{u}\right\Vert _{\mathcal{S}_{\mathcal{H}}[t,\infty]}\lesssim\left\Vert \mathcal{I}(Q)\right\Vert _{\mathcal{S}_{\mathcal{H}}[t,\infty]}\lesssim e^{-\frac{\delta}{2}t}.\label{eq:centerdecay}
		\end{equation} by Strichartz estimates
		where we absorbed higher order terms to the LHR due to the smallness
		of $\left\Vert \hm{u}\right\Vert _{\mathcal{S}_{\mathcal{H}}}$.
		The exponential term $ e^{-\frac{\delta}{2}t}$ is from the interaction of solitons
		$\left\Vert \mathcal{I}(Q)\right\Vert _{\mathcal{S}^{*}_{\mathcal{H}}[t,\infty]}\lesssim  e^{-\frac{\delta}{2}t}$.

		From the
		modulation equation, we get
		\[
		\left|\beta_{j}(t)-\beta_{j}\right|\lesssim\frac{1}{2\rho}e^{-2\rho t}
		\]
		and
		\[
		\left|y'_{j}(t)-\beta_{j}(t)\right|\lesssim e^{-2\rho t}.
		\]
		It follows that
		\[
		\left|(y_{j}(t)-\beta_{j}t)'\right|=\left|y_{j}'(t)-\beta_{j}(t)+\beta_{j}(t)-\beta_{j}\right|\lesssim\frac{1}{2\rho}e^{-2\rho t}.
		\]
		Therefore, one can find $x_j\in\mathbb{R}^3$ such that
		\[
		\left|y_{j}(t)-\beta_{j}t-x_{j}\right|\lesssim\left(\frac{1}{2\rho}\right)^{2}e^{-2\rho t}.
		\]
		Hence, we obtain
		\begin{align*}
			\text{\ensuremath{\left\Vert \boldsymbol{\psi}(t)-\boldsymbol{Q}(\hm{\beta},\hm{\beta}t+\hm{x}_{0})\right\Vert _{\mathcal{H}}}=\ensuremath{\left\Vert \boldsymbol{\psi}(t)-\boldsymbol{Q}(\hm{y}(t),\hm{\beta}(t))+\boldsymbol{Q}(\hm{y}(t),\hm{\beta}(t))-\boldsymbol{Q}(\hm{\beta},\hm{\beta}t+\hm{x}_{0})\right\Vert _{\mathcal{H}}}}\\
			\left\Vert \hm{u}(t)\right\Vert _{\mathcal{H}}+\left\Vert \boldsymbol{Q}(\hm{y}(t),\hm{\beta}(t))-\boldsymbol{Q}(\hm{\beta},\hm{\beta}t+\hm{x}_{0})\right\Vert _{\mathcal{H}}\ensuremath{\lesssim e^{-\rho t}}
		\end{align*}
		as claimed. Note that here $\rho$ is independent of $\hm{u}$. One
		concludes that every pure multi-soliton is an exponential
		multi-soliton.

	\end{proof}
	Recall that  for any  given $|\beta_{j}|<1$ such that $\beta_{j}\neq\beta_{k}$
	for $j\neq k$, and arbitrary $x_{j}\in\mathbb{R}^{3}$, there exists
	a pure multi-soliton $\psi$ to \eqref{eq:nkg} such that
	\begin{equation}
		\hm{\psi}\rightarrow\boldsymbol{Q}(\hm{\beta},\hm{\beta}t+\hm{x}_{0})=\sum_{j=1}^{N}\boldsymbol{Q}_{\beta_{j}}\left(x-\beta_{j}t-x_{j}\right),\,\text{as}\,t\rightarrow\infty.\label{eq:asyMuti}
	\end{equation}
	The existence of the solution above was shown by C\^ote-Mu\~noz \cite{CMu} where a weak convergence and Brouwer's fixed point theorem are applied. 
	%	We again switch to the Hamiltonian formalism. Denote $$\boldsymbol{Q}(\hm{\beta},\hm{\beta}t+\hm{x}_{0})=\sum_{j=1}^{N}\boldsymbol{Q}_{\beta_{j}}\left(x-\beta_{j}t-x_{j}\right).$$
	From Theorem \ref{thm:exp}, we conclude that every pure multi-soliton in the sense of Definition \ref{def:puremulti} actually converges $\boldsymbol{Q}(\hm{\beta},\hm{\beta}t+\hm{x}_{0})$ exponentially.
	
	With Strichartz estimates  in hand, the remaining part  of this section is to
	classify all pure multi-solitons in the energy space and find all solution $\bs \psi $ satisfying \eqref{eq:expconv}.   We will see later on, our analysis can give an alternative proof of the existence of pure multi-solitons via the Picard iteration.
	\iffalse
	Our first result in this section is that any pure multi-soliton actually converges to the multi-soliton exponentially.
	\begin{thm}
		Let $\hm{\beta}=\text{(\ensuremath{\beta_{1},\ldots,\beta_{N})} \ensuremath{\in\mathbb{R}^{3N}}}$
		satisfy $|\beta_{j}|<1$ and $\beta_{j}\neq\beta_{k}$ for $j\neq k$
		and take an arbitrary $\hm{x}_{0}=(x_{1},\ldots,x_{N})\in\mathbb{R}^{3N}$.
		Suppose $\hm{\psi}(t)$ is a solution to \eqref{eq:dynkg}
		satisfying
		\begin{equation}
			\lim_{t\rightarrow\infty}\text{\ensuremath{\left\Vert \boldsymbol{\psi}(t)-\boldsymbol{Q}(\hm{\beta},\hm{\beta}t+\hm{x}_{0})\right\Vert _{\mathcal{H}}}}=0\label{eq:puremuti}
		\end{equation}
		then actually one has
		\[
		\text{\text{\ensuremath{\left\Vert \boldsymbol{\psi}(t)-\boldsymbol{Q}(\hm{\beta},\hm{\beta}t+\hm{x}_{0})\right\Vert _{\mathcal{H}}}}\ensuremath{\lesssim e^{-\rho t}}}
		\]
		for small $\rho>0$. 
	\end{thm}

	%\newpage 
	Next, we classify pure multi-solitons 
	\fi
	
	The analysis here be in spirit of our work \cite{CJ3} where we obtained the uniqueness of exponential multi-kink via energy estimates.  From the theorem above, we already know that every pure multi-soliton is an exponential multi-soliton. So it suffices for us to do the classification analysis in this class.  But we can not use energy estimates in general since not all subcritical nonlinear terms  can have  Lipschitz estimates in the energy norm in $3$d, and for this reason, we choose to use weighted Strichartz estimates. (In $1$d, via Sobolev's embedding, the nonlinear term is Lipschitz using energy estimates.)
	
	%The second approach would be of the style of the construction of the centre-stable manifold. In this approach, we do not need the exponential decay information. Moreover, the exponential convergence will be a by-product of the analysis. This approach also gives an alternative way to construct pure multi-solitons.
	\begin{thm}\label{thm:puremanifold}
		For fixed $\hm{\beta}=\text{(\ensuremath{\beta_{1},\ldots,\beta_{N})} \ensuremath{\in\mathbb{R}^{3N}}}$
		satisfy $|\beta_{j}|<1$ and $\beta_{j}\neq\beta_{k}$ for $j\neq k$
		and any given $\hm{x}_{0}=(x_{1},\ldots,x_{N})\in\mathbb{R}^{3N}$, the set
		of solution $\hm{\psi}$ satisfying \eqref{eq:expconv} is a dimension
		$N$ Lipschitz manifold.
	\end{thm}
	
	\begin{proof}

		%	\noindent{\bf First approach:}		
		%First idea to use Chen-Jendrej energy type argument with exponential
		%weights in time after taking stable/unstable modes into account. The
		%reason this can work is our have already conclude that every pure
		%multi-soliton is an exponential multi-soliton. Then
		%we just need to classify solutions inside the class of exponential
		%multi-solitons.
		\iffalse
		For $\eta>$ small, we set 
		\[
		\chi_{j}(t,x):=\chi\left(\eta(x-y_{j}(t))\right)
		\]
		and define
		\[
		Q(t;\hm{u}_{0},\hm{u}_{0}):=\frac{1}{2}\int_{\mathbb{R}^{3}}\left((\dot{u}_{0})^{2}+2\sum_{j=1}^{N}\chi_{j}(t)\dot{u}_{0}\left(\beta_{j}(t)\cdot\nabla u_{0}\right)+|\nabla u_{0}|^{2}+(1+V(t))u_{0}^{2}\right)\,dx.
		\]
		Then we have
		\begin{align*}
			Q(t;\hm{u}_{0},\hm{u}_{0}) & \geq c\left\Vert \hm{u}_{0}\right\Vert _{\mathcal{H}}-\frac{1}{c}\sum_{j=1}^{N}\left(\left\langle \alpha_{\beta_{j}(t)}^{+}(\cdot-y_{j}(t),\hm{u}_{0}\right\rangle ^{2}+\left\langle \alpha_{\beta_{j}(t)}^{-}(\cdot-y_{j}(t),\hm{u}_{0}\right\rangle ^{2}\right)\\
			& \frac{1}{c}\sum_{j=1}^{N}\sum_{m=1}^{3}\left(\left\langle \alpha_{m,\beta_{j}(t)}^{0}(\cdot-y_{j}(t),\hm{u}_{0}\right\rangle ^{2}+\left\langle \alpha_{m,\beta_{j}(t)}^{1}(\cdot-y_{j}(t),\hm{u}_{0}\right\rangle ^{2}\right)
		\end{align*}
		\fi
		Due to the exponential convergence of modulation parameters, we can
		decompose any pure multi-soliton $\hm{\psi}(t)$ as
		\[
		\boldsymbol{\psi}(t)=\boldsymbol{Q}(\hm{\beta},\hm{\beta}t+\hm{x}_{0})+\hm{v}(t).
		\]
		We have shown in Theorem \ref{thm:exp}  that  $\hm{v}(t)\rightarrow0$ exponentially with a rate $\rho$ only depending on prescribed constants.  
		
		We write the equation of $\hm{v}$ as
		\begin{align*}
			\frac{d}{dt}\hm{v}(t) & =JH_{0}(t)\hm{v}(t)+\mathcal{I}(Q)+\mathfrak{I}_{1}(Q^{2},v)+\mathfrak{I}_{2}(Q,v^{2})+\hm{F}(v)\\
			& =:JH_{0}(t)\hm{v}(t)+\mathcal{W}(x,t)
		\end{align*}
		%	\[
		%	\left|\left\langle \alpha_{m,\beta_{j}}^{0}(\cdot-\beta_{j}t-x_{0,j}),\hm{v}(t)\right\rangle \right|+\left|\left\langle \alpha_{m,\beta_{j}(t)}^{1}(\cdot-\beta_{j}t-x_{0,j}),\hm{v}(t)\right\rangle \right|.
		%	\]
		where $H_0(t)$ is defined as the corresponding time-depend operator with respect to the path $(\hm{\beta},\hm{\beta}t+\hm{x}_{0})$.
		Then with respect to the same path, we decompose $\hm{v}$ as
		\begin{equation}
			\hm{v}(t)=\pi_{0,0}(t)\hm{v}(t)+\pi_{0,+}(t)\hm{v}(t)+\pi_{0,-}(t)\hm{v}(t)+\pi_{0,c}(t)\hm{v}(t)
		\end{equation}
		where $\pi_{0,\ell},\,\ell=0,c,\pm$ are projections with respect to $(\hm{\beta},\hm{\beta}t+\hm{x}_{0})$.
		Using the Duhamel formula, one can write
		\begin{equation}\label{eq:vduham}
			\hm{v}(t)=-\int_{t}^{\infty}\mathcal{T}_{0}(t,s)\mathcal{W}(\cdot,s)\,ds
		\end{equation}
		where $\mathcal{T}_0(t,s)$ is the propagator with respect the path $(\hm{\beta},\hm{\beta}t+\hm{x}_{0})$, and  for discrete modes, one has	
		\begin{align}
			a_{j}^{+}(t) &= \left\langle \alpha_{\beta_{j}}^{+}(\cdot-\beta_{j}t-x_{j}),\hm{v}(t)\right\rangle =-\int_{t}^{\infty}\exp\left((t-s)\frac{\nu}{\gamma_{j}}\right)\left\langle \alpha_{\beta_{j}}^{+}(\cdot-\beta_{j}s-x_{j}),\mathcal{W}(\cdot,s)\right\rangle \,ds\label{eq:unstableclass}\\
			& -\sum_{k\neq j}\int_{t}^{\infty}\exp\left((t-s)\frac{\nu}{\gamma_{j}}\right)\left\langle \alpha_{\beta_{j}}^{+}\big(\cdot-\beta_{j}s-x_{j}\big),\mathcal{V}_{\beta_{k}}\big(\cdot-\beta_{k}s-x_{k}\big)\hm{v}(s)\right\rangle \,ds\nonumber
		\end{align}
		
		\begin{align}
			a_{j}^{-}(t) & = \left\langle \alpha_{\beta_{j}}^{-}(\cdot-\beta_{j}t-x_{j}),\hm{v}(t)\right\rangle =a_{j}^{-}(0)\exp\left(-t\frac{\nu}{\gamma_{j}}\right)\label{eq:stableclass2}\\&-\int_{0}^{t}\exp\left((s-t)\frac{\nu}{\gamma_{j}}\right)\left\langle \alpha_{\beta_{j}}^{-}(\cdot-\beta_{j}s-x_{j}),\mathcal{W}(\cdot,s)\right\rangle \,ds\nonumber\\
			& -\sum_{k\neq j}\int_{0}^{t}\exp\left((s-t)\frac{\nu}{\gamma_{j}}\right)\left\langle \alpha_{\beta_{j}}^{-}(\cdot-\beta_{j}s-x_{j}),\mathcal{V}_{\beta_{k}}\big(\cdot-\beta_{k}s-x_{k}\big)\hm{v}(s)\right\rangle \,ds\nonumber.
		\end{align}
		Our goal is to show that for any given set $\left(a_{j}^{-}(0)\right)_{j=1}^{N}$, the system
		above has a unique solution. %Therefore, we find a dimensional $J$
		%manifold such that the solutions from this manifold converges to the
		%desired purely multi-soliton.

		We perform the fixed point analysis using
		weighted  norms. For a small fixed $\varrho>0$, we define
		\[
		\left\Vert \hm{v}\right\Vert _{GS_{\varrho}}=\sup_{t\geq0}e^{\varrho t}\left\Vert \hm{v}(t)\right\Vert _{S_{\mathcal{H}}[t,\infty)},\,\left\Vert \hm{v}\right\Vert _{GS^{*}_{\varrho}}=\sup_{t\geq0}e^{\varrho t}\left\Vert \hm{v}(t)\right\Vert _{S^{^*}_{\mathcal{H}}[t,\infty)}.
		\]
		For a function of $t$, we define
		\[
		\left\Vert f\right\Vert _{G_\varrho L^P_t}:=\sup_{t\in\mathbb{R}^+}e^{-\varrho t}\left\Vert f(\cdot) \right\Vert_{L^p[t,\infty]}.
		\]
		\iffalse
		As in our paper, we will compute
		\[
		\frac{d}{d}Q\left(t;\hm{v}(t),\hm{v}(t)\right)
		\]
		\[
		\frac{d}{dt}\left\langle \alpha_{m,\beta_{j}}^{0}(\cdot-\beta_{j}t-x_{0,j}),\hm{v}(t)\right\rangle 
		\]
		=and
		\[
		\frac{d}{dt}\left\langle \alpha_{m,\beta_{j}}^{1}(\cdot-\beta_{j}t-x_{0,j}),\hm{v}(t)\right\rangle .
		\]
		\fi
		We  claim that for any $\hm{f}(t)\in G\mathcal{S}_{\varrho+\varpi}^{*}$ 
		and $\left(a_{j}^{-}(0)\right)_{j=1}^{N}=\left( \left\langle \alpha_{\beta_{j}}^{-}(\cdot-x_{j}),\hm{v}(0)\right\rangle\right)_{j=1}^N$, there is a unique solution $\hm{v}$
		in $G\mathcal{S}_{\varrho}$ such that
		\begin{equation}\label{eq:claimeq}
			\frac{d}{dt}\hm{v}(t)=JH_{0}(t)\hm{v}(t)+\hm{f}(t)
		\end{equation}
		and it satisfies
		\begin{equation}\label{eq:claim}
			\left\Vert \hm{v}\right\Vert _{GS_{\varrho}}\lesssim\sum_{j=1}^{N}\left|a_{j}^{-}(0)\right|+\left\Vert \hm{f}\right\Vert _{GS_{\varrho+\varpi}^{*}}
		\end{equation}
		for any $\varpi>0$. 
		
		We will show the claim above by two steps: we first perform \emph{a priori} estimates, and then to show the existence and uniqueness of the solution.
		
		\noindent { STEP 1: {\it a priori} estimates. }
		Applying Stricharz estimates from Theorem \eqref{thm:mainthmlinear} to \eqref{eq:vduham} for the interval $[t,\infty)$, one has
		\begin{align}
			\left\Vert \hm{v}(\cdot)\right\Vert _{S_{\mathcal{H}}[t,\infty)}&\lesssim\left\Vert \hm{f}(\cdot)\right\Vert _{\mathcal{S}_{\mathcal{H}}^{*}[t,\infty)}+\left\Vert \pi_{0,+}(\cdot)\hm{v}(\cdot)\right\Vert _{S^{*}_{\mathcal{H}}[t,\infty)}+\left\Vert \pi_{0,-}(\cdot)\hm{v}(\cdot)\right\Vert _{S^{*}_{\mathcal{H}}[t,\infty)}\\
			&+\left\Vert\pi_{0,0}(\cdot)\hm{v}(\cdot)\right\Vert _{S^{*}_{\mathcal{H}}[t,\infty)}.\nonumber
		\end{align}
		Multiplying both sides by $e^{\varrho t}$ and then taking the supremum on both sides, we conclude that
		\begin{align}
			\left\Vert \hm{v}(\cdot)\right\Vert _{GS_{\varrho}}&\lesssim\left\Vert \hm{f}(\cdot)\right\Vert _{G\mathcal{S}_{\rho}^{*}}+\left\Vert \pi_{0,+}\hm{v}(\cdot)\right\Vert _{GS^*_{\varrho}}+\left\Vert \pi_{0,-}(\cdot)\hm{v}(\cdot)\right\Vert _{GS^*_{\varrho}}+\left\Vert\pi_{0,0}(\cdot)\hm{v}(\cdot)\right\Vert_{GS^*_{\varrho}}.
		\end{align}
		Now we estimate last three terms above separately.
		%since $s\geq t$) with
		
		By a similar argument above, using the ODE for the unstable modes,\eqref{eq:unstableclass}, we have
		\begin{align}
			\left\Vert \pi_{0,+}(\cdot)\hm{v}(\cdot)\right\Vert _{GS^{*}_{\varrho}}\sim\sum_{j=1}^{N}\left\Vert a_{j}^{+}(\cdot)\right\Vert _{G_{\varrho}L_{t}^{\infty}\bigcap G_{\varrho}L_{t}^{1}}\lesssim\frac{1}{\delta} e^{-(\frac{L}{2})}\left\Vert \hm{v}\right\Vert _{G\mathcal{S}_{\varrho}}+\left\Vert \hm{f}(\cdot)\right\Vert _{G\mathcal{S}_{\varrho}^{*}}
		\end{align}
		For the stable modes, we pick $\varrho>0$ small enough, say $0<\varrho<\nu$,
		then $\exp\left((s-t)\frac{\nu}{\gamma_{j}}\right)e^{(t-s)\varrho}$
		is still of exponential decay in terms of $t-s$. Then from \eqref{eq:stableclass2}, one has
		{\small	\[
			\left\Vert \pi_{0,-}\hm{v}(\cdot)\right\Vert _{GS_{\varrho}}\sim\sum_{j=1}^{N}\left\Vert a_{j}^{-}(\cdot)\right\Vert _{G_{\varrho}L_{t}^{\infty}\bigcap G_{\varrho}L_{t}^{1}}\lesssim\sum_{j=1}^{N}\left|a_{j}^{-}(0)\right|+\frac{1}{\delta} e^{-(\frac{L}{2})}\left\Vert \hm{v}\right\Vert _{G\mathcal{S}_{\varrho}}+\left\Vert \hm{f}(\cdot)\right\Vert _{G\mathcal{S}_{\varrho}^{*}}.
			\]}
		Finally, for zero modes, we first compute
		{\small	\begin{align*}
				\frac{d}{dt}\left\langle \alpha_{m,\beta_{j}}^{0}(\cdot-\beta_{j}t-x_{0,j}),\hm{v}(t)\right\rangle  & =\left\langle \frac{d}{dt}\alpha_{m,\beta_{j}}^{0}(\cdot-\beta_{j}t-x_{0,j}),\hm{v}(t)\right\rangle +\left\langle \alpha_{m,\beta_{j}}^{0}(\cdot-\beta_{j}t-x_{0,j}),\frac{d}{dt}\hm{v}(t)\right\rangle \\
				& =\left\langle \frac{d}{dt}\alpha_{m,\beta_{j}}^{0}(\cdot-\beta_{j}t-x_{0,j}),\hm{v}(t)\right\rangle +\left\langle \alpha_{m,\beta_{j}}^{0}(\cdot-\beta_{j}t-x_{0,j}),JH_{0}(t)\hm{v}(t)\right\rangle \\
				& +\left\langle \alpha_{m,\beta_{j}}^{0}(\cdot-\beta_{j}t-x_{0,j}),\hm{f}(t)\right\rangle \\
				& =\sum_{j\neq k}\left\langle \alpha_{m,\beta_{j}}^{0}(\cdot-\beta_{j}t-x_{0,j}),\mathcal{V}_{\beta_{k}}(\cdot-\beta_{k}t-x_{0,k})\hm{v}(t)\right\rangle \\
				& +\left\langle \alpha_{m,\beta_{j}}^{0}(\cdot-\beta_{j}t-x_{0,j}),\hm{f}(t)\right\rangle .
		\end{align*}}
		Integrating from $\infty$, using the exponential decay of potentials
		and separation conditions of paths, one has
		\[
		\left|\left\langle \alpha_{m,\beta_{j}}^{0}(\cdot-\beta_{j}t-x_{0,j}),\hm{v}(t)\right\rangle \right|\lesssim\frac{1}{\varrho+\varepsilon}e^{-(\text{\ensuremath{\varrho}}+\varepsilon)t}\left\Vert \hm{v}\right\Vert _{\mathcal{S}_{\varrho}}+\frac{1}{\varrho+\varpi}e^{-(\varrho+\varpi)t}\left\Vert \hm{f}\right\Vert _{G\mathcal{S}_{\varrho+\varpi}^{*}}
		\]
		where $\varepsilon=\frac{L}{2}$ is due to the separation of paths. Then we
		conclude that
		\begin{align*}
			\left|\left\langle \alpha_{m,\beta_{j}}^{0}(\cdot-\beta_{j}t-x_{0,j}),\hm{v}(t)\right\rangle \right|_{G_{\varrho}L_{t}^{\infty}\bigcap G_{\varrho}L_{t}^{1}} & \lesssim\epsilon\left(\frac{1}{\varrho+\varepsilon}+\frac{1}{(\varrho+\varepsilon)^{2}}\right)\left\Vert \hm{v}\right\Vert _{\mathcal{S}_{\varrho}}\\
			& +\left(\frac{1}{\varrho+\varpi}+\frac{1}{(\varrho+\varpi)^{2}}\right)\left\Vert \hm{f}\right\Vert _{G\mathcal{S}_{\varrho+\varpi}^{*}}.
		\end{align*}
		We can also compute that
		{\small\begin{align*}
				\frac{d}{dt}\left\langle \alpha_{m,\beta_{j}}^{1}(\cdot-\beta_{j}t-x_{0,j}),\hm{v}(t)\right\rangle  & =\left\langle \frac{d}{dt}\alpha_{m,\beta_{j}}^{1}(\cdot-\beta_{j}t-x_{0,j}),\hm{v}(t)\right\rangle +\left\langle \alpha_{m,\beta_{j}}^{1}(\cdot-\beta_{j}t-x_{0,j}),\frac{d}{dt}\hm{v}(t)\right\rangle \\
				& =\left\langle \frac{d}{dt}\alpha_{m,\beta_{j}}^{1}(\cdot-\beta_{j}t-x_{0,j}),\hm{v}(t)\right\rangle +\left\langle \alpha_{m,\beta_{j}}^{1}(\cdot-\beta_{j}t-x_{0,j}),JH_{0}(t)\hm{v}(t)\right\rangle \\
				& +\left\langle \alpha_{m,\beta_{j}}^{1}(\cdot-\beta_{j}t-x_{0,j}),\hm{f}(t)\right\rangle \\
				& =\sum_{j\neq k}\left\langle \alpha_{m,\beta_{j}}^{1}(\cdot-\beta_{j}t-x_{0,j}),\mathcal{V}_{\beta_{k}}(\cdot-\beta_{k}t-x_{0,k})\hm{v}(t)\right\rangle \\
				& -\frac{1}{\gamma_{j}}\left\langle \alpha_{m,\beta_{j}}^{0}(\cdot-\beta_{j}t-x_{0,j}),\hm{v}(t)\right\rangle \\
				& +\left\langle \alpha_{m,\beta_{j}}^{0}(\cdot-\beta_{j}t-x_{0,j}),\hm{f}(t)\right\rangle .
		\end{align*}}
		From the identical computations above, we get
		\begin{align*}
			\left|\left\langle \alpha_{m,\beta_{j}}^{1}(\cdot-\beta_{j}t-x_{0,j}),\hm{v}(t)\right\rangle \right|_{G_{\varrho}L_{t}^{\infty}\bigcap G_{\varrho}L_{t}^{1}} & \lesssim\frac{1}{\delta} e^{-(\frac{L}{2})}\frac{1}{\varrho+\varepsilon}\left(\frac{1}{\varrho+\varepsilon}+\frac{1}{(\varrho+\varepsilon)^{2}}\right)\left\Vert \hm{v}\right\Vert _{\mathcal{S}_{\varrho}}\\
			& +\frac{1}{\varrho+\varpi}\left(\frac{1}{\varrho+\varpi}+\frac{1}{(\varrho+\varpi)^{2}}\right)\left\Vert \hm{f}\right\Vert _{G\mathcal{S}_{\varrho+\varpi}^{*}}.
		\end{align*}
		Taking $L$ large enough which is ensured by the separation
		condition and make the starting time $T_{0}$ large enough, putting computations above together,  we conclude
		that
		\[
		\left\Vert \hm{v}\right\Vert _{GS_{\varrho}}\lesssim\sum_{j=1}^{N}\left|a_{j}^{-}(0)\right|+\left\Vert \hm{f}\right\Vert _{GS_{\varrho+\varpi}^{*}}.
		\]
		
		\noindent{STEP 2: {\it existence and uniqueness of the solution}.}  
		Fix a set of data  $\left(a_{j}^{-}(0)\right)_{j=1}^{N}$ and $\hm{f}(t)\in G\mathcal{S}_{\varrho+\varpi}^{*}$. Take $T\geq T_0$ where $T_0$ is a large time, and let $\hm{f}_T (t)=\chi (t-T)\hm{f}(t)$ where $\chi(t)$ is a smooth decreasing function such that $\chi=1$ for $t\leq-1$ and $\chi(t)=0$ for $t\geq0$. Then clearly, $\lim_{T\rightarrow\infty}\left\Vert \hm{f}_T- \hm{f}\right\Vert _{GS_{\varrho+\varpi}^{*}}=0$. Let $\hm{v}_T$ be the solution to	\[
		\frac{d}{dt}\hm{v}_T(t)=JH_{0}(t)\hm{v}_T(t)+\hm{f}_T(t)
		\]
		satisfying	$\left(a_{j}^{-}(0)\right)_{j=1}^{N}=\left( \left\langle \alpha_{\beta_{j}}^{-}(\cdot-x_{j}),\hm{v}_T(0)\right\rangle\right)_{j=1}^N$ and $\hm{v}_T(T)-\pi_{0,s}(T)\hm{v}_T(T)=0$.  This solution exists by using the standard linear theory. From the {\it a priori} estimates in STEP 1, one has $	\left\Vert \hm{v}_T\right\Vert _{GS_{\varrho}}<\infty$ and $ \big(\hm{v}_T\big)_T$ satisfies the Cauchy condition  in $GS_{\varrho}$ as $T\rightarrow\infty$.  Denote the limit of the sequence as $\hm{v}\in GS_{\varrho}$ and then we  can conclude that it solves \eqref{eq:claimeq} in the sense of distribution. 
		The uniqueness again follows from the {\it a priori} estimates.

		After STEP 1 and STEP 2, 		the claim is proved.

		This claim allows us to get the {\it a priori} estimates. Setting $\hm{f}=\mathcal{I}(Q)+\mathfrak{I}_{1}(Q^{2},v)+\mathfrak{I}_{2}(Q,v^{2})+\hm{F}(v)$,
		we get
		\begin{align*}
			\left\Vert \hm{v}\right\Vert _{GS_{\varrho}} & \lesssim\sum_{j=1}^{N}\left|a_{j}^{-}(0)\right|+\left\Vert \hm{f}\right\Vert _{GS_{\varrho+\varpi}^{*}}\lesssim \frac{1}{\delta} e^{-(\frac{L}{2})}+\sum_{j=1}^{N}\left|a_{j}^{-}(0)\right|\\
			+ & \frac{1}{\delta} e^{-(\frac{L}{2})}\left\Vert \hm{v}\right\Vert _{GS_{\varrho}}+\left\Vert \hm{v}\right\Vert _{GS_{\varrho}}^{2}+\left\Vert \hm{v}\right\Vert _{GS_{\varrho}}^{3}.
		\end{align*}
		If $\left\Vert \hm{v}\right\Vert _{GS_{\varrho}}$ is small enough,
		then we get
		\[
		\left\Vert \hm{v}\right\Vert _{GS_{\varrho}}\lesssim \frac{1}{\delta} e^{-(\frac{L}{2})}+\sum_{j=1}^{N}\left|a_{j}^{-}(0)\right|.
		\]
		Note that $\frac{1}{\delta} e^{-(\frac{L}{2})}$ is independent of $\hm{v}$. This only depends on the prescribed linear trajectories.
		
		The fixed point argument here is the same as in Chen-Jendrej \cite{CJ3}. 
		Taking two solutions $\hm{v}_i$,  $i=1,2$ corresponding to $\left(a_{i,j}^{-}(0)\right)_{j=1}^{N}$.
		By the same argument as above, we have
		\begin{align*}
			\left\Vert \hm{v}_1-\hm{v}_2\right\Vert _{GS_{\varrho}} &\lesssim \sum_{j=1}^{N}\left|a_{1,j}^{-}(0)-a_{2,j}^{-}(0)\right|+  \frac{1}{\delta} e^{-(\frac{L}{2})}\left\Vert \hm{v}_1-\hm{v}_2\right\Vert _{GS_{\varrho}}\\
			&+\big(\left\Vert \hm{v}_1-\hm{v}_2\right\Vert _{GS_{\varrho}}\big)\big(\left\Vert \hm{v}_1\right\Vert _{GS_{\varrho}}+\left\Vert \hm{v}_2\right\Vert _{GS_{\varrho}}\big)+\big(\left\Vert \hm{v}_1-\hm{v}_2\right\Vert _{GS_{\varrho}}\big)\big(\left\Vert \hm{v}_1\right\Vert _{GS_{\varrho}}+\left\Vert \hm{v}_2\right\Vert _{GS_{\varrho}}\big)^2.
		\end{align*}
		Therefore, if $\left\Vert \hm{v}_i\right\Vert _{GS_{\varrho}}$ are small, we obtain a contraction map. 		
		So we get conclude that for any given $\left(a_{j}^{-}(0)\right)_{j=1}^{N}$, there is
		a unique solution $\hm{v}$ in $G\mathcal{S}_{\varrho}$. So all solutions
		satisfying \eqref{eq:asyMuti} form a $N$ dimension manifold and they are parameterized by  $\left(a_{j}^{-}(0)\right)_{j=1}^{N}.$

	\end{proof}
	\begin{rem}\label{rem:cubic}
		For the cubic nonlinearity, due to the Sobolev embedding from $H^1(\mathbb{R}^3)$ into $L^6(\mathbb{R}^3)$, in the first approach above, one can actually use energy estimates instead of Strichartz estimates.  We sketch a proof here.
		
		For $\varepsilon>$ small, we set 
		\[
		\chi_{j}(t,x):=\chi\left(\varepsilon(x-\beta_{j}t-x_{0,j})\right)
		\]
		and define
		\[
		Q(t;\hm{u}_{0},\hm{u}_{0}):=\frac{1}{2}\int_{\mathbb{R}^{3}}\left((\dot{u}_{0})^{2}+2\sum_{j=1}^{N}\chi_{j}(t)\dot{u}_{0}\left(\beta_{j}(t)\cdot\nabla u_{0}\right)+|\nabla u_{0}|^{2}+(1+V(t))u_{0}^{2}\right)\,dx.
		\]
		Then we have the following coercivity estimate:
		\begin{align*}
			Q(t;\hm{u}_{0},\hm{u}_{0}) & \geq c\left\Vert \hm{u}_{0}\right\Vert _{\mathcal{H}}-\frac{1}{c}\sum_{j=1}^{N}\left(\left\langle \alpha_{\beta_{j}}^{+}(\cdot-\beta_{j}t-x_{0,j},\hm{u}_{0}\right\rangle ^{2}+\left\langle \alpha_{\beta_{j}(t)}^{-}(\cdot-\beta_{j}t-x_{0,j}),\hm{u}_{0}\right\rangle ^{2}\right)\\
			& -\frac{1}{c}\sum_{j=1}^{N}\sum_{m=1}^{3}\left(\left\langle \alpha_{m,\beta_{j}(t)}^{0}(\cdot-\beta_{j}t-x_{0,j},\hm{u}_{0}\right\rangle ^{2}+\left\langle \alpha_{m,\beta_{j}(t)}^{1}(\cdot-\beta_{j}t-x_{0,j},\hm{u}_{0}\right\rangle ^{2}\right).
		\end{align*}
		See Chen-Jendrej \cite{CJ} for the proof in more general settings.
		
		Then for the solution to \eqref{eq:claimeq}, one can show that 
		\begin{equation}
			\big| \frac{d}{dt} Q(t:\hm{v}(t),\hm{v}(t)) \big|\lesssim c_0 \left\Vert \hm{v}(t)\right\Vert _{\mathcal{H}}^2+ \left\Vert \hm{v}(t)\right\Vert _{\mathcal{H}}\left\Vert \hm{f}(t)\right\Vert _{\mathcal{H}}
		\end{equation}
		for $c_0$ small. Integrating from $\infty$, one has
		\begin{equation}
			\big| Q(t:\hm{v}(t),\hm{v}(t)) \big|\lesssim c_0 \frac{e^{-2\varrho t}}{2\varrho}\left\Vert \hm{v}(t)\right\Vert _{\mathcal{H}_\varrho}^2+\frac{e^{-2\varrho t}}{2\varrho} \left\Vert \hm{v}(t)\right\Vert _{\mathcal{H}_\varrho}\left\Vert \hm{f}(t)\right\Vert _{\mathcal{H}_\varrho}
		\end{equation}
		where $\mathcal{H}_\varphi$ is the corresponding weighted energy norm. 
		
		The computations and estimates for $a_j^{\pm}(t)$ are the same in the computations in the proof above. Then differentiating in time for projections onto the generalized kernel, the same results from the proof above hold.
		
		Using the coercivity, one can conclude that
		\[
		\left\Vert \hm{v}\right\Vert _{\mathcal{H}_{\varrho}}\lesssim\sum_{j=1}^{N}\left|a_{j}^{-}(0)\right|+\left\Vert \hm{f}\right\Vert _{\mathcal{H}_{\varrho+\varpi}}.
		\]
		Setting  $\hm{f}=\mathcal{I}(Q)+\mathfrak{I}_{1}(Q^{2},v)+\mathfrak{I}_{2}(Q,v^{2})+\hm{F}(v)$, we note that
		\begin{equation}
			\left\Vert \mathfrak{I}_{2}(Q,v^{2})+\hm{F}(v)\right\Vert _{\mathcal{H}}\lesssim \left\Vert v\right\Vert _{L^6}^2 + \left\Vert v\right\Vert _{L^6}^3\lesssim \left\Vert \hm{v}\right\Vert _{\mathcal{H}}^2 +\left\Vert \hm{v}\right\Vert _{\mathcal{H}}^3
		\end{equation} by Sobolev's embedding. Then the remaining steps are the same as those in the first approach above. 
		
		For full details in the 1-d setting, see Chen-Jendrej \cite{CJ3}. In particular, following the same analysis in \cite{CJ3}, using spacially localized energy norms, one can also  conclude that $\hm{v}(t)$ is exponentially localized.

		Strichartz estimates will allow us to work for other  mass supercritical and energy subcritical nonlinearities provided that the spectral conditions hold.
	\end{rem}

	\bigskip

\begin{thebibliography}{RodSch}
		\bibitem{Agm} Agmon S. Spectral properties of Schr\"odinger operators and scattering theory.\emph{ Ann. Scuola Norm. Sup. Pisa Cl. Sci.}
		(4) 2 (1975), no. 2, 151--21.
		
		\bibitem{BCD} Bahouri, H.; Chemin, J.-Y.; Danchin, R. Fourier analysis and nonlinear partial differential equations. Grundlehren der mathematischen Wissenschaften [Fundamental Principles of Mathematical Sciences], 343. \emph{Springer, Heidelberg}, 2011. xvi+523 pp.
		
		
		%		\bibitem{Bec1}Beceanu, M. New estimates for a time-dependent
		%		Schr\"odinger equation. Duke Math. J. 159 (2011), no. 3, 417\textendash 477. 
		
		\bibitem{Bec2}Beceanu, M. A critical centre-stable manifold
		for Schr\"odinger's equation in three dimensions. Comm. Pure Appl.
		Math. 65 (2012), no. 4, 431\textendash 507. 
		
		\bibitem{BC} Beceanu, M.; Chen, G. Strichartz estimates for the
		Klein--Gordon equation in $\mathbb{R}^{3+1}$. To appear in Pure and Applied Analysis,  arXiv:2108.06390.
		
		
		\bibitem{BGL} Bellazzini, J.; Ghimenti, M; Le Coz, S. Multi-solitary waves for the nonlinear Klein-Gordon equation.\emph{Comm. Partial Differential Equations} 39 (2014), no. 8, 1479--1522.
		\bibitem{BL} Bergh, J.; L\"ofstr\"om, J.: Interpolation spaces.
		\emph{Springer, Berlin}, 1976.
		
		\bibitem{C1} Chen, G. Strichartz estimates for wave equations
		with charge transfer Hamiltonians, arXiv:1610.05226 (2016).\emph{
			Mem. Amer. Math. Soc}. 273 (2021), no. 1339, v + 84 pp.
		
		\bibitem{C2} Chen, G.  Wave equations with moving potentials. \emph{Comm. Math. Phys.} 375 (2020), no. 2, 1503--1560. 
		\bibitem{C3} 
		Chen, G. Multisolitons for the defocusing energy critical wave equation with potentials. \emph{Comm. Math. Phys.} 364 (2018), no. 1, 45--82.
		
		\bibitem{CJ}Chen, G.; Jendrej, J. Lyapunov-type characterisation
		of exponential dichotomies with applications to the heat and Klein-Gordon
		equations. \emph{Trans. Amer. Math. Soc}. 372 (2019), no. 10, 7461\textendash 7496.
		
		%		\bibitem{CJ1}Chen, G.; Jendrej, J. An alternative approach to  exponential dichotomies  for Klein-Gordon equations. Notes. 
		
		\bibitem{CJ2} Chen, G.; Jendrej, J. Strichartz estimates for Klein-Gordon equations with moving potentials.  arXiv:2210.03462. (2022).
		
		\bibitem{CJ3}
		Chen, G.; Jendrej, J. Kink networks for scalar fields in dimension 1+1. \emph{Nonlinear Anal.} 215 (2022), Paper No. 112643, 23 pp.
		
		
		%\bibitem{CJ4} Chen, G.; Jendrej, J. Invariant  Centre-stable manifolds around muti-soliton family to Klein-Gordon equations. In preparation. 
		\bibitem{Com} Combet, V. Multi--soliton solutions for the supercritical gKdV equations. \emph{Comm. Partial Differential Equations} 36 (2011), no. 3, 380--419.	
		
		%		\bibitem{CP} Chen, G.; Pusateri, F. Asymptotic stability of moving kinks of the double sine-Gordon equation.	In preparation.
		
		\bibitem{CMart}C\^ote, R.; Martel, Y. Multi-travelling
		waves for the nonlinear Klein-Gordon equation. \emph{Trans. Amer.
			Math. Soc.} 370 (2018), no. 10, 7461\textendash 7487. 
		
		\bibitem{CMu}C\^ote, R.; Mu\~noz, C. Multi-solitons for
		nonlinear Klein-Gordon equations. \emph{Forum Math. Sigma} 2 (2014),
		Paper No. e15, 38 pp.
		
		\bibitem{CoteF}C\^ote R.; Friederich, X. On smoothness and uniqueness of multi-solitons of the non-linear Schr\"odinger equations.\emph{ Comm. Partial Differential Equations} 46 (2021), no. 12, 2325--2385.
		
		%\bibitem{CMYZ}	 C\^ote, R.; Martel, Y.; Yuan, X.; Zhao, L. Description and classification of 2-solitary waves for nonlinear damped Klein--Gordon equations. \emph{Comm. Math. Phys.} 388 (2021), no. 3, 1557--1601.	
		\bibitem{Cu1} Cuccagna, S. Stabilization of solutions to nonlinear
		Schr\"odinger equations,\emph{ Comm. Pure Appl. Math.} 54 (2001), no.
		9, 1110\textendash 1145.	
		
		\bibitem{DaF} D'Ancona, P.; Fanelli, L. Strichartz and smoothing
		estimates of dispersive equations with magnetic potentials, \emph{Comm.
			Partial Differential Equations} 33 (2008), no. 4-6, 1082\textendash 1112.
		\bibitem {DS} Demanet, L.; Schlag, W. Numerical verification of a gap condition for a linearized nonlinear Schr\"odinger equation. \emph{Nonlinearity} 19 (2006), no. 4, 829--852.
		%	\bibitem{ESch} 
		%	Erdogan, M. B.; Schlag, W. Dispersive estimates for Schr\"odinger operators in the presence of a resonance and/or an eigenvalue at zero energy in dimension three. II.\emph{ J. Anal. Math.} 99 (2006), 199--248.
		\bibitem{Fri}  Friederich , X. On existence and uniqueness of asymptotic N-soliton-like solutions of the nonlinear klein-gordon equation. preprint 2021. arXiv:2106.01106. 
		
		\bibitem{IMN} Ibrahim, S.; Masmoudi, N.; Nakanishi, K. Scattering
		threshold for the focusing nonlinear Klein-Gordon equation. \emph{Anal.
			PDE} 4 (2011), no. 3, 405--460.
		
		%	\bibitem{KT} Keel, M.; Tao, T. Endpoint Strichartz estimates.\emph{
		%		Amer. J. Math}. 120 (1998), no. 5, 955\textendash 980.
		
		%	\bibitem{KoKo} Komech, A. I.; Kopylova, E. A. Weighted energy decay for 3D Klein-Gordon equation. \emph{J. Differential Equations} 248 (2010), no. 3, 501--520.
		%	\bibitem{KoKo1}	Komech, A. I.; Kopylova, E. A. On asymptotic stability of moving kink for relativistic Ginzburg-Landau equation.\emph{ Comm. Math. Phys.} 302 (2011), no. 1, 225--252.
		%	\bibitem{Kop}
		%	Kopylova, E. On long-time decay for modified Klein-Gordon equation.\emph{ Commun. Math. Anal.} 2011, Conference 3, 137--152. 
		
		\bibitem{LLT}Le Coz, S.; Li, D.; Tsai, T.-P. Fast-moving finite and infinite trains of solitons for nonlinear Schr\"odinger equations. \emph{Proc. Roy. Soc. Edinburgh Sect. A} 145 (2015), no. 6, 1251--1282.
		\bibitem{KS} Krieger J.; Schlag W. On the focusing critical
		semi-linear wave equation. Amer. J. Math. 129 (2007), no. 3, 843--913.
		
		\bibitem{KS2} Krieger, J.;  Schlag, W. Stable manifolds for
		all monic supercritical focusing nonlinear Schr\"odinger equations in
		one dimension. \emph{J. Amer. Math. Soc}. 19 (2006), no. 4, 815\textendash 920
		
		
		%	\bibitem{MNNO} Machihara, S.; Nakamura, M.; Nakanishi, K.;
		%	Ozawa, T. Endpoint Strichartz estimates and global solutions for the
		%	nonlinear Dirac equation, \emph{J. Funct. Anal.} 219 (2005), no. 1,
		%	1--20.
		
		
		\bibitem{Mart}  Martel, Y. Asymptotic N--solito--like solutions of the subcritical and critical generalized Korteweg--de Vries equations. \emph{Amer. J. Math.} 127 (2005), no. 5, 1103--1140.
		
		
		\bibitem{MM1}Martel, Y.; Merle, F. Asymptotic stability of
		solitons of the subcritical gKdV equations revisited. \emph{Nonlinearity}
		18 (2005), no. 1, 55\textendash 80.
		%	\bibitem{MST}
		%	Metcalfe, J.; Sterbenz, J.; Tataru, D. Local energy decay for scalar fields on time dependent non-trapping backgrounds. \emph{ Amer. J. Math.} 142 (2020), no. 3, 821--883.
		
		\bibitem{MMT} Martel, Y,; Merle, F.; Tsai, T.-P. Stability
		and asymptotic stability in the energy space of the sum of $N$ solitons
		for subcritical gKdV equations. \emph{Comm. Math. Phys.} 231 (2002),
		no. 2, 347\textendash 373.
		%	\bibitem{MMT2}
		%	Martel, Y,; Merle, F.; Tsai, T.-P. Stability in $H^1$ of the sum of K solitary waves for
		%	some nonlinear Schr\"odinger equations, \emph{Duke Math. J.} 133 (2006), no. 3, 405--466,		
		
		%	\bibitem{MS} Muscalu, C. and Schlag, W. \emph{Classical and multilinear
		%		harmonic analysis. Vol. I. Cambridge Studies in Advanced Mathematics},
		%	138. Cambridge University Press\emph{, Cambridge,} 2013. xvi+324 pp.
		
		\bibitem{NSch} Nakanishi, K.; Schlag, W. Global dynamics above
		the ground state for the nonlinear Klein-Gordon equation without a
		radial assumption. \emph{Arch. Ration. Mech. Anal}. 203 (2012), no.
		3, 809\textendash 851. 
		
		\bibitem{NSch1}Nakanishi, K.; Schlag, W. Invariant manifolds
		and dispersive Hamiltonian evolution equations. Zurich Lectures in
		Advanced Mathematics. \emph{European Mathematical Society (EMS), Z\"urich},
		2011. vi+253 pp. ISBN: 978-3-03719-095-1
		
		\bibitem{NSch2} Nakanishi, K.; Schlag, W. Global dynamics
		above the ground state energy for the cubic NLS equation in 3D. \emph{Calc.
			Var. Partial Differential Equations} 44 (2012), no. 1-2, 1\textendash 45.
		
		\bibitem{NSch3} Nakanishi, K.; Schlag, W. Global dynamics
		above the ground state energy for the focusing nonlinear Klein-Gordon
		equation.\emph{ J. Differential Equations} 250 (2011), no. 5, 2299\textendash 2333.
		
		\bibitem{PW} Pego, R.; Weinstein, M. I. Asymptotic stability
		of solitary waves. \emph{Comm. Math. Phys}. 164 (1994), no. 2, 305\textendash 349.
		\bibitem{Perl} Perelman, G. Asymptotic stability of multi-soliton solutions for nonlinear Schr\"odinger equations. \emph{Comm. Partial Differential Equations} 29 (2004), no. 7--8, 1051--1095.
		
		\bibitem{RSS} Rodnianski, I; Schlag, W.; Soffer, A. Dispersive
		analysis of charge transfer models. \emph{Comm. Pure Appl. Math}.
		58 (2005), no. 2, 149\textendash 216.
		
		
		\bibitem{RSS2} Rodnianski, I; Schlag, W.; Soffer, A. Asymptotic
		stability of N-soliton states of NLS. preprint (2003), arXiv preprint
		math/0309114.
		
		%		\bibitem{RodSch} Rodnianski, I.; Schlag, W. Time decay for
		%		solutions of Schr\"odinger equations with rough and time-dependent
		%		potentials. \emph{Invent. Math.} 155 (2004), no. 3, 451\textendash 513.
		
		%	\bibitem{Sch} Schlag, W. Dispersive estimates for Schr\"odinger
		%	operators: a survey. Mathematical aspects of nonlinear dispersive
		%	equations, 255\textendash 285, Ann. of Math. Stud., 163, \emph{Princeton
		%		Univ. Press, Princeton}, NJ, 2007. 
		
		\bibitem{Sch1} Schlag, W. Stable manifolds for an orbitally unstable nonlinear Schr\"odinger equation. \emph{Ann. of Math.} (2) 169 (2009), no. 1, 139--227.
		
		
		
		\bibitem{SW2} Soffer, A.; Weinstein, M. I. Multichannel nonlinear
		scattering for nonintegrable equations. \emph{Comm. Math. Phys.} 133
		(1990), no. 1, 119\textendash 146.
		
		\bibitem{SW3} Soffer, A.; Weinstein, M. I. Multichannel nonlinear
		scattering for nonintegrable equations. II. The case of anisotropic
		potentials and data.\emph{ J. Differential Equations} 98 (1992), no.
		2, 376\textendash 390. 
		
		\bibitem{Tao} Tao, T. \emph{Nonlinear dispersive equations.
			Local and global analysis. CBMS Regional Conference Series in Mathematics},
		106. Published for the Conference Board of the Mathematical Sciences,
		Washington, DC; by the American Mathematical Society, Providence,
		RI, 2006. xvi+373 pp. 
		
		
		\bibitem{Tao2} Tao, T. Why are solitons stable? \emph{Bull. Amer. Math. Soc.} (N.S.) 46 (2009), no. 1, 1--33.
		
		
		
		\bibitem{Wein} Weinstein, M. I. Modulational stability of ground
		states of nonlinear Schr\"odinger equations. \emph{SIAM J. Math. Anal.}
		16 (1985), no. 3, 472\textendash 491. 
		
		\bibitem{Wein2} Weinstein, M. I. Lyapunov stability of ground
		states of nonlinear dispersive evolution equations, \emph{Comm. Pure
			Appl. Math.} 39 (1986), no. 1, 51\textendash 67.
	\end{thebibliography}
\end{document}